\documentclass[11pt]{amsart}
\usepackage{graphicx}
\usepackage{amssymb}
\usepackage{amsmath}
\usepackage{amsthm}
\usepackage{mathtools}
\usepackage{tikz}
\usepackage{psfrag}
\usepackage{xcolor}
\usepackage{verbatim}
\usepackage{hyperref}
\usepackage{enumerate}
\allowdisplaybreaks[4]

\newtheorem{Theorem}{Theorem}[section]

\newtheorem{Lemma}[Theorem]{Lemma}
\newtheorem{Proposition}{Proposition}[section]
\newtheorem{Remark}{Remark}[section]

\newcommand{\pt}{\partial_{t}}
\newcommand{\px}{\partial_{x}}
\newcommand{\py}{\partial_{y}}
\newcommand{\ka}{\kappa}
\newcommand{\ty}{\infty}
\newcommand{\De}{\Delta_k^{\mathrm{h}}}
\newcommand{\s}{S_{k-1}^{\mathrm{h}}}
\newcommand{\f}{\mathcal{F}}
\newcommand{\va}{\varphi}

\newcommand{\R}{\mathbb{R}}
\newcommand{\h}{\mathrm{h}}
\newcommand{\vm}{\mathrm{v}}
\newcommand{\T}{T^{\mathrm{h}}}
\newcommand{\Gt}{\tilde{G}}
\newcommand{\psit}{\tilde{\psi}}
\newcommand{\ltr}{\langle t\rangle}
\newcommand{\lsr}{\langle s\rangle}
\newcommand{\de}{\delta}
\newcommand{\ta}{\theta}
\newcommand{\tad}{\dot{\theta}}
\newcommand{\la}{\lambda}
\newcommand{\ga}{\gamma}
\newcommand{\al}{\alpha}
\newcommand{\si}{\sigma}

\newcommand{\bp}{b_{\Phi}}
\newcommand{\uph}{u_{\Phi}}
\newcommand{\vp}{v_{\Phi}}
\newcommand{\hp}{h_{\Phi}}
\newcommand{\Pp}{P_{\Phi}}
\newcommand{\Np}{N_{\Phi}}
\newcommand{\Gp}{G_{\Phi}}
\newcommand{\Hp}{H_{\Phi}}
\newcommand{\Htp}{\tilde{H}_{\Phi}}
\newcommand{\Gtp}{\tilde{G}_{\Phi}}
\newcommand{\Ll}{\mathcal{L}}
\newcommand{\Lk}{\mathcal{L_{\ka}}}
\newcommand{\U}{\mathcal{U}}
\newcommand{\ze}{\zeta}
\newcommand{\zet}{\tilde{\zeta}}
\newcommand{\ep}{\epsilon}

\makeatletter
\@namedef{subjclassname@2020}{%
  \textup{2020} Mathematics Subject Classification}
\numberwithin{equation}{section}
\makeatother
\numberwithin{equation}{section} \allowdisplaybreaks
\usepackage[left=1 in, right=1 in,top=1 in, bottom=1 in]{geometry}
\begin{document}
\title[Global small solutions of MHD boundary layer equations ]{\bf Global small solutions of MHD boundary layer equations in Gevrey function space}
\author{Zhong Tan}
 \address{School of Mathematical Sciences\\ Xiamen University\\  Xiamen, Fujian Province 361005, China
 \& Shenzhen Research Institute of Xiamen University\\ Shenzhen 518057, China}
      \email{tan85@xmu.edu.cn}
    
 \author{Zhonger Wu}
          \address{School of Mathematical Sciences\\ Xiamen University\\  Xiamen, Fujian Province 361005, China}
    \email{wze622520@163.com}

\begin{abstract}
In this paper, we obtain global small solutions and decay estimates for the MHD boundary layer in Gevrey space without any structural assumptions, generalizing the results of \cite{NL} in analytic space. The proof method is mainly inspired by \cite{WXLY} and \cite{CW}, using new auxiliary functions and finer structural analysis to overcome the difficulty of the loss of derivatives and then we obtain  the global well-posedness of the MHD boundary layer in the Gevrey $\frac{3}{2}$ space.
\end{abstract}
 
\bigbreak

\maketitle
\section{Introduction}
In this paper, we study the well-posedness of the two dimensional MHD boundary layer in Gevrey space in the upper half-space $\R^2_+=\{(x,y)|x\in\R,y\in\R_+\}$.The equations take the following form (see \cite{DG,XYL,CJL}):
\begin{align}\label{mhd1}
\left\{\begin{array}{l}
\left(\partial_{t}+u \partial_{x}+v \partial_{y}- \partial_{y}^{2}\right) u-\left(b \partial_{x}+h \partial_{y}\right) b+\px p=0 \\
\partial_{t} b+\partial_{y}(v b-u h)-\kappa \partial_{y}^{2} b=0 \\
\partial_{t} h-\partial_{x}(v b-u h)-\kappa \partial_{y}^{2} h=0
\end{array}\right.
\end{align}
with the divergence-free and homogeneous Dirichlet boundary conditions,
\begin{align}\label{cb}
\left\{\begin{array}{l}
\px u+\py v=\px b+\py h=0,
 \\
 (u,v,b,h)|_{y=0}=(0,0,0,0),~(u,b)|_{y\rightarrow +\ty}=(\overline{u},\overline{b}),
\\
(u,b)|_{t=0}=(u_0,b_0).
\end{array}\right.
\end{align}
where $(u,v)$ represent the velocity of fluid and $(b,h)$ represent the magnetic field. And $\ka>0$ is a constant that represents the ratio between Reynolds number and magnetic Reynolds number. Furthermore $(\overline{u},\overline{b},p)$ represent the traces of the tangential field and the pressure of the outflow on the boundary, satisfying Bernoulli's law:
\begin{align}\label{Bernoulli}
\left\{\begin{array}{l}
\pt \overline{u}+\overline{u}\px\overline{u}+\px p
=\overline{b}\px\overline{b},
 \\
\pt \overline{b}+\overline{u}\px\overline{b}
=\overline{b}\px\overline{u}.
\end{array}\right.
\end{align}

When the magnetic field $(b,h)=(0,0)$, the equation \eqref{mhd1} degenerates to the classical Prandtl equation. The Prandtl equation was originally proposed by Prandtl in 1904 to explain the difference between an ideal fluid and a low viscosity fluid near the boundary. The Prandtl equation has been studied for a long time, and a lot of results have been obtained. Under the assumption of monotonicity, Oleinik\cite{OAO} proved the uniqueness of the local existence of classical solution using the Crocco transform. And in 2015, \cite{RA} and \cite{NM} independently proved similar results in weighted Sobolev space by energy methods. Under the favourable pressure condition, Zhouping Xin and Liqun Zhang \cite{ZPX} obtained the global existence of the weak solution, which is due to the favourable pressure condition to avoid the back-flow.

When the velocity has no monotonicity assumption, the Prandtl equation cannot overcome the derivative loss in the space of finite regularity, so it needs to be studied in the space of smooth functions. When the data is analytic in both the $x$ variable and $y$ variable, Sammartino and Caflisch\cite{MS} obtained for the first time the local well-posedness. Later Lombardo, Cannone and Sammartino\cite{MCL}
removed the analyticity requirement from the $y$ variable.
Further relaxation of the analyticity conditions, we can refer to \cite{DC,HD,DG2015,WXLMY,WXLY2} for the well-posedness results in Gevrey space. In particular, without any structural assumptions, \cite{HD} obtained the local existence and uniqueness of the solution of the two dimensional Prandtl equation in the Gevrey 2 space, and \cite{WXLMY} extended this result to the three dimensional case.The Gevrey index 2 is optimal in the sense of the ill-posedness result of \cite{DG2010}.
These results are all local. Regarding the large time results, Ping Zhang and Zhifei Zhang \cite{PZ} first obtained the long-time existence of the solution of the Pandtl equation with small analytic data, and then \cite{MI} obtained the almost global existence of the solution. Finally, the global existence and decay estimates of the solution are obtained in \cite{MP}. Recently, Ping Zhang et al. \cite{CW} extended the results in \cite{MP} to the Gevrey 2 space.

Back to the MHD boundary layer. Under the structural assumption of the dominance of the tangential magnetic field, see the well-posedness results of \cite{CJLJFA,CJLCPAM,WXLERA} in Sobolev space. The comparison of these results with those of the classical Prandtl equation under the assumption of monotonicity shows that the magnetic field has a stabilizing effect on the MHD boundary layer regardless of the presence of resistance.
On the other hand without imposing structural assumptions, Feng Xie
and Tong Yang \cite{FX} obtained a lower bound on the existence time of the analytic solution of the equation \eqref{mhd1}. This result was subsequently improved by Ping Zhang et al.\cite{NL} to obtain the global existence of the analytic solution. In addition, Feng Xie et al. \cite{SXL} further improved the result of \cite{NL} by removing the smallness requirement of the outflow at the initial moment.
In the Gevrey space, Weixi Li and Tong Yang \cite{WXLY} obtained the local well-posedness of \eqref{mhd1} when the Gevrey indicator is $\frac{3}{2}$ without imposing any structural assumptions. It is worth noting that, combined with the ill-posedness results in \cite{CJL}, we can guess that the optimal Gevrey indicator for the MHD boundary layer \eqref{mhd1} should be one of the numbers in the interval $[\frac{3}{2},2]$.

Inspired by \cite{WXLY} and \cite{CW}, we investigate the global well-posedness of \eqref{mhd1}-\eqref{cb} with small initial data, which is Gevrey $\frac{3}{2}$ regular in the $x$ variable and Sobolev regular in the $y$ variable. For the sake of brevity of the proof process, we shall assume without loss of generality
\begin{align}\label{cb2}
(\overline{u},\overline{b},p)=(0,0,C),~\ka\in(0,2).
\end{align}
In fact, for general $(\overline{u},\overline{b})$ with smallness and $\ka\in(\frac{1}{2},\ty)$, see the treatment in \cite{NL}.

The remainder of this paper is organized as follows: in Section \ref{mainresult}, we will give our main theorem, and some frequently used lemmas. In Section \ref{zhengmingdingli}, we list some main a priori estimates and give the proof of theorem \ref{th}. In section \ref{gujiU}-\ref{gujipy3G}, we will prove all the a priori estimates listed in section \ref{zhengmingdingli}.

\section{Main Result and Preliminary}\label{mainresult}
In this section we will first give an introduction on the basic notation in the paper, then give our main theorem, and finally list some useful lemmas.

\subsection{Notation}
First we give an introduction to the basic notation in the paper.

First $a\lesssim b$ means there exists a unform positive constant $C$ (the constant may be different between each line) such that $a\leq Cb$ holds. Let $\ltr\stackrel{\text { def }}{=}1+t$, $[\xi]\stackrel{\text { def }}{=}(1+\xi^2)^{\frac{1}{2}}$. And $s+$ means a number a little larger than $s$.

Regarding the function space, first $L_+^p=L^p(\R_+^2)$, where $\R_+^2=\R\times\R_+. $ When $p=2$,
$$(a,b)_{L_+^2}=\int_{\R_+^2}a(x,y)b(x,y)dxdy,$$
denotes the standard $L^2$ inner product. If $X$ is a Banach space and $I$ is a subinterval of $\R$, then $L^q(I,X)$ means that the map $t\mapsto \|f(t)\|_{X}$ belongs to $L^q(I)$. In particular, we use $L^p_T(L^q_{\mathrm{h}}(L^r_{\vm}))$ to denote $L^p([0,T];L^q(\R_x;L^r(\R_y^+)))$.

The anisotropic Sobolev space is defined as follows, let $s\in\R,k\in\mathbb{N}$, then
$$\|f\|_{H_{\mathrm{h}}^s}\stackrel{\text { def }}{=}\left(\int_{\R}(1+|\xi|^2)^s
|\hat{f}(\xi)|^2d\xi\right)^{\frac{1}{2}},$$
$$\|f\|_{H^{s,k}}\stackrel{\text { def }}{=}\sum\limits_{0\leq l\leq k}\left(\int_{0}^{\ty}\int_{\R}(1+|\xi|^2)^s
|\py^l\hat{f}(\xi)|^2d\xi dy\right)^{\frac{1}{2}}.$$
In particular, the weighted space is defined as follows:
$$\|f\|_{H^{s,k}_{\Psi}}\stackrel{\text { def }}{=}\sum\limits_{0\leq l\leq k}\left(\int_{0}^{\ty}\int_{\R}e^{2\Psi(t,y)}(1+|\xi|^2)^s
|\py^l\hat{f}(\xi)|^2d\xi dy\right)^{\frac{1}{2}},$$
here $\Psi(t,y)=\frac{y^2}{8\ltr}.$
The inner product of the inner product space $H_{\Psi}^{\si,0}$ is defined as follows:
$$(f,g)_{H_{\Psi}^{\si,0}}\stackrel{\text { def }}{=}\mathrm{Re}\int_0^{+\ty}\int_{\R}e^{2\Psi}
[D_x]^{\si}f\overline{[D_x]^{\si}g}dxdy.$$
In addition, to keep the notation simple, for the Banach space $X$, we specify
$$\|\py^k(f,g)\|_X^{m}\stackrel{\text { def }}{=}
\|\py^k f\|_X^{m}+\|\py^k g\|_X^{m}.$$
\subsection{Main Result}
Using \eqref{cb}, we can rewrite \eqref{mhd1} in the following form
\begin{align}\label{mhd}
\left\{\begin{array}{l}
\left(\partial_{t}+u \partial_{x}+v \partial_{y}- \partial_{y}^{2}\right) u=\tad P, \\
\left(\partial_{t}+u \partial_{x}+v \partial_{y}- \ka \partial_{y}^{2}\right) b=\tad N, \\
\left(\partial_{t}+u \partial_{x}+v \partial_{y}- \ka \partial_{y}^{2}\right) h= b\px v-h\px u,
\end{array}\right.
\end{align}
where $\ta(t)$ is determined by \eqref{theta},$P$ and $N$ are as follows:
\begin{align}\label{PN}
P=\frac{1}{\tad}\left(b \partial_{x}+h \partial_{y}\right) b,~ N=\frac{1}{\tad}\left(b \partial_{x}+h \partial_{y}\right) u.
\end{align}
\begin{Remark}
In contrast to \cite{WXLY}, we isolate $\tad(t)$ in the definition of $(P,N)$, which is crucial for this paper, and the role of $\tad$ can be seen in the remark $\ref{duoletad}$.
\end{Remark}

Since \eqref{cb}, we can know that there exist stream functions $(\psi,\psit)$ such that $(u,b)=\py(\psi,\psit)$ and $(v,h)=-\px(\psi,\psit)$ hold and have boundary values
$$(\psi,\psit)|_{y=0}=(\psi,\psit)|_{y\rightarrow+\ty}=(0,0).$$
The detailed derivation can be found in \cite{NL}.

Integrating the equation \eqref{mhd} with respect to the $y$ variable over $[y,+\ty)$, we can obtain the stream function satisfying the following equation
\begin{align*}
\left\{\begin{array}{l}
\partial_{t} \psi+u \partial_{x} \psi+2 \int_{y}^{\infty}\left(\partial_{y} u \partial_{x} \psi\right) d y'-\partial_{y}^{2} \psi+\int_{y}^{\infty}\tad P d y'=0, \\
\partial_{t} \psit+u \partial_{x} \psit+ \int_{y}^{\infty}\left(\partial_{y} u \partial_{x} \psit+\partial_{y} b \partial_{x} \psi\right) d y^{\prime}-\ka\partial_{y}^{2} \psit+\int_{y}^{\infty}\tad N d y^{\prime}=0, \\
\left.(\psi,\psit)\right|_{y=0}=
\left.(\psi,\psit)\right|_{y\rightarrow+\ty}=(0,0), \\
\left.(\psi,\psit)\right|_{t=0}=(\psi_0,\psit_0)
=(-\int_y^{\ty}u_0dz,-\int_y^{\ty}b_0dz).
\end{array}\right.
\end{align*}

In order to obtain faster decay estimates for the solution of the equation \eqref{mhd}, we need to control the Gevrey radius of the solution, so we need to introduce the following ``good function'' (see \cite{NL,MP,CW}):
\begin{align}\label{gf}
G=u+\frac{y}{2\ltr}\psi,~~\tilde{G}=b+\frac{y}{2\ka\ltr}\psit.
\end{align}
By direct calculation, we can obtain that $(G,\Gt)$ satisfies the following equation:
\begin{align}\label{gfe}
\left\{\begin{array}{l}
\partial_{t} G-\partial_{y}^{2} G+\langle t\rangle^{-1} G+u \partial_{x} G+v \partial_{y} G
-\frac{1}{2\langle t\rangle} v \partial_{y}(y \psi)\\
~~+\frac{y}{\langle t\rangle} \int_{y}^{\infty}\left(\partial_{y} u \partial_{x} \psi\right) d y^{\prime}-\tad P+\frac{y}{2\ltr}\int_y^{\ty}\tad P dy'=0, \\
\partial_{t} \Gt-\ka\partial_{y}^{2} \Gt+\langle t\rangle^{-1} \Gt+u \partial_{x} \Gt+v \partial_{y} \Gt
-\frac{1}{2\ka\langle t\rangle} v \partial_{y}(y \psit)\\
~~+\frac{y}{2\ka\langle t\rangle} \int_{y}^{\infty}\left(\partial_{y} u \partial_{x} \psit+\py b\px\psi\right) d y^{\prime}-\tad N+\frac{y}{2\ka\ltr}\int_y^{\ty}\tad N dy'=0, \\
\left.(G,\Gt)\right|_{y=0}= \lim _{y \rightarrow+\infty} (G,\Gt)=(0,0), \\
\left.(G,\Gt)\right|_{t=0}=(G_{0},\Gt_0) \stackrel{\text { def }}{=} (u_{0}+\frac{y}{2} \psi_{0},b_0+\frac{y}{2\ka}\psit_0) .
\end{array}\right.
\end{align}

Let $[\xi]\stackrel{\text { def }}{=}(1+\xi^2)^{\frac{1}{2}}$,$D_x\stackrel{\text { def }}{=}\frac{1}{i}\px$, then $[D_x]^s$ denotes the differential operator whose corresponding Fourier product is $(1+|\xi|^2)^{\frac{s}{2}}$. Let $f_{\Phi}\stackrel{\text { def }}{=}e^{\Phi(t,D_x)}f$, where
\begin{align}\label{phi}
\Phi(t,D_x)=\delta(t)[D_x]^{\frac{2}{3}},~\delta(t)=
\delta-\lambda\theta(t).
\end{align}
Our main theorem is stated as follows:
\begin{Theorem}\label{th}
Let $\ka\in(0,2)$ be a given constant, and $\eta$ is also a fixed constant satisfying $0<\eta<l_{\ka}$, where $l_{\ka}=\frac{\ka(2-\ka)}{4}\in(0,\frac{1}{4}].$ The initial values satisfy $u_0|_{y=0}=b_0|_{y=0}=0,\int_0^{\ty}u_0dy=\int_0^{\ty}b_0dy=0.$
Define $E(t)$ and $D(t)$ as follows:
\begin{align}\label{Et}
\begin{aligned}
E(t)\stackrel{\text { def }}{=}&\|\ltr^{l_{\ka}-\eta}(\uph,\bp)\|_{H_{\Psi}^{7,0}}
+\|\ltr^{l_{\ka}-\eta}\U\|_{H_{\Psi}^{7,0}}
+\|\ltr^{l_{\ka}-\eta}(\ze,\zet)\|_{H_{\Psi}^{\frac{22}{3},0}}\\
&+\|\ltr^{l_{\ka}-\eta}(\Pp,\Np)\|_{H_{\Psi}^{\frac{20}{3},0}}
+\|\ltr^{1+l_{\ka}-\eta}(\Gp,\Gtp)\|_{H_{\Psi}^{4,0}}\\
&+\|\ltr^{\frac{3}{2}+l_{\ka}-\eta}(\py\Gp,\py\Gtp)\|_{H_{\Psi}^{3,0}}
+\|\ltr^{2+l_{\ka}-\eta}(\py^2\Gp,\py^2\Gtp)\|_{H_{\Psi}^{3,0}}\\
&+\|\ltr^{\frac{5}{2}+l_{\ka}-\eta}(\py^3\Gp,\py^3\Gtp)\|_{H_{\Psi}^{2,0}},
\end{aligned}
\end{align}
and
\begin{align}\label{Dt}
\begin{aligned}
D(t)\stackrel{\text { def }}{=}&\|\ltr^{l_{\ka}-\eta}(\py\uph,\py\bp)\|_{H_{\Psi}^{7,0}}
+\|\ltr^{l_{\ka}-\eta}\py\U\|_{H_{\Psi}^{7,0}}
+\|\ltr^{l_{\ka}-\eta}(\py\ze,\py\zet)\|_{H_{\Psi}^{\frac{22}{3},0}}\\
&+\|\ltr^{l_{\ka}-\eta}(\py\Pp,\py\Np)\|_{H_{\Psi}^{\frac{20}{3},0}}
+\|\ltr^{1+l_{\ka}-\eta}(\py\Gp,\py\Gtp)\|_{H_{\Psi}^{4,0}}\\
&+\|\ltr^{\frac{3}{2}+l_{\ka}-\eta}(\py^2\Gp,\py^2\Gtp)\|_{H_{\Psi}^{3,0}}
+\|\ltr^{2+l_{\ka}-\eta}(\py^3\Gp,\py^3\Gtp)\|_{H_{\Psi}^{3,0}}\\
&+\|\ltr^{\frac{5}{2}+l_{\ka}-\eta}(\py^4\Gp,\py^4\Gtp)\|_{H_{\Psi}^{2,0}}.
\end{aligned}
\end{align}
where $\U$ is defined by \eqref{U}. If the initial data satisfy
\begin{align}\label{E0}
\begin{aligned}
\|(\uph(0),\bp(0))\|_{H_{\Psi}^{\frac{22}{3},0}}+E(0)\leq \ep^2.
\end{aligned}
\end{align}
Then there exist $\ep_0,\la_0>0$ such that when $0<\ep<\ep_0,\la>\la_0$, we have
\\
\textbf{1)}$\theta(t)$ in \eqref{phi} satisfies $\sup_{0\leq t<\ty}\theta(t)\leq\frac{\delta}{4\la}$.
\\
\textbf{2)}The system \eqref{mhd1}-\eqref{cb} has a unique global solution $(u,b)$, satisfying
\begin{align}\label{ED}
\begin{aligned}
E(t)+\frac{79}{100}\eta\int_0^t D(s)ds\leq C\ep^2
\end{aligned}
\end{align}
for any $t\in[0,\ty)$.

\end{Theorem}
\subsection{Littlewood-Paley theory}
For the convenience of the reader, we give some basics of Littlewood-Paley theory in this subsection. First we have (see \cite{HB}),
\begin{align}\label{sk}
S_k^{\mathrm{h}}f=\mathcal{F}^{-1}(\chi(2^{-k}|\xi|)\hat{f})~~and~~\De f=\left\{\begin{array}{l}
\f^{-1}(\va(2^{-k}|\xi|)\hat{f})~~if~~k\geq~0,\\
S_0^{\mathrm{h}}f~~~~~~~~~~~~~~~~~~~~~~~~if~~k=~-1,\\
0~~~~~~~~~~~~~~~~~~~~~~~~~~~~~if~~k\leq-2.
\end{array}\right.
\end{align}

In this paper, $\f f$ and $\hat{f}$ always denote the Fourier transform of the distribution $f$ with respect to $x$ variable. On the other hand, $\chi$ and $\va$ denote smooth functions with compact support, the support of which is shown below:
$$\text{Supp}~\va\subset\{\xi\in\mathbb{R}|\frac{3}{4}\leq|\xi|\leq\frac{8}{3}\}~~
and~~\forall\xi\neq0,~\sum_{k\in\mathbb{Z}}\va(2^{-k}\xi)=1,$$
$$\text{Supp}~\chi\subset\{\xi\in\mathbb{R}|~|\xi|\leq\frac{4}{3}\}~~
and~~\chi(\xi)+\sum_{k\geq0}\va(2^{-k}\xi)=1.$$

Next, in order to estimate the product of the two distributions, we need to use the Bony's decomposition of the horizontal variable $x$ frequently (see \cite{JMB}):
\begin{align}\label{bony}
fg=T^{\mathrm{h}}_fg+T^{\mathrm{h}}_gf+R^{\mathrm{h}}(f,g),
\end{align}
where $\T_fg=\sum\limits_k\s f\De g,~R^{\mathrm{h}}(f,g)=\sum\limits_{|k-k'|\leq1}\De f\Delta_{k'}^{\mathrm{h}} g.$
\subsection{Para-product estimates}
Next we always assume that $t<T^*$, where $T^*$ is defined as follows:
\begin{align}\label{T*}
T^*\stackrel{\text { def }}{=}\sup\left\{t>0,\ta(t)<\frac{\de}{2\la}\right\}.
\end{align}
By \eqref{phi}, we know that when $t<T^*$, we have $\de(t)\geq\frac{\de}{2}.$ And since $|([\xi])'|=|\dfrac{\xi}{(1+\xi^2)^{\frac{1}{2}}}|<1$, so
$$[\xi]=[\xi-\eta+\eta]\leq[\xi-\eta]+|\eta|\leq[\xi-\eta]+[\eta].$$
Then
$$[\xi]^{\frac{2}{3}}\leq|[\xi]-[\eta]|^{\frac{2}{3}}+[\eta]^{\frac{2}{3}}
\leq[\xi-\eta]^{\frac{2}{3}}+[\eta]^{\frac{2}{3}},$$
thus we can obtain the convex inequality for $\Phi(t,\xi)=\de(t)[\xi]^{\frac{2}{3}}$:
\begin{align}\label{tu}
\Phi(t,\xi)\leq\Phi(t,\xi-\eta)+\Phi(t,\eta),
\end{align}
holds for any $\xi,\eta\in\R$. Using the same proof as in \cite{HB} (for the case $\Phi(t,\xi)=0$), combined with convexity, we can obtain the following two lemmas (see also \cite{CW}):
\begin{Lemma}\label{gu1}
Let $s\in\R$, and $(\T_f)^*$ denotes the adjoint operator of $\T_f$. If $\si>\frac{1}{2},$ then
$$\left\|\left(T_{f}^{\mathrm{h}} g\right)_{\Phi}\right\|_{H_{\mathrm{h}}^{s}} \leq C\left\|f_{\Phi}\right\|_{H_{\mathrm{h}}^{\sigma}}
\left\|g_{\Phi}\right\|_{H_{\mathrm{h}}^{s}},$$
$$\left\|\left(\left(T_{f}^{\mathrm{h}}\right)^{*} g\right)_{\Phi}\right\|_{H_{\mathrm{h}}^{s}} \leq C\left\|f_{\Phi}\right\|_{H_{\mathrm{h}}^{\sigma}}
\left\|g_{\Phi}\right\|_{H_{\mathrm{h}}^{s}}.$$
If $s_1+s_2>s+\frac{1}{2}>0,$ we also have
$$\left\|\left(R^{\mathrm{h}}(f, g)\right)_{\Phi}\right\|_{H_{\mathrm{h}}^{s}} \leq C\left\|f_{\Phi}\right\|_{H_{\mathrm{h}}^{s_{1}}}
\left\|g_{\Phi}\right\|_{H_{\mathrm{h}}^{s_{2}}}.$$
\end{Lemma}
\begin{Lemma}\label{gu2}
Let $s>0,\si>\frac{3}{2}$, then we have
$$
\left\|\left(\left(T_{a}^{\mathrm{h}} T_{b}^{\mathrm{h}}-T_{a b}^{\mathrm{h}}\right) f\right)_{\Phi}\right\|_{H_{\mathrm{h}}^{s}} \leq C\left\|a_{\Phi}\right\|_{H_{\mathrm{h}}^{\sigma}}\left\|b_{\Phi}\right\|_{H_{\mathrm{h}}^{\sigma}}\left\|f_{\Phi}\right\|_{H_{\mathrm{h}}^{s-1}},
$$
$$
\left\|\left(\left[\left[D_{x}\right]^{s} ; T_{a}^{\mathrm{h}}\right] f\right)_{\Phi}\right\|_{L_{\mathrm{h}}^{2}} \leq C\left\|a_{\Phi}\right\|_{H_{\mathrm{h}}^{\sigma}}\left\|f_{\Phi}\right\|_{H_{\mathrm{h}}^{s-1}},
$$
$$
\left\|\left(\left(T_{a}^{\mathrm{h}}-\left(T_{a}^{\mathrm{h}}\right)^{*}\right) f\right)_{\Phi}\right\|_{H_{\mathrm{h}}^{s}} \leq C\left\|a_{\Phi}\right\|_{H_{\mathrm{h}}^{\sigma}}\left\|f_{\Phi}\right\|_{H_{\mathrm{h}}^{s-1}},
$$
$$
\left\|\left(\left[T_{a}^{\mathrm{h}} ; T_{b}^{\mathrm{h}}\right] f\right)_{\Phi}\right\|_{H_{\mathrm{h}}^{s}} \leq C\left\|a_{\Phi}\right\|_{H_{\mathrm{h}}^{\sigma}}\left\|b_{\Phi}\right\|_{H_{\mathrm{h}}^{\sigma}}\left\|f_{\Phi}\right\|_{H_{\mathrm{h}}^{s-1}},
$$
where $[a;b]$ denotes the commutator.
\end{Lemma}
Compared to the Prandtl equation, \eqref{mhd} adds the effect of the magnetic field. In order to deal with the difficulties caused by the Lorentz force, we need the following lemma:
\begin{Lemma}\label{gu3}
Let $\si>\frac{3}{2},s_1,s_2>0$, then
$$\left|\left([D_x]^{s_1}\T_a\px f,[D_x]^{s_2}g\right)_{L^2_{\h}}+\left([D_x]^{s_2}\T_a\px g,[D_x]^{s_1}f\right)_{L^2_{\h}}\right|\leq C\|a\|_{H^{\si}_{\h}}\|f\|_{H^{s_1}_{\h}}\|g\|_{H^{s_2}_{\h}}.$$
\end{Lemma}
\begin{proof}
\begin{align*}
&~~~\left([D_x]^{s_1}\T_a\px f,[D_x]^{s_2}g\right)_{L^2_{\h}}+\left([D_x]^{s_2}\T_a\px g,[D_x]^{s_1}f\right)_{L^2_{\h}}
\\
&=\left(\left[[D_x]^{s_1};\T_a\right]\px f,[D_x]^{s_2}g\right)_{L^2_{\h}}+\left(\T_a[D_x]^{s_1}\px f,[D_x]^{s_2}g\right)_{L^2_{\h}}+\left([D_x]^{s_2}\T_a\px g,[D_x]^{s_1}f\right)_{L^2_{\h}}\\
&=\left(\left[[D_x]^{s_1};\T_a\right]\px f,[D_x]^{s_2}g\right)_{L^2_{\h}}+\left((\T_a-(\T_a)^*)[D_x]^{s_1}\px f,[D_x]^{s_2}g\right)_{L^2_{\h}}\\
&~~~~+\left(\left[[D_x]^{s_2};\T_a\right]\px g,[D_x]^{s_1}f\right)_{L^2_{\h}}+\left((\T_a)^*[D_x]^{s_1}\px f,[D_x]^{s_2}g\right)_{L^2_{\h}}\\
&~~~~+\left(\T_a[D_x]^{s_2}\px g,[D_x]^{s_1}f\right)_{L^2_{\h}}\\
&=\left(\left[[D_x]^{s_1};\T_a\right]\px f,[D_x]^{s_2}g\right)_{L^2_{\h}}+\left((\T_a-(\T_a)^*)[D_x]^{s_1}\px f,[D_x]^{s_2}g\right)_{L^2_{\h}}\\
&~~~~+\left(\left[[D_x]^{s_2};\T_a\right]\px g,[D_x]^{s_1}f\right)_{L^2_{\h}}
+\left(\left[\T_a;\px\right][D_x]^{s_2} g,[D_x]^{s_1}f\right)_{L^2_{\h}}.
\end{align*}
Using the H$\mathrm{\ddot{o}}$lder's inequality and the lemma \ref{gu2}, we can complete the proof of the lemma \ref{gu3}.
\end{proof}
\begin{Remark}
In particular, if we take $f=g,s_1=s_2=s>0$ in lemma \ref{gu3}, then we have
$$\left|\left([D_x]^{s}\T_a\px f,[D_x]^{s}f\right)_{L^2_{\h}}\right|\leq C\|a\|_{H^{\si}_{\h}}\|f\|^2_{H^{s}_{\h}}.$$
This inequality is used in \cite{CW}.
\end{Remark}
Next, we give two estimates of the commutator of $e^{\Phi}$ and $\T_a$. The proof of $\Phi(t,\xi)=\de(t)[\xi]^{\frac{1}{2}}$ is given in \cite{CW}, and the result when $\Phi(t,\xi)=\de(t)[\xi]^{\frac{2}{3}}$ can be obtained with a slight modification of the proof.
\begin{Lemma}\label{gu4}
Let $\Phi(t,\xi)=\de(t)[\xi]^{\frac{2}{3}}$,$\si>\frac{3}{2},s\in\R$. Then we have
$$
\left\|\left(T_{a}^{\mathrm{h}} \partial_{x} f\right)_{\Phi}-T_{a}^{\mathrm{h}} \partial_{x} f_{\Phi}\right\|_{H_{\mathrm{h}}^{s}} \leq C \delta(t)\left\|a_{\Phi}\right\|_{H_{\mathrm{h}}^{\sigma}}\left\|f_{\Phi}\right\|_{H_{\mathrm{h}}^{s+\frac{2}{3}}}.
$$
\end{Lemma}
\begin{Lemma}\label{gu5}
Let $\Phi(t,\xi)=\de(t)[\xi]^{\frac{2}{3}},Q(\xi)=\xi(1+\xi^2)^{-\frac{2}{3}}$,
$\si>\frac{5}{2},s\in\R.$ If $0<\de(t)\leq L$, then the following inequality holds,
$$
\left\|\left(T_{a}^{\mathrm{h}} \partial_{x} f\right)_{\Phi}-T_{a}^{\mathrm{h}} \partial_{x} f_{\Phi}-\frac{2}{3}\delta(t) T_{D_{x} a}^{\mathrm{h}} Q(D_{x}) \partial_{x} f_{\Phi}\right\|_{H_{\mathrm{h}}^{s}} \leq C_{L}\left\|a_{\Phi}\right\|_{H_{\mathrm{h}}^{\sigma}}
\left\|f_{\Phi}\right\|_{H_{\mathrm{h}}^{s+\frac{1}{3}}} .
$$
\end{Lemma}
In order to obtain the decay estimates of the solution, we need to use the following two lemmas. The first is a Poincar$\mathrm{\acute{e}}$ type inequality, see \cite{MI,CW}.
\begin{Lemma}\label{gu6}
Let $\Psi=\frac{y^2}{8\ltr},d$ be a non-negative integer. The function $u$ is smooth enough on $\R^d\times\R_+$ and decays rapidly to $0$ as $y$ tends to $+\ty$. At this point we have the following two inequalities,
$$
\int_{\mathbb{R}^{d} \times \mathbb{R}_{+}}\left|\partial_{y} u(X, y)\right|^{2} e^{2 \Psi} d X d y \geq \frac{1}{2\langle t\rangle} \int_{\mathbb{R}^{d} \times \mathbb{R}_{+}}|u(X, y)|^{2} e^{2 \Psi} d X d y,
$$
$$
\begin{aligned}
\int_{\mathbb{R}^{d} \times \mathbb{R}_{+}}\left|\partial_{y} u(X, y)\right|^{2} e^{2 \Psi} d X d y & \geq \frac{s}{2\langle t\rangle} \int_{\mathbb{R}^{d} \times \mathbb{R}_{+}}|u(X, y)|^{2} e^{2 \Psi} d X d y \\
&+\frac{s(1-s)}{4} \int_{\mathbb{R}^{d} \times \mathbb{R}_{+}}\left|\frac{y}{\langle t\rangle} u(X, y)\right|^{2} e^{2 \Psi} d X d y,
\end{aligned}
$$
where $s>0.$
\end{Lemma}
The following lemma shows that the decay rate of the solution $(u,b)$ of the equation \eqref{mhd} and the stream function $(\psi,\psit)$ can be controlled by the decay rate of the ``good function'' $(G,\Gt)$.
\begin{Lemma}\label{gu7}
Let $(u,b)$ be the sufficiently smooth solution of the equation \eqref{mhd} on $[0,T]$, $(\psi,\psit)$ be the corresponding stream function. $(G,\Gt)$ is defined by \eqref{gf}, and the weight function is $\Psi(t,y)=\frac{y^2}{8\ltr}$. If $\ka\in(0,2)$, then for any $\ga\in(0,1),t\leq T,$ we have the following inequalities hold,
\begin{align}\label{henyouyong}
\begin{aligned}
\left\|e^{\gamma \Psi} \Delta_{k}^{\mathrm{h}} u_{\Phi}(t)\right\|_{L_{+}^{2}} & \lesssim\left\|e^{\Psi} \Delta_{k}^{\mathrm{h}} G_{\Phi}(t)\right\|_{L_{+}^{2}}, \\
\left\|e^{\gamma \Psi} \Delta_{k}^{\mathrm{h}} \partial_{y} u_{\Phi}(t)\right\|_{L_{+}^{2}}  &\lesssim\left\|e^{\Psi} \Delta_{k}^{\mathrm{h}} \partial_{y} G_{\Phi}(t)\right\|_{L_{+}^{2}}, \\
\left\|e^{\gamma \Psi} \Delta_{k}^{\mathrm{h}} \partial_{y}^{2} \uph(t)\right\|_{L_{+}^{2}}  &\lesssim\left\|e^{\Psi} \Delta_{k}^{\mathrm{h}} \partial_{y}^{2} G_{\Phi}(t)\right\|_{L_{+}^{2}}, \\
\left\|e^{\gamma \Psi} \Delta_{k}^{\mathrm{h}} \partial_{y}^{3} \uph(t)\right\|_{L_{+}^{2}}  &\lesssim\left\|e^{\Psi} \Delta_{k}^{\mathrm{h}} \partial_{y}^{3} G_{\Phi}(t)\right\|_{L_{+}^{2}}, \\
\left\|e^{\gamma \Psi} \Delta_{k}^{\mathrm{h}} \partial_{y}^{2} \uph(t)\right\|_{L_{\mathrm{v}}^{\infty}(L_{\mathrm{h}}^{2})} &\lesssim\left\|e^{\Psi} \Delta_{k}^{\mathrm{h}} \partial_{y}^{2} G_{\Phi}(t)\right\|_{L_{\mathrm{v}}^{\infty}(L_{\mathrm{h}}^{2})},\\
\langle t\rangle^{-1}\left\|e^{\gamma \Psi} \Delta_{k}^{\mathrm{h}} \partial_{y}(y \psi)_{\Phi}(t)\right\|_{L_{+}^{2}}&+\langle t\rangle^{-\frac{1}{2}}\left\|e^{\gamma \Psi} \Delta_{k}^{\mathrm{h}} \partial_{y}^{2}(y \psi)_{\Phi}(t)\right\|_{L_{+}^{2}}\\
+\langle t\rangle^{-\frac{3}{4}}\left\|e^{\gamma \Psi} \Delta_{k}^{\mathrm{h}} \partial_{y}(y \psi)_{\Phi}(t)\right\|_{L_{\mathrm{v}}^{\infty}\left(L_{\mathrm{h}}^{2}\right)}+\langle t\rangle^{-\frac{1}{4}}&\left\|e^{\gamma \Psi} \Delta_{k}^{\mathrm{h}} \partial_{y}^{2}(y \psi)_{\Phi}(t)\right\|_{L_{\mathrm{v}}^{\infty}\left(L_{\mathrm{h}}^{2}\right)} \lesssim\left\|e^{\Psi} \Delta_{k}^{\mathrm{h}} \partial_{y} G_{\Phi}(t)\right\|_{L_{+}^{2}},\\
\langle t\rangle^{-\frac{1}{2}}\left\|e^{\gamma \Psi} \Delta_{k}^{\mathrm{h}} \partial_{y}^{3}(y \psi)_{\Phi}(t)\right\|_{L_{+}^{2}} &\lesssim\left\|e^{\Psi} \Delta_{k}^{\mathrm{h}} \partial_{y}^{2} G_{\Phi}(t)\right\|_{L_{+}^{2}} .
\end{aligned}
\end{align}
The above inequalities also hold if we replace $(u,\psi,G)$ with $(b,\psit,\Gt)$.
\end{Lemma}
Except for $\eqref{henyouyong}_4$ which is not proved, the proof of the rest of the inequalities is given in \cite{NL} and \cite{CW}. The proof of $\eqref{henyouyong}_4$ is the same as that of $\eqref{henyouyong}_3$, so for the sake of brevity, we do not give the proof here.
\begin{Remark}
Notice that each term of our lemma $\ref{gu7}$ contains the operator $\Phi$ $($compared to \cite{CW} with some differences$)$, this is because we have to do the estimates of $\py^2 G_{\Phi}$ and $\py^3 G_{\Phi}$ to close the a priori estimates, while in \cite{CW} only the estimates of $\py^2 G$ and $\py^3 G$. See \eqref{T**} the definition of $T_*$ for detail.
\end{Remark}

\section{Sketch of The Proof to Theorem \ref{th}}\label{zhengmingdingli}
In this section, we give the outline of the proof of theorem \ref{th}. First we rewrite the form of the equation \eqref{mhd}.
\subsection{New formulation of the equation}
To make Gevrey norm estimates for the solution of the equation \eqref{mhd}, we use the operator $e^{\Phi(t,D_x)}$ to act on the equation \eqref{mhd} and do the Bony's decomposition of all nonlinear terms using \eqref{bony}. In order to control the higher order terms by using the lemma \ref{gu3}, lemma \ref{gu4} and lemma \ref{gu5}, we write the equation in the following form:
\begin{align}\label{mhdp}
\left\{\begin{array}{l}
\Ll \uph=-\T_{\py u}\vp-\frac{2}{3}\de(t)\T_{\py D_x u}Q(D_x)\vp-A+\tad\Pp\\
\Lk \bp=-\T_{\py b}\vp-\frac{2}{3}\de(t)\T_{\py D_x b}Q(D_x)\vp-B+\tad\Np\\
\px \uph+\py \vp=\px \bp+\py \hp=0,
 \\
 (\uph,\vp,\bp,\hp)|_{y=0}=(0,0,0,0),~(\uph,\bp)|_{y\rightarrow +\ty}=(0,0),
\\
(\uph,\bp)|_{t=0}=(e^{\de[D_x]^{\frac{2}{3}}}u_0,
e^{\de[D_x]^{\frac{2}{3}}}b_0).
\end{array}\right.
\end{align}
where the operator $\mathcal{L_{\ka}}$ is defined as follows:
\begin{align}\label{lk}
\Lk=\pt+\la\tad(t)[D_x]^{\frac{2}{3}}+\T_u\px+\T_v\py+\frac{2}{3}\de(t)
\T_{D_x u}Q(D_x)\px-\ka\py^2,
\end{align}
In particular, it is denoted by $\Ll$ when $\ka=1$. The definitions of $A$ and $B$ are as follows:
\begin{align}\label{A}
\begin{split}
A=&\left(T_{u}^{\mathrm{h}} \partial_{x} u\right)_{\Phi}-T_{u}^{\mathrm{h}} \partial_{x} u_{\Phi}-\frac{2}{3}\delta(t) T_{D_{x} u}^{\mathrm{h}} Q(D_{x}) \partial_{x} u_{\Phi}\\
&+\left(T_{\py u}^{\mathrm{h}} v\right)_{\Phi}-T_{\py u}^{\mathrm{h}} v_{\Phi}-\frac{2}{3}\delta(t) T_{\py D_{x} u}^{\mathrm{h}} Q(D_{x}) v_{\Phi}+\left(\T_v \py u\right)_{\Phi}-\T_v\py\uph\\
&+\left(\T_{\px u}u+R^{\mathrm{h}}(u,\px u)+R^{\mathrm{h}}(v,\py u)\right)_{\Phi},
\end{split}
\end{align}
\begin{align}\label{B}
\begin{split}
B=&\left(T_{u}^{\mathrm{h}} \partial_{x} b\right)_{\Phi}-T_{u}^{\mathrm{h}} \partial_{x} b_{\Phi}-\frac{2}{3}\delta(t) T_{D_{x} u}^{\mathrm{h}} Q(D_{x}) \partial_{x} b_{\Phi}\\
&+\left(T_{\py b}^{\mathrm{h}} v\right)_{\Phi}-T_{\py b}^{\mathrm{h}} v_{\Phi}-\frac{2}{3}\delta(t) T_{\py D_{x} b}^{\mathrm{h}} Q(D_{x}) v_{\Phi}+\left(\T_v \py b\right)_{\Phi}-\T_v\py\bp\\
&+\left(\T_{\px b}u+R^{\mathrm{h}}(u,\px b)+R^{\mathrm{h}}(v,\py b)\right)_{\Phi}.
\end{split}
\end{align}

Inspired by \cite{HD,WXLY,CW}, in order to deal with the difficult term $\vp$, let $\U$ be the solution of the following initial boundary value problem,
\begin{align}\label{U}
\left\{\begin{array}{l}
\Ll\int_y^{\ty}\U dz=-\tad(t)\vp,
 \\
\py\U|_{y=0}=\lim\limits_{y\rightarrow+\ty}\U=0,
\\
\U|_{t=0}=0.
\end{array}\right.
\end{align}
Here the existence of $\U$ can be guaranteed by the proposition \ref{propositionU} and the proposition \ref{propositionze}.
\begin{Remark}
The reason for using $\int_y^{\ty}\U dz$ (instead of $\int_{0}^{y}\U dz$) in \eqref{U} is to ensure that the solution $\U$ decays quickly to $0$ as $y$ tends to $+\ty$.
\end{Remark}
\begin{Remark}\label{duoletad}
In \eqref{U}, $\vp$ is preceded by a $\tad(t),$ which is intended to make Gevrey norm estimates of $\U$ along with estimates of its Gevrey radius. This idea goes back to Chemin\cite{JYC}, who used this method for the classical Navier-stokes equations. Ping Zhang and others \cite{CW} used this method to obtain the global small solution of the two-dimensional Prandtl equation in the optimal Gevrey space.
\end{Remark}
By applying $\py$ to \eqref{U}, we obtain the following equation
\begin{align}\label{U2}
\tad\px\uph=-\Ll\U+\T_{\py u}\px\int_y^{\ty}\U dz-\T_{\py v}\U+\frac{2}{3}\de(t)\T_{\py D_x u}Q(D_x)\px\int_y^{\ty}\U dz.
\end{align}

In addition to $\U$, we introduce the following two quantities:
\begin{align}\label{ze}
\begin{split}
\ze=\uph-\frac{1}{\tad}\T_{\py u}\int_y^{\ty}\U dz
-\frac{2\de(t)}{3\tad}\T_{\py D_x u}Q(D_x)\int_y^{\ty}\U dz,
\\
\zet=\bp-\frac{1}{\tad}\T_{\py b}\int_y^{\ty}\U dz
-\frac{2\de(t)}{3\tad}\T_{\py D_x b}Q(D_x)\int_y^{\ty}\U dz.
\end{split}
\end{align}
Following the pioneering discovery in \cite{CW}, we define $\ta$ by the following equation
\begin{align}\label{theta}
\tad(t)=\ep^{\frac{1}{2}}\ltr^{-\al},~~\ta(0)=0.
\end{align}
where $\ep>0$ is sufficiently small and $\al>1$ is a constant that will be fixed in the section \ref{proof}.

Notice that $(G,\Gt)$ and $T^*$ are determined by \eqref{gf} and \eqref{T*}, respectively. Let $\ga_0\in(1,1+l_{\ka})$ and $C$ is a sufficiently large constant, then finally define $T_*$:
\begin{align}\label{T**}
\begin{aligned}
T_* \stackrel{\text { def }}{=} &\sup \left\{t<T^{*},\left\|(G_{\Phi}(t),\Gt_{\Phi}(t))\right\|_{H_{\Psi}^{4,0}}+
\langle t\rangle^{\frac{1}{2}}\left\|(\py G_{\Phi}(t),\py\Gt_{\Phi}(t))\right\|_{H_{\Psi}^{3,0}}\right. \\
&\left.+\langle t\rangle\left\|(\partial_{y}^{2} G_{\Phi}(t),\py^2\Gt_{\Phi}(t))\right\|_{H_{\Psi}^{3,0}}+\langle t\rangle^{\frac{3}{2}}\left\|(\partial_{y}^{3} G_{\Phi}(t),\py^3\Gt_{\Phi}(t))\right\|_{H_{\Psi}^{2,0}} \leq C \ep\langle t\rangle^{-\gamma_{0}}\right\} .
\end{aligned}
\end{align}

\subsection{The a priori estimates}
In this subsection, we list all the a priori estimates needed to prove theorem \ref{th}.

To obtain the Gevrey norm estimates for $(\uph,\bp)$, we first derive a priori estimates for the auxiliary functions $\U$ and $(\ze,\zet)$ in section \ref{gujiU} and section \ref{gujize}.
\begin{Proposition}\label{propositionU}
Let $\ka\in(0,2)$ be a given constant. Let $\U$ be a sufficiently smooth solution of the equation \eqref{U} on $[0,T_*]$ and $\U$ decays rapidly to $0$ as $y$ tends to $+\ty$. Then there exist $\ep_1$ and $\la_1$ such that when $\ep<\ep_1,\la>\la_1$ and $\al\leq\ga_0+\frac{1}{4}$, the following inequality holds for any $t<T_*$ and sufficiently small $\eta>0$:
\begin{align}\label{propositionU1}
\begin{aligned}
&\|\ltr^{l_{\ka}-\eta}\U(t)\|^2_{H_{\Psi}^{7,0}}+\eta\int_0^t \|\lsr^{l_{\ka}-\eta}\py\U(s)\|^2_{H_{\Psi}^{7,0}}ds\\
+&\la\int_0^t
\tad(s)\|\lsr^{l_{\ka}-\eta}\U(s)\|^2_{H_{\Psi}^{\frac{22}{3},0}}ds
\leq \int_0^t\tad(s)\|\lsr^{l_{\ka}-\eta}\ze(s)\|^2_{H_{\Psi}^{\frac{23}{3},0}}ds,
\end{aligned}
\end{align}
where $l_{\ka}=\frac{\ka(2-\ka)}{4}\in(0,\frac{1}{4}].$
\end{Proposition}
\begin{Proposition}\label{propositionze}
Let $\ka\in(0,2)$ be a given constant. Let $(\ze,\zet)$ be defined by \eqref{ze}. Then there exist $\ep_2$ and $\la_2$ such that when $\ep<\ep_2,\la>\la_2$ and $\al\leq\min\{\ga_0+\frac{1}{4},\frac{2}{3}\ga_0+\frac{1}{2},\frac{1}{2}\ga_0+\frac{5} {8}\}$, for any $t<T_*$ and sufficiently small $\eta>0$, the following inequality holds:
\begin{align}\label{propositionze1}
\begin{aligned}
&\|\ltr^{l_{\ka}-\eta}(\ze(t),\zet(t))\|^2_{H_{\Psi}^{\frac{22}{3},0}}
+\eta\int_0^t \|\lsr^{l_{\ka}-\eta}(\py\ze,\py\zet)\|^2_{H_{\Psi}^{\frac{22}{3},0}}ds\\
&+\la\int_0^t
\tad\|\lsr^{l_{\ka}-\eta}(\ze,\zet)\|^2_{H_{\Psi}^{\frac{23}{3},0}}ds\\
\leq&
\|(\uph(0),\bp(0))\|^2_{H_{\Psi}^{\frac{22}{3},0}}
+\frac{2}{25}\eta\int_0^t
\|\lsr^{l_{\ka}-\eta}\py\U\|^2_{H_{\Psi}^{7,0}}
ds
\\
&+\frac{1}{5}\la\int_0^t\tad
\left(\|\lsr^{l_{\ka}-\eta}
\U\|^2_{H_{\Psi}^{\frac{22}{3},0}}+\|\lsr^{l_{\ka}-\eta}
(\Pp,\Np)\|^2_{H_{\Psi}^{7,0}}
\right)ds,
\end{aligned}
\end{align}
where $l_{\ka}=\frac{\ka(2-\ka)}{4}\in(0,\frac{1}{4}].$
\end{Proposition}

Using the estimates of the auxiliary functions $\U$ and $(\ze,\zet)$, in section \ref{gujiub} we give the a priori estimates of $(\uph,\bp)$.
\begin{Proposition}\label{propositionub}
Let $\ka\in(0,2)$ be a given constant. Let $(u,b)$ be a sufficiently smooth solution of the equation \eqref{mhd} on $[0,T_*]$ and $(u,b)$ decays rapidly to $0$ as $y$ tends to $+\ty$. Then there exist $\ep_2$ and $\la_2$ such that when $\ep<\ep_2,\la>\la_2$ and $\al\leq\min\{\ga_0+\frac{1}{4},\frac{2}{3}\ga_0+\frac{1}{2},\frac{1}{2}\ga_0+\frac{5} {8}\}$, for any $t<T_*$ and sufficiently small $\eta>0$, the following inequality holds:
\begin{align}\label{propositionub1}
\begin{aligned}
&\|\ltr^{l_{\ka}-\eta}(\uph(t),\bp(t))\|^2_{H_{\Psi}^{7,0}}
+\la\int_0^t
\tad\|\lsr^{l_{\ka}-\eta}(\uph,\bp)\|^2_{H_{\Psi}^{\frac{22}{3},0}}ds\\
&+\eta\int_0^t \|\lsr^{l_{\ka}-\eta}
(\py\uph,\py\bp)\|^2_{H_{\Psi}^{7,0}}ds\\
\leq&
2\|(\uph(0),\bp(0))\|^2_{H_{\Psi}^{\frac{22}{3},0}}
+\frac{1}{5}\la\int_0^t\tad
\|\lsr^{l_{\ka}-\eta}(\Pp,\Np)\|^2_{H_{\Psi}^{7,0}}
ds,
\end{aligned}
\end{align}
where $l_{\ka}=\frac{\ka(2-\ka)}{4}\in(0,\frac{1}{4}].$
\end{Proposition}

In section \ref{lalalademaxiya}, we will deal with the difficulty term $(\Pp,\Np)$ caused by the magnetic field.
\begin{Proposition}\label{propositionPN}
Let $\ka\in(0,2)$ be a given constant. Let $(P,N)$ be defined by \eqref{PN}. Then there exist $\ep_3$ and $\la_2$ such that when $\ep<\ep_3,\la>\la_2$ and $\al\leq\min\{\ga_0+\frac{1}{4},\frac{2}{3}\ga_0+\frac{1}{2},\frac{1}{2}\ga_0+\frac{5} {8}\}$, for any $t<T_*$ and sufficiently small $\eta>0$, the following inequality holds:
\begin{align}\label{propositionPN1}
\begin{aligned}
&\|\ltr^{l_{\ka}-\eta}
(\Pp(t),\Np(t))\|^2_{H_{\Psi}^{\frac{20}{3},0}}
+\eta\int_0^t \|\lsr^{l_{\ka}-\eta}
(\py\Pp,\py\Np)\|^2_{H_{\Psi}^{\frac{20}{3},0}}ds\\
&
+\la\int_0^t
\tad\|\lsr^{l_{\ka}-\eta}(\Pp,\Np)\|^2_{H_{\Psi}^{7,0}}ds\\
\leq&
\|(\Pp(0),\Np(0))\|^2_{H_{\Psi}^{\frac{20}{3},0}}
+\frac{1}{10}\la\int_0^t\tad
\left(\|\lsr^{l_{\ka}-\eta}\U\|^2_{H_{\Psi}^{\frac{22}{3},0}}
+\|\lsr^{l_{\ka}-\eta}
(\uph,\bp)\|^2_{H_{\Psi}^{\frac{22}{3},0}}\right)
ds\\
&+\frac{13}{100}\eta\int_0^t
\left(\|\lsr^{l_{\ka}-\eta}\py\U\|^2_{H_{\Psi}^{7,0}}
+\|\lsr^{l_{\ka}-\eta}
(\py\ze,\py\zet)\|^2_{H_{\Psi}^{\frac{22}{3},0}}\right)ds,
\end{aligned}
\end{align}
where $l_{\ka}=\frac{\ka(2-\ka)}{4}\in(0,\frac{1}{4}].$
\end{Proposition}

From section \ref{gujiG} to section \ref{gujipy3G}, we will give the Gevrey norm estimates for $(\Gp,\Gtp)$ and $(\py^i\Gp,\py^i\Gtp),$ $i=1,2,3$ respectively. These estimates are very important to control the Gevrey radius and thus obtain the global solution.
\begin{Proposition}\label{propositionG}
Let $\ka\in(0,2)$ be a given constant. Let $(G,\Gt)$ be defined by \eqref{gf}. Then there exist $\ep_3$ and $\la_2$ such that when $\ep<\ep_3,\la>\la_2$ and $\al\leq\ga_0+\frac{5}{36}$, for any $t<T_*$ and sufficiently small $\eta>0$ the following inequality holds:
\begin{align}\label{propositionG1}
\begin{aligned}
&\|\ltr^{1+l_{\ka}-\eta}(\Gp(t),\Gtp(t))\|^2_{H_{\Psi}^{4,0}}
+\eta
\int_0^t\|
\lsr^{1+l_{\ka}-\eta}(\py\Gp,\py\Gtp)\|^2_{H_{\Psi}^{4,0}}ds\\
&+\la
\int_0^t\tad(s)\|\lsr^{1+l_{\ka}-\eta}(\Gp,\Gtp)\|^2_{H_{\Psi}^{\frac{13}{3},0}}ds\\
\leq&\|(\Gp(0),\Gtp(0))\|^2_{H_{\Psi}^{4,0}}
+\int_0^t\tad
\|\lsr^{l_{\ka}-\eta}\sqrt{a}(\uph,\bp)\|^{2}_{H_{\Psi}^{\frac{22}{3},0}}
ds
,
\end{aligned}
\end{align}
where $l_{\ka}=\frac{\ka(2-\ka)}{4}\in(0,\frac{1}{4}].$
\end{Proposition}
\begin{Proposition}\label{propositionpyG}
Let $\ka\in(0,2)$ be a given constant. Let $(G,\Gt)$ be defined by \eqref{gf}. Then there exist $\ep_3$ and $\la_2$ such that when $\ep<\ep_3,\la>\la_2$ and $\al\leq\ga_0+\frac{5}{36}$, for any $t<T_*$ and sufficiently small $\eta>0$ the following inequality holds:
\begin{align}\label{propositionpyG1}
\begin{aligned}
&\|\ltr^{\frac{3}{2}+l_{\ka}-\eta}(\py\Gp(t),\py\Gtp(t))\|^2_{H_{\Psi}^{3,0}}
+\la
\int_0^t\tad\|\lsr^{\frac{3}{2}+l_{\ka}-\eta}(\py\Gp,
\py\Gtp)\|^2_{H_{\Psi}^{\frac{10}{3},0}}ds
\\
&+(4l_{\ka}-\frac{7}{25}\eta)
\int_0^t\|\lsr^{\frac{3}{2}+l_{\ka}-\eta}
(\py^2\Gp,\py^2\Gtp)\|^2_{H_{\Psi}^{3,0}}ds\\
\leq&C\eta^{-1}\left(\|(\py\Gp(0),\py\Gtp(0))\|^2_{H_{\Psi}^{3,0}}
+\|(\Gp(0),\Gtp(0))\|^2_{H_{\Psi}^{4,0}}\right)\\
&+C\eta^{-1}\int_0^t\tad
\|\lsr^{l_{\ka}-\eta}\sqrt{a}(\uph,\bp)\|^{2}_{H_{\Psi}^{\frac{22}{3},0}}
ds,
\end{aligned}
\end{align}
where $l_{\ka}=\frac{\ka(2-\ka)}{4}\in(0,\frac{1}{4}].$
\end{Proposition}
The estimates of $(\py^i\Gp,\py^i\Gtp),i=2,3$ are crucial to deal with the difficulties posed by the magnetic field.
\begin{Proposition}\label{propositionpy2G}
Let $\ka\in(0,2)$ be a given constant. Let $(G,\Gt)$ be defined by \eqref{gf}. Then there exist $\ep_3$ and $\la_2$ such that when $\ep<\ep_3,\la>\la_2$ and $\al\leq\ga_0+\frac{5}{36}$, for any $t<T_*$ and sufficiently small $\eta>0$ the following inequality holds:
\begin{align}\label{propositionpy2G1}
\begin{aligned}
&\|\ltr^{2+l_{\ka}-\eta}(\py^2\Gp(t),\py^2\Gtp(t))
\|^2_{H_{\Psi}^{3,0}}
+(4l_{\ka}-\frac{13}{100}\eta)\int_0^t\|\lsr^{2+l_{\ka}-\eta}
(\py^3\Gp,\py^3\Gtp)\|^2_{H_{\Psi}^{3,0}}ds\\
&+\la\int_0^t\tad\|\lsr^{2+l_{\ka}-\eta}
(\py^2\Gp,\py^2\Gtp)\|^2_{H_{\Psi}^{\frac{10}{3},0}}ds
\\
\leq&
\dfrac{C}{\eta(4l_{\ka}-\frac{7}{25}\eta)}
\left(\|(\py\Gp(0),\py\Gtp(0))\|^2_{H_{\Psi}^{3,0}}
+\|(\Gp(0),\Gtp(0))\|^2_{H_{\Psi}^{4,0}}
+\|(\py^2\Gp(0),\py^2\Gtp(0))\|^2_{H_{\Psi}^{3,0}}\right)\\
&+\dfrac{C}{\eta(4l_{\ka}-\frac{7}{25}\eta)}\int_0^t\tad
\|\lsr^{l_{\ka}-\eta}\sqrt{a}(\uph,\bp)\|^{2}_{H_{\Psi}^{\frac{22}{3},0}}
ds,
\end{aligned}
\end{align}
where $l_{\ka}=\frac{\ka(2-\ka)}{4}\in(0,\frac{1}{4}].$
\end{Proposition}
\begin{Proposition}\label{propositionpy3G}
Let $\ka\in(0,2)$ be a given constant. Let $(G,\Gt)$ be defined by \eqref{gf}. Then there exist $\ep_3$ and $\la_2$ such that when $\ep<\ep_3,\la>\la_2$ and $\al\leq\ga_0+\frac{5}{36}$, for any $t<T_*$ and sufficiently small $\eta>0$ the following inequality holds:
\begin{align}\label{propositionpy3G1}
\begin{aligned}
&\|\ltr^{\frac{5}{2}+l_{\ka}-\eta}
(\py^3\Gp(t),\py^3\Gtp(t))\|^2_{H_{\Psi}^{2,0}}
+(4l_{\ka}-\frac{1}{25}\eta)\int_0^t\|\lsr^{\frac{5}{2}+l_{\ka}-\eta}
(\py^4\Gp,\py^4\Gtp)\|^2_{H_{\Psi}^{2,0}}ds\\
&+\la
\int_0^t\tad\|\lsr^{\frac{5}{2}+l_{\ka}-\eta}
(\py^3\Gp,\py^3\Gtp)\|^2_{H_{\Psi}^{\frac{7}{3},0}}ds\\
\leq&
\dfrac{C}{\eta(4l_{\ka}-\frac{7}{25}\eta)(4l_{\ka}-\frac{13}{100}\eta)}
\left(\|(\py\Gp(0),\py\Gtp(0))\|^2_{H_{\Psi}^{3,0}}
+\|(\Gp(0),\Gtp(0))\|^2_{H_{\Psi}^{4,0}}\right.\\
&\left.+\|(\py^2\Gp(0),\py^2\Gtp(0))\|^2_{H_{\Psi}^{3,0}}
+\|(\py^3\Gp(0),\py^3\Gtp(0))\|^2_{H_{\Psi}^{2,0}}\right)\\
&+\dfrac{C}{\eta(4l_{\ka}-\frac{7}{25}\eta)(4l_{\ka}-\frac{13}{100}\eta)}
\int_0^t\tad
\|\lsr^{l_{\ka}-\eta}\sqrt{a}(\uph,\bp)\|^{2}_{H_{\Psi}^{\frac{22}{3},0}}
ds,
\end{aligned}
\end{align}
where $l_{\ka}=\frac{\ka(2-\ka)}{4}\in(0,\frac{1}{4}].$
\end{Proposition}

\subsection{Proof of theorem \ref{th}}\label{proof}
Using the previous a priori estimates, we will give the proof of theorem \ref{th} in this section.

\textbf{Proof of Theorem \ref{th}:}
First of all, the existence of the local solution of the equation \eqref{mhd1}-\eqref{cb} can be found in \cite{WXLY}. Although the boundary condition in \cite{WXLY} is $\py b|_{y=0}=0$, it is easy to verify that the conclusion also holds for the boundary condition $b|_{y=0}=0$, since no additional boundary term is generated when we use integrating by parts.

Therefore it may be assumed that the equation \eqref{mhd1}-\eqref{cb} has a solution on $[0,\mathfrak{T}]$. If $\mathfrak{T}<\ty,$ then choose $\eta\in(0,l_{\ka})$ and let $\ga_0$ in \eqref{T**} be taken as $\ga_0=1+l_{\ka}-\eta$. At this point there is obvious
$$\min\{\ga_0+\frac{5}{36},\frac{2}{3}\ga_0+\frac{1}{2},
\frac{1}{2}\ga_0+\frac{5}{8}\}=\frac{9}{8}
+\frac{1}{2}l_{\ka}-\frac{1}{2}\eta.$$
Since \eqref{T**}, we define $T_*^{\eta}$ as follows:
\begin{align}\label{T**eta}
\begin{aligned}
T_*^{\eta} \stackrel{\text { def }}{=} &\sup \left\{t<T^{*},\left\|(G_{\Phi}(t),\Gt_{\Phi}(t))\right\|_{H_{\Psi}^{4,0}}+
\langle t\rangle^{\frac{1}{2}}\left\|(\py G_{\Phi}(t),\py\Gt_{\Phi}(t))\right\|_{H_{\Psi}^{3,0}}\right. \\
&\left.+\langle t\rangle\left\|(\partial_{y}^{2} G_{\Phi}(t),\py^2\Gt_{\Phi}(t))\right\|_{H_{\Psi}^{3,0}}+\langle t\rangle^{\frac{3}{2}}\left\|(\partial_{y}^{3} G_{\Phi}(t),\py^3\Gt_{\Phi}(t))\right\|_{H_{\Psi}^{2,0}} \leq C_0 \ep\langle t\rangle^{-1-l_{\ka}+\eta}\right\},
\end{aligned}
\end{align}
where $T^*$ is defined by \eqref{T*}, the constant $C_0$ is to be determined later on.

Let $E(t)$ and $D(t)$ are defined by \eqref{Et} and \eqref{Dt} respectively. And define $H(t)$ as follows:
\begin{align}\label{Ht}
\begin{aligned}
H(t)\stackrel{\text { def }}{=}&\|\ltr^{l_{\ka}-\eta}(\uph,\bp)\|_{H_{\Psi}^{\frac{22}{3},0}}
+\|\ltr^{l_{\ka}-\eta}\U\|_{H_{\Psi}^{\frac{22}{3},0}}
+\|\ltr^{l_{\ka}-\eta}(\ze,\zet)\|_{H_{\Psi}^{\frac{23}{3},0}}\\
&+\|\ltr^{l_{\ka}-\eta}(\Pp,\Np)\|_{H_{\Psi}^{7,0}}
+\|\ltr^{1+l_{\ka}-\eta}(\Gp,\Gtp)\|_{H_{\Psi}^{\frac{13}{3},0}}\\
&+\|\ltr^{\frac{3}{2}+l_{\ka}-\eta}(\py\Gp,\py\Gtp)\|_{H_{\Psi}^{\frac{10}{3},0}}
+\|\ltr^{2+l_{\ka}-\eta}(\py^2\Gp,\py^2\Gtp)\|_{H_{\Psi}^{\frac{10}{3},0}}\\
&+\|\ltr^{\frac{5}{2}+l_{\ka}-\eta}(\py^3\Gp,\py^3\Gtp)\|_{H_{\Psi}^{\frac{7}{3},0}}.
\end{aligned}
\end{align}

From the choice of $\{\ep_1,\ep_2,
\ep_3,\la_1,\la_2\}$, it follows that $\ep_1>\ep_2>
\ep_3$ÇÒ$\la_1<\la_2$. Thus when $\al\leq\frac{9}{8}
+\frac{1}{2}l_{\ka}-\frac{1}{2}\eta,\ep<\ep_3$ and $\la>\la_2$, combining
\eqref{propositionU1}-\eqref{propositionpy3G1} we can obtain
\begin{align}\label{jijiangwanjie}
\begin{aligned}
&E(t)+\frac{79}{100}\eta\int_0^t D(s)ds+\la\int_0^t\tad H(s)ds\\
\leq& C_{\eta}\left(\|(\uph(0),\bp(0))\|_{H_{\Psi}^{\frac{22}{3},0}}+E(0)\right)
+(1+C_{\eta}+\frac{1}{2}\la)\int_0^t\tad H(s)ds
\end{aligned}
\end{align}
Let $\la_{0}=\max\{2+2C_{\eta},\la_2\}$. Then when $\la>\la_3$, using
\eqref{E0} and \eqref{jijiangwanjie}, we can get
\begin{align}\label{jijiangwanjie2}
\begin{aligned}
E(t)+\frac{79}{100}\eta\int_0^t D(s)ds
\leq C_{\eta}\left(\|(\uph(0),\bp(0))
\|_{H_{\Psi}^{\frac{22}{3},0}}+E(0)\right)
\leq C_{\eta}\ep^2\leq\frac{C_0^2}{4}\ep^2,
\end{aligned}
\end{align}
where we take $C_0=2\sqrt{C_{\eta}}$ in \eqref{T**eta}.

On the other hand, take $\al=\frac{9}{8}
+\frac{1}{2}l_{\ka}-\frac{1}{2}\eta,\ep<\ep_3,$ satisfying all the requirements of proposition \ref{propositionU}-proposition \ref{propositionpy3G}. At this time, by
\eqref{theta} we have
\begin{align}\label{xiaobudian}
\begin{aligned}
\theta=\int_0^t\tad(s)ds=\ep^{\frac{1}{2}}\int_0^t\lsr^{-\al}
\leq C\ep^{\frac{1}{2}}\leq\frac{\delta}{4\la}
\end{aligned}
\end{align}
holds when $\ep<\ep_0$, where $\ep_0=\min\{\ep_3,\dfrac{\delta^2}{16C^2\la^2}\}$.

Therefore when $\ep<\ep_0$ and $\la>\la_0$, \eqref{jijiangwanjie2} and
\eqref{xiaobudian} contradict \eqref{T*} and \eqref{T**eta}. Thus, by a standard continuous argument, it follows that $\mathfrak{T}=\ty.$ Then we have completed the proof of theorem \ref{th}.

\hfill $\square$

\section{The Gevrey Estimates of $\U$}\label{gujiU}
In this section we will give a priori estimates for $\U$. First the smallness and decay of the lower order derivatives of $(\uph,\bp)$ are given by the following lemma:
\begin{Lemma}\label{ub}
Let $(u,b)$ be a sufficiently smooth solution of the equation \eqref{mhd} on $[0,T_*]$ with $\ka\in(0,2)$. Then for any $\ga\in(0,1),t\leq T_*$, we have
\begin{align*}
\left\|(\uph(t),\bp(t))\right\|_{H_{\gamma \Psi}^{4,0}}+\langle t\rangle^{\frac{1}{2}}\|(\py\uph(t),\py\bp&(t))\|_{H_{\gamma \Psi}^{3,0}}+\langle t\rangle\left\|(\py^{2} u_{\Phi}(t),\py^2 b_{\Phi}(t))\right\|_{H_{\gamma \Psi}^{3,0}}\\
+
\langle t\rangle^{\frac{3}{2}}\|(\py^{3} u_{\Phi}(t),\py^3 &b_{\Phi}(t))\|_{H_{\gamma \Psi}^{2,0}}\leq C \ep\langle t\rangle^{-\gamma_{0}},
\\
\left\|(\uph(t),\bp(t))\right\|_{L^{\ty}_{\vm,
\ga\Psi}(H_{\h}^3)}\leq C\ep\ltr^{-(\ga_0+\frac{1}{4})}&,~~\left\|(\vp(t),
\hp(t))\right\|_{L^{\ty}_{\vm,
\ga\Psi}(H_{\h}^3)}\leq C\ep\ltr^{-(\ga_0-\frac{1}{4})},
\\
\left\|(\py u_{\Phi}(t),\py b_{\Phi}(t))\right\|_{L^{\ty}_{\vm,
\ga\Psi}(H_{\h}^3)}\leq C\ep\ltr^{-(\ga_0+\frac{3}{4})}&,~~\left\|(\py^2 u_{\Phi}(t),\py^2 b_{\Phi}(t))\right\|_{L^{\ty}_{\vm,
\ga\Psi}(H_{\h}^2)}\leq C\ep\ltr^{-(\ga_0+\frac{5}{4})}.
\end{align*}
ÆäÖÐ$\|f\|_{L_{\vm,\Psi}^{\ty}}\stackrel{\text { def }}{=}\|e^{\Psi}f\|_{L_{\vm}^{\ty}}$.
\end{Lemma}
This lemma is slightly different from lemma 4.1 in \cite{CW}, but the proof procedure is exactly the same and is not given here.
\begin{Remark}
Notice that $|f|\leq|e^{\ga\Psi}f|$, so in fact the conclusion of lemma $\mathrm{\ref{ub}}$  holds for the unweighted case $(\ga=0)$ as well. On the other hand, noting the definition of $(\uph,\bp)$, we can know that the above control also holds for $(u,b)$.
\end{Remark}

Next, we give a priori estimates for $\U$, and the detailed results are shown below:
\begin{Proposition}\label{guU}
Let $\ka\in(0,2)$ be a given constant. Let $\U$ be a sufficiently smooth solution of the equation \eqref{U} on $[0,T_*]$ and $\U$ decays rapidly to $0$ as $y$ tends to $+\ty$. Let $a(t)$ be a nonnegative and nondecreasing function on $\R_+$. Then when $\al\leq\ga_0+\frac{1}{4}$, for any $t<T_*$ and sufficiently small $\eta>0$ the following inequality holds:
\begin{align}\label{Udekongzhi}
\begin{aligned}
\|\sqrt{a}&\U(t)\|^2_{H_{\Psi}^{7,0}}-\int_0^t \|\sqrt{a'}\U(s)\|^2_{H_{\Psi}^{7,0}}ds+(1-\frac{1}{50}\eta)\int_0^t \|\sqrt{a}\py\U(s)\|^2_{H_{\Psi}^{7,0}}ds\\
&+2\left(\la-C(1+\eta^{-1}\ep^{\frac{3}{2}})\right)\int_0^t
\tad(s)\|\sqrt{a}\U(s)\|^2_{H_{\Psi}^{\frac{22}{3},0}}ds\leq \int_0^t\tad(s)\|\sqrt{a}\ze(s)\|^2_{H_{\Psi}^{\frac{23}{3},0}}ds.
\end{aligned}
\end{align}
\end{Proposition}
\begin{proof}
Combining \eqref{U2} and \eqref{ze}, we can deduce that
\begin{align}\label{guU1}
\Ll\U=-\tad\px\ze-\T_{\py \px u}\int_y^{\ty}\U dz-\T_{\py v}\U-\frac{2}{3}\de(t)\T_{\py \px D_x u}Q(D_x)\int_y^{\ty}\U dz.
\end{align}
Taking $H_{\Psi}^{7,0}$ inner product of $a(t)\U$ with \eqref{guU1}, we obtain
\begin{align*}
a\left(\Ll\U,\U\right)_{H_{\Psi}^{7,0}}=&-\tad a\left(\px\ze,\U\right)_{H_{\Psi}^{7,0}}-a\left(\T_{\py \px u}\int_y^{\ty}\U dz,\U\right)_{H_{\Psi}^{7,0}}
\\
&-a\left(\T_{\py v}\U,\U\right)_{H_{\Psi}^{7,0}}-a\left(\frac{2}{3}\de(t)\T_{\py \px D_x u}Q(D_x)\int_y^{\ty}\U dz,\U\right)_{H_{\Psi}^{7,0}}.
\end{align*}
Using Cauchy's inequality, we have
\begin{align*}
\left|\tad a\left(\px\ze,\U\right)_{H_{\Psi}^{7,0}}\right|\leq \frac{1}{2}\tad\|\sqrt{a}\ze\|^2_{H_{\Psi}^{\frac{23}{3},0}}+
\frac{1}{2}\tad\|\sqrt{a}\U\|^2_{H_{\Psi}^{\frac{22}{3},0}}.
\end{align*}
Next by H$\mathrm{\ddot{o}}$lder's inequality and lemma \ref{gu1}, we can deduce
\begin{align*}
\left|a\left(\T_{\py \px u}\int_y^{\ty}\U dz,\U\right)_{H_{\Psi}^{7,0}}\right|\leq\|\py u\|_{L^2_{\vm}(H_{\h}^{\frac{3}{2}+})}\|\sqrt{a}\int_y^{\ty}\U dz\|_{L^{\ty}_{\vm,\Psi}(H_{\h}^{7})}\|\sqrt{a}\U\|_{H_{\Psi}^{7,0}}.
\end{align*}

Notice that the following equation holds for any $\si\in\R$,
\begin{align}\label{ty}
\begin{aligned}
\|\int_y^{\ty}f dz\|_{L^{\ty}_{\vm,\Psi}(H_{\h}^{\si})}&\leq\|\int_y^{\ty}
e^{-\frac{(y-z)^2}{8\ltr}}e^{\frac{z^2}{8\ltr}}\|f(\cdot,z)\|_{H_{\h}^{\si}}dz\|_{L^{\ty}_{\vm}}\\
&\leq\|e^{-\frac{y^2}{8\ltr}}\|_{L^2_{\vm}}\|f\|_{H_{\Psi}^{\si,0}}\leq C\ltr^{\frac{1}{4}}\|f\|_{H_{\Psi}^{\si,0}}.
\end{aligned}
\end{align}
Combine \eqref{theta},\eqref{ty} and lemma \ref{ub}, we obtain
\begin{align*}
\left|a\left(\T_{\py \px u}\int_y^{\ty}\U dz,\U\right)_{H_{\Psi}^{7,0}}\right|&\leq C\ep\ltr^{-\ga_0-\frac{1}{4}}\|\sqrt{a}\U\|^2_{H_{\Psi}^{7,0}}\\
&\leq C\ep^{\frac{1}{2}}\ltr^{\al-\ga_0-\frac{1}{4}}\tad\|\sqrt{a}\U
\|^2_{H_{\Psi}^{\frac{22}{3},0}}\\
&\leq C\ep^{\frac{1}{2}}\tad\|\sqrt{a}\U\|^2_{H_{\Psi}^{\frac{22}{3},0}},
\end{align*}
here we use the fact that $\al\leq\ga_0+\frac{1}{4}$.

Noticing that $Q(\xi)=\xi(1+\xi^2)^{-\frac{2}{3}},$ we can obtain the following in the same way,
\begin{align*}
\left|a\left(\frac{2}{3}\de(t)\T_{\py\px D_x u}Q(D_x)\int_y^{\ty}\U dz,\U\right)_{H_{\Psi}^{7,0}}\right|\leq C\ep^{\frac{1}{2}}\tad\|\sqrt{a}\U\|^2_{H_{\Psi}^{\frac{22}{3},0}}.
\end{align*}

On the other hand, since $\px u+\py v=0,$, we can use H$\mathrm{\ddot{o}}$lder's inequality, lemma \ref{gu1} and lemma \ref{ub} to obtain that
\begin{align*}
\left|a\left(\T_{\py v}\U,\U\right)_{H_{\Psi}^{7,0}}\right|
&\leq\|\py v\|_{L^{\ty}_{\vm}(H_{\h}^{\frac{1}{2}+})}\|\sqrt{a}\U\|^2_{H_{\Psi}^{7,0}}
\\
&\leq C\ep\ltr^{-\ga_0-\frac{1}{4}}\|\sqrt{a}\U\|^2_{H_{\Psi}^{7,0}}\\
&\leq C\ep^{\frac{1}{2}}\ltr^{\al-\ga_0-\frac{1}{4}}\tad\|\sqrt{a}\U
\|^2_{H_{\Psi}^{\frac{22}{3},0}}\\
&\leq C\ep^{\frac{1}{2}}\tad\|\sqrt{a}\U\|^2_{H_{\Psi}^{\frac{22}{3},0}}.
\end{align*}
Next since $\pt\Psi+2(\py\Psi)^2=0$, using integrating by parts and Cauchy's inequality we can get
\begin{align}\label{ptpy}
\begin{split}
&~~~\left(\pt \U-\py^2\U,\U\right)_{H_{\Psi}^{7,0}}\\
&=\frac{1}{2}\dfrac{d}{dt}\|\U(t)\|^2_{H_{\Psi}^{7,0}}-\left(\pt \Psi,\U^2\right)_{H_{\Psi}^{7,0}}+\|\py\U(t)\|^2_{H_{\Psi}^{7,0}}
+\left(\py \U,2\U\py\Psi\right)_{H_{\Psi}^{7,0}}\\
&\geq \frac{1}{2}\dfrac{d}{dt}\|\U(t)\|^2_{H_{\Psi}^{7,0}}-\left(\pt \Psi,\U^2\right)_{H_{\Psi}^{7,0}}+\frac{1}{2}\|\py\U(t)\|^2_{H_{\Psi}^{7,0}}
-\left(2(\py\Psi)^2,\U^2\right)_{H_{\Psi}^{7,0}}\\
&=\frac{1}{2}\dfrac{d}{dt}\|\U(t)\|^2_{H_{\Psi}^{7,0}}
+\frac{1}{2}\|\py\U(t)\|^2_{H_{\Psi}^{7,0}}.
\end{split}
\end{align}
then combine \eqref{lk}, we can deduce
\begin{align}\label{UU}
\begin{split}
a(\Ll\U,\U)_{H_{\Psi}^{7,0}}&\geq\frac{1}{2}\dfrac{d}{dt}
\|\sqrt{a}\U(t)\|^2_{H_{\Psi}^{7,0}}-\frac{1}{2}\|\sqrt{a'}\U(t)\|^2_{H_{\Psi}^{7,0}}
+\la\tad\|\sqrt{a}\U(t)\|^2_{H_{\Psi}^{\frac{22}{3},0}}\\
&+\frac{1}{2}\|\sqrt{a}\py\U(t)\|^2_{H_{\Psi}^{7,0}}+
a\left(\T_u\px\U,\U\right)_{H_{\Psi}^{7,0}}+a\left(\T_v\py\U,\U\right)_{H_{\Psi}^{7,0}}\\
&+a\left(\frac{2}{3}\de(t)
\T_{D_x u}Q(D_x)\px\U,\U\right)_{H_{\Psi}^{7,0}}.
\end{split}
\end{align}
Using \eqref{theta}, lemma \ref{gu3} and lemma \ref{ub}, we have
\begin{align*}
\left|a\left(\T_u\px\U,\U\right)_{H_{\Psi}^{7,0}}\right|
&\leq C\|u\|_{L^{\ty}_{\vm}(H_{\h}^{\frac{3}{2}+})}\|\sqrt{a}\U\|^2_{H_{\Psi}^{7,0}}\\
&\leq C\ep^{\frac{1}{2}}\ltr^{\al-\ga_0-\frac{1}{4}}\tad\|\sqrt{a}\U
\|^2_{H_{\Psi}^{\frac{22}{3},0}}\\
&\leq C\ep^{\frac{1}{2}}\tad\|\sqrt{a}\U\|^2_{H_{\Psi}^{\frac{22}{3},0}}.
\end{align*}
By lemma \ref{gu1} and lemma\ref{ub}, it gives
\begin{align*}
\left|a\left(\T_v\py\U,\U\right)_{H_{\Psi}^{7,0}}\right|
&\leq C \|v\|_{L^{\ty}_{\vm}(H_{\h}^{\frac{1}{2}+})}
\|\sqrt{a}\U\|_{H_{\Psi}^{\frac{22}{3},0}}
\|\sqrt{a}\py \U\|_{H_{\Psi}^{7,0}}\\
&\leq C\eta^{-1}\ep^{\frac{3}{2}}\ltr^{\al-2\ga_0
+\frac{1}{2}}\tad\|\sqrt{a}\U\|^2_{H_{\Psi}^{\frac{22}{3},0}}
+\frac{1}{100}\eta\|\sqrt{a}\py \U\|^2_{H_{\Psi}^{7,0}}\\
&\leq C\eta^{-1}\ep^{\frac{3}{2}}\tad\|\sqrt{a}\U\|^2_{H_{\Psi}^{\frac{22}{3},0}}
+\frac{1}{100}\eta\|\sqrt{a}\py \U\|^2_{H_{\Psi}^{7,0}},
\end{align*}
here we use the fact that $\al\leq\ga_0+\frac{1}{4}$ and $\ga_0>1$.

Similarly, applying lemma \ref{gu1} and lemma \ref{ub} yields
\begin{align*}
\left|a\left(\frac{2}{3}\de(t)
\T_{D_x u}Q(D_x)\px\U,\U\right)_{H_{\Psi}^{7,0}}\right|
&\leq C\|u\|_{L^{\ty}_{\vm}(H_{\h}^{\frac{3}{2}+})}\|\sqrt{a}\U\|^2_{H_{\Psi}^{\frac{22}{3},0}}\\
&\leq C\ep^{\frac{1}{2}}\ltr^{\al-\ga_0-\frac{1}{4}}\tad\|\sqrt{a}\U
\|^2_{H_{\Psi}^{\frac{22}{3},0}}\\
&\leq C\ep^{\frac{1}{2}}\tad\|\sqrt{a}\U\|^2_{H_{\Psi}^{\frac{22}{3},0}}.
\end{align*}
Combining all the above estimates and integrating over $[0,t]$, we can obtain
\begin{align*}
\|\sqrt{a}&\U(t)\|^2_{H_{\Psi}^{7,0}}-\int_0^t \|\sqrt{a'}\U(s)\|^2_{H_{\Psi}^{7,0}}ds+(1-\frac{1}{50}\eta)\int_0^t \|\sqrt{a}\py\U(s)\|^2_{H_{\Psi}^{7,0}}ds\\
&+2\left(\la-C(1+\eta^{-1}\ep^{\frac{3}{2}})\right)\int_0^t
\tad(s)\|\sqrt{a}\U(s)\|^2_{H_{\Psi}^{\frac{22}{3},0}}ds\leq \int_0^t\tad(s)\|\sqrt{a}\ze(s)\|^2_{H_{\Psi}^{\frac{23}{3},0}}ds.
\end{align*}
This completes the proof of proposition \ref{guU}.
\end{proof}

Using proposition \ref{guU}, we can give a proof of proposition \ref{propositionU}.

\textbf{Proof of Proposition \ref{propositionU}:}
Taking $a(t)=\ltr^{2l_{\ka}-2\eta}$ in \eqref{Udekongzhi},where $l_{\ka}=\frac{\ka(2-\ka)}{4}\in(0,\frac{1}{4}].$ Using lemma \ref{gu6} gives
\begin{align*}
&-(2l_{\ka}-2\eta)\int_0^t \|\lsr^{l_{\ka}-\eta-\frac{1}{2}}\U(s)\|^2_{H_{\Psi}^{7,0}}ds+(1-\frac{1}{50}\eta)\int_0^t \|\lsr^{l_{\ka}-\eta}\py\U(s)\|^2_{H_{\Psi}^{7,0}}ds\\
\geq&
(-4l_{\ka}+4\eta)\int_0^t \|\lsr^{l_{\ka}-\eta}\py\U(s)\|^2_{H_{\Psi}^{7,0}}ds+
(1-\frac{1}{50}\eta)\int_0^t \|\lsr^{l_{\ka}-\eta}\py\U(s)\|^2_{H_{\Psi}^{7,0}}ds\\
\geq&\eta
\int_0^t \|\ltr^{l_{\ka}-\eta}\py\U(s)\|^2_{H_{\Psi}^{7,0}}ds,
\end{align*}
then we have
\begin{align*}
&\|\ltr^{l_{\ka}-\eta}\U(t)\|^2_{H_{\Psi}^{7,0}}+\eta\int_0^t \|\lsr^{l_{\ka}-\eta}\py\U(s)\|^2_{H_{\Psi}^{7,0}}ds\\
+2&\left(\la-C(1+\eta^{-1}\ep^{\frac{3}{2}})\right)\int_0^t
\tad(s)\|\lsr^{l_{\ka}-\eta}\U(s)\|^2_{H_{\Psi}^{\frac{22}{3},0}}ds
\leq \int_0^t\tad(s)\|\lsr^{l_{\ka}-\eta}\ze(s)\|^2_{H_{\Psi}^{\frac{23}{3},0}}ds.
\end{align*}
Take $\ep_1<1$ sufficiently small such that when $\ep<\ep_1$, we have $1+\eta^{-1}\ep^{\frac{3}{2}}<2.$ Take $\la_1>1$ sufficiently large, and $\la$ satisfies
$\la>\la_1>4C.$ So proposition \ref{propositionU} holds.

\hfill $\square$
\section{The Gevrey Estimates of $\ze$ and $\zet$}\label{gujize}
In this section, we will derive a priori estimates for $\ze$ and $\zet$. Since the two derivations are similar, we only give the proof procedure for $\zet$ here. Using \eqref{mhdp}, \eqref{U} and \eqref{ze}, we can obtain that $\zet$ satisfies the following equation,
\begin{align}\label{zetequation}
\begin{split}
\Lk \zet=&-B+\tad \Np+\frac{1-\ka}{\tad}\T_{\py b}\py\U+\left[\frac{1}{\tad}\T_{\py b};\Lk\right]\int_{y}^{\ty}\U dz\\
&+\frac{2(1-\ka)}{3\tad}\de(t)\T_{\py D_x b}Q(D_x)\py\U
+\left[\frac{2\de(t)}{3\tad}\T_{\py D_x b}Q(D_x);\Lk\right]\int_{y}^{\ty}\U dz.
\end{split}
\end{align}
By \eqref{U}, we can get
\begin{align}\label{fangbianyinyong}
\left(\pt+\la\tad(t)[D_x]^{\frac{2}{3}}\right)\int_0^{\ty}\U dz=0.
\end{align}
By taking $L^2_{\h}(\R)$ inner product of \eqref{fangbianyinyong} with $\int_0^{\ty}\U dz$, we observe
$$\frac{d}{dt}\left\|\int_0^{\ty}\U(t,\cdot,z) dz\right\|^2_{L^2_{\h}}\leq0.$$
Since $\U|_{t=0}=0,$ so $\int_0^{\ty}\U(t,x,z) dz=0.$ Thus we can obtain $\zet$ satisfying the following initial boundary values
\begin{align}\label{zetchubianzhi}
\zet|_{y=0}=\lim_{y\rightarrow+\ty}\zet=0,~~\zet|_{t=0}=\bp|_{t=0}=
e^{\de[D_x]^{\frac{2}{3}}}b_0.
\end{align}

To complete the a priori estimates we also need estimates of the source term $(A,B)$ of the equation \eqref{mhdp}.
\begin{Lemma}\label{guAB}
Let $(A,B)$ be given by \eqref{A} and \eqref{B}, respectively, and $\ka\in(0,2)$. Then for any $s>0,$ we have
\begin{align*}
\|(A,B)\|_{H_{\Psi}^{s,0}}
\leq &C\ltr^{\frac{1}{4}}\left(
\|(\py\Gp,\py\Gtp)\|_{H_{\Psi}^{\frac{5}{2}+,0}}\|(\uph,\bp)\|_{H_{\Psi}^{s+\frac{1}{3},0}}\right.
\\
&+\left.\|(\Gp,\Gtp)\|_{H_{\Psi}^{\frac{5}{2}+,0}}\|(\py\uph,\py\bp)\|_{H_{\Psi}^{s-\frac{1}{3},0}}\right).
\end{align*}
\end{Lemma}
The proof of this lemma can be found in lemma 4.2 of \cite{CW}, the only thing to note is that our lemma \ref{gu4} and lemma \ref{gu5} are slightly different from \cite{CW}.

Next we can give the estimates of $(\ze,\zet)$.
\begin{Proposition}\label{guze}
Let $\ka\in(0,2)$ be a given constant. Let $(\ze,\zet)$ be defined by \eqref{ze}, and $a(t)$ be a non-negative and non-decreasing function on $\R_+$. Then when $\al\leq\min\{\ga_0+\frac{1}{4},\frac{2}{3}\ga_0+\frac{1}{2},\frac{1}{2}\ga_0+\frac{5}{8}\}$ and $\ep<\ep_*$ $($$\ep_*$ is a sufficiently small number$)$, for any $t<T_*$ and sufficiently small $\eta>0$ the following inequality holds,
\begin{align}\label{zedekongzhi}
\begin{aligned}
&~~~~~\|\sqrt{a}(\ze(t),\zet(t))\|^2_{H_{\Psi}^{\frac{22}{3},0}}
-\int_0^t \|\sqrt{a'}(\ze,\zet)\|^2_{H_{\Psi}^{\frac{22}{3},0}}ds\\
&~~~~~+(4l_{\ka}-\frac{1}{10}\eta)\int_0^t \|\sqrt{a}(\py\ze,\py\zet)\|^2_{H_{\Psi}^{\frac{22}{3},0}}ds\\
&~~~~~
+2\left(\la-C(1+\la\ep^{\frac{1}{2}}+\eta^{-1}\ep^{\frac{1}{2}})\right)\int_0^t
\tad\|\sqrt{a}(\ze,\zet)\|^2_{H_{\Psi}^{\frac{23}{3},0}}ds\\
&\leq
\|\sqrt{a}(0)(\uph(0),\bp(0))\|^2_{H_{\Psi}^{\frac{22}{3},0}}
+\frac{2}{25}\eta\int_0^t
\|\sqrt{a}\py\U\|^2_{H_{\Psi}^{7,0}}
ds
\\
&~~~~~+C(1+\ep^{\frac{1}{2}}\la)\int_0^t\tad
\left(\|\sqrt{a}\U\|^2_{H_{\Psi}^{\frac{22}{3},0}}+\|\sqrt{a}(\Pp,\Np)\|^2_{H_{\Psi}^{7,0}}
\right)ds,
\end{aligned}
\end{align}
where $l_{\ka}=\frac{\ka(2-\ka)}{4}\in(0,\frac{1}{4}].$
\end{Proposition}
\begin{proof}
By taking $H_{\Psi}^{\frac{22}{3},0}$ inner product of \eqref{zetequation} with $a(t)\zet$, we get
\begin{align}\label{zetzet}
\begin{split}
&a\left(\Lk \zet,\zet\right)_{H_{\Psi}^{\frac{22}{3},0}}\\
=&-a\left(B,\zet\right)_{H_{\Psi}^{\frac{22}{3},0}}
+a\left(\tad \Np,\zet\right)_{H_{\Psi}^{\frac{22}{3},0}}
+a\left(\frac{1-\ka}{\tad}\T_{\py b}\py\U,\zet\right)_{H_{\Psi}^{\frac{22}{3},0}}\\
&+a\left(\left[\frac{1}{\tad}\T_{\py b};\Lk\right]\int_{y}^{\ty}\U dz,\zet\right)_{H_{\Psi}^{\frac{22}{3},0}}
+a\left(\frac{2(1-\ka)}{3\tad}\de(t)\T_{\py D_x b}Q(D_x)\py\U,\zet\right)_{H_{\Psi}^{\frac{22}{3},0}}\\
&+a\left(\left[\frac{2\de(t)}{3\tad}\T_{\py D_x b}Q(D_x);\Lk\right]\int_{y}^{\ty}\U dz,\zet\right)_{H_{\Psi}^{\frac{22}{3},0}}\\
=&\sum_{i=1}^6 F_i.
\end{split}
\end{align}
Now start estimating the right side of \eqref{zetzet} term by term.

For $F_1$, using \eqref{T**} and lemma \ref{guAB}, we have
\begin{align*}
&~~~~~\left|a\left(B,\zet\right)_{H_{\Psi}^{\frac{22}{3},0}}\right|
\leq \|\sqrt{a}B\|_{H_{\Psi}^{7,0}}\|\sqrt{a}\zet\|_{H_{\Psi}^{\frac{23}{3},0}}\\
&\leq C\ltr^{\frac{1}{4}}\left(
\|(\py\Gp,\py\Gtp)\|_{H_{\Psi}^{\frac{5}{2}+,0}}
\|\sqrt{a}(\uph,\bp)\|_{H_{\Psi}^{\frac{22}{3},0}}\right.
\\
&~~~~~+\left.\|(\Gp,\Gtp)\|_{H_{\Psi}^{\frac{5}{2}+,0}}
\|\sqrt{a}(\py\uph,\py\bp)\|_{H_{\Psi}^{\frac{20}{3},0}}\right)\|\sqrt{a}\zet\|_{H_{\Psi}^{\frac{23}{3},0}}\\
&\leq C\ltr^{\frac{1}{4}}\left(\ep\ltr^{-\ga_0-\frac{1}{2}}\|\sqrt{a}(\uph,\bp)\|_{H_{\Psi}^{\frac{22}{3},0}}
+\ep\ltr^{-\ga_0}\|\sqrt{a}(\py\uph,\py\bp)\|_{H_{\Psi}^{7,0}}\right)\|\sqrt{a}\zet\|_{H_{\Psi}^{\frac{23}{3},0}}\\
&\leq C\ep^{\frac{1}{2}}\ltr^{\al-\ga_0-\frac{1}{4}}\tad
\|\sqrt{a}(\uph,\bp)\|_{H_{\Psi}^{\frac{22}{3},0}}
\|\sqrt{a}\zet\|_{H_{\Psi}^{\frac{23}{3},0}}\\
&~~~~~+C\ep\ltr^{-\ga_0+\frac{1}{4}}
\|\sqrt{a}(\py\uph,\py\bp)\|_{H_{\Psi}^{7,0}}
\|\sqrt{a}\zet\|_{H_{\Psi}^{\frac{23}{3},0}}
\\
&\leq C\ep^{\frac{1}{2}}\ltr^{\al-\ga_0-\frac{1}{4}}\tad
\|\sqrt{a}(\uph,\bp)\|^2_{H_{\Psi}^{\frac{22}{3},0}}
+C\ep^{\frac{1}{2}}\ltr^{\al-\ga_0-\frac{1}{4}}\tad
\|\sqrt{a}\zet\|^2_{H_{\Psi}^{\frac{23}{3},0}}\\
&~~~~~+\frac{1}{100}\eta\|\sqrt{a}(\py\uph,\py\bp)\|^2_{H_{\Psi}^{7,0}}
+C\eta^{-1}\ep^{\frac{3}{2}}\ltr^{\al-2\ga_0+\frac{1}{2}}\tad
\|\sqrt{a}\zet\|^2_{H_{\Psi}^{\frac{23}{3},0}},
\end{align*}
For the last inequality, we use the Cauchy's inequality and \eqref{theta}. On the other hand since $\al\leq\ga_0+\frac{1}{4}$ and $\ga>1$, we can deduce
\begin{align*}
\left|a\left(B,\zet\right)_{H_{\Psi}^{\frac{22}{3},0}}\right|
\leq &C(\ep^{\frac{1}{2}}+\eta^{-1}\ep^{\frac{3}{2}})
\tad\|\sqrt{a}\zet\|^2_{H_{\Psi}^{\frac{23}{3},0}}
+C\ep^{\frac{1}{2}}
\tad\|\sqrt{a}(\uph,\bp)\|^2_{H_{\Psi}^{\frac{22}{3},0}}\\
&+\frac{1}{100}\eta\|\sqrt{a}(\py\uph,\py\bp)\|^2_{H_{\Psi}^{7,0}}.
\end{align*}
For $F_2$, it's obvious that
\begin{align*}
\left|a\left(\tad \Np,\zet\right)_{H_{\Psi}^{\frac{22}{3},0}}\right|\leq
\frac{1}{2}\tad\|\sqrt{a}\Np\|_{H_{\Psi}^{7,0}}+
\frac{1}{2}\tad\|\sqrt{a}\zet\|_{H_{\Psi}^{\frac{23}{3},0}}.
\end{align*}
For $F_3$, applying \eqref{theta}, lemma \ref{gu1} and lemma \ref{ub}, it gives that
\begin{align*}
\left|a\left(\frac{1-\ka}{\tad}\T_{\py b}\py\U,\zet\right)_{H_{\Psi}^{\frac{22}{3},0}}\right|
&\leq C \frac{1}{\tad}\|\py b\|_{L_{\vm}^{\ty}(H_{\h}^{\frac{1}{2}+})}
\|\sqrt{a}\py\U\|_{H_{\Psi}^{7,0}}\|\sqrt{a}\zet\|_{H_{\Psi}^{\frac{23}{3},0}}
\\
&\leq C\ep^{\frac{1}{2}}\ltr^{\al-\ga_0-\frac{3}{4}}
\|\sqrt{a}\py\U\|_{H_{\Psi}^{7,0}}\|\sqrt{a}\zet\|_{H_{\Psi}^{\frac{23}{3},0}}
\\
&\leq C\eta^{-1}\ep\ltr^{2\al-2\ga_0-\frac{3}{2}}\|\sqrt{a}\zet\|^2_{H_{\Psi}^{\frac{23}{3},0}}
+\frac{1}{100}\eta\|\sqrt{a}\py\U\|^2_{H_{\Psi}^{7,0}}
\\
&\leq C\eta^{-1}\ep^{\frac{1}{2}}\ltr^{3\al-2\ga_0-\frac{3}{2}}\tad\|\sqrt{a}\zet\|^2_{H_{\Psi}^{\frac{23}{3},0}}
+\frac{1}{100}\eta\|\sqrt{a}\py\U\|^2_{H_{\Psi}^{7,0}}\\
&\leq C\eta^{-1}\ep^{\frac{1}{2}}\tad\|\sqrt{a}\zet\|^2_{H_{\Psi}^{\frac{23}{3},0}}
+\frac{1}{100}\eta\|\sqrt{a}\py\U\|^2_{H_{\Psi}^{7,0}},
\end{align*}
here we use the fact that $\al\leq\frac{2}{3}\ga_0+\frac{1}{2}$.

For $F_4$, using \eqref{lk} and direct calculation, we can obtain
\begin{align*}
\left[\frac{1}{\tad}\T_{\py b};\Lk\right]=&-\pt\left(\frac{1}{\tad}\right)\T_{\py b}
-\frac{1}{\tad}\T_{\pt\py b-\ka\py^3 b}+\la\left[\T_{\py b};[D_x]^{\frac{2}{3}}\right]+\frac{1}{\tad}\left[\T_{\py b};\T_u \px\right]\\
&+\frac{1}{\tad}\left[\T_{\py b};\T_v \py\right]+\frac{2}{3}\de(t)\frac{1}{\tad}\left[\T_{\py b};\T_{D_x u}Q(D_x)\px\right]+\frac{2\ka}{\tad}\T_{\py^2 b}\py.
\end{align*}
Thus we have
$$
a\left(\left[\frac{1}{\tad}\T_{\py b};\Lk\right]\int_{y}^{\ty}\U dz,\zet\right)_{H_{\Psi}^{\frac{22}{3},0}}=\sum_{i=1}^7 E_i.
$$

Now start to estimate term by term. For $E_1$, notice that \eqref{theta}, so $\pt\left(\frac{1}{\tad}\right)=\al\ep^{-\frac{1}{2}}\ltr^{\al-1}$. Using lemma \ref{gu1}, lemma \ref{ub} and \eqref{ty}, we can derive
\begin{align*}
\left|E_1\right|&\leq C\al\ep^{-\frac{1}{2}}\ltr^{\al-1}\|\py b\|_{L_{\vm}^2(H_{\h}^{\frac{1}{2}+})}\|\sqrt{a}\int_{y}^{\ty}\U dz\|_{L_{\vm,\Psi}^{\ty}(H_{\h}^{\frac{22}{3}})}\|\sqrt{a}\zet\|_{H_{\Psi}^{\frac{23}{3},0}}\\
&\leq C \ltr^{2\al-\ga_0-\frac{5}{4}}\tad\|\sqrt{a}\U \|_{H_{\Psi}^{\frac{22}{3},0}}\|\sqrt{a}\zet\|_{H_{\Psi}^{\frac{23}{3},0}}\\
&\leq C \tad\|\sqrt{a}\U \|_{H_{\Psi}^{\frac{22}{3},0}}\|\sqrt{a}\zet\|_{H_{\Psi}^{\frac{23}{3},0}}\\
&\leq C \tad\|\sqrt{a}\U \|^2_{H_{\Psi}^{\frac{22}{3},0}}+C \tad\|\sqrt{a}\zet\|^2_{H_{\Psi}^{\frac{23}{3},0}},
\end{align*}
here we use $\al\leq\frac{1}{2}\ga_0+\frac{5}{8}$. For $E_2$, similarly we have
\begin{align*}
\left|E_2\right|&\leq C\ep^{-\frac{1}{2}}\ltr^{\al}\|\pt\py b-\ka\py^3 b\|_{L_{\vm}^2(H_{\h}^{\frac{1}{2}+})}\|\sqrt{a}\int_{y}^{\ty}\U dz\|_{L_{\vm,\Psi}^{\ty}(H_{\h}^{\frac{22}{3}})}\|\sqrt{a}\zet\|_{H_{\Psi}^{\frac{23}{3},0}}\\
&\leq C\ep^{-\frac{1}{2}}\ltr^{\al+\frac{1}{4}}\|\pt\py b-\ka\py^3 b\|_{L_{\vm}^2(H_{\h}^{\frac{1}{2}+})}\|\sqrt{a}\U \|_{H_{\Psi}^{\frac{22}{3},0}}\|\sqrt{a}\zet\|_{H_{\Psi}^{\frac{23}{3},0}}.
\end{align*}
By \eqref{mhd} and \eqref{PN}, we know
\begin{align*}
\pt\py b-\ka\py^3 b=&\py b\px u+\py h\py u+ b\py\px u+h\py^2u\\
&-\py u\px b-\py v\py b-u\py\px b-v\py^2b.
\end{align*}
To handle the product, we need to use the following basic fact of Sobolev space: let $s>\frac{1}{2},$ then
\begin{align}\label{jiben}
\|fg\|_{H_{\h}^s}\leq C\|f\|_{H_{\h}^s}\|g\|_{H_{\h}^s}.
\end{align}
Noticing that $\py v=-\px u,\py h=-\px b$, at this time we can use lemma \ref{ub} to get the following,
\begin{align*}
&~~~~~\|\pt\py b-\ka\py^3 b\|_{L_{\vm}^2(H_{\h}^{\frac{1}{2}+})}\\
&\leq
\|(u,b)\|_{L_{\vm}^{\ty}(H_{\h}^{\frac{1}{2}+})}\|(\py u,\py b)\|_{H^{\frac{3}{2}+,0}}+\|(v,h)\|_{L_{\vm}^{\ty}(H_{\h}^{\frac{1}{2}+})}
\|(\py^2 u,\py^2 b)\|_{H^{\frac{1}{2}+,0}}\\
&~~~~~+\|(\py u,\py b)\|_{L_{\vm}^{\ty}(H_{\h}^{\frac{1}{2}+})}\|( u, b)\|_{H^{\frac{3}{2}+,0}}\\
&\leq C \ep^2\ltr^{-2\ga_0-\frac{3}{4}}.
\end{align*}
Thus $E_2$ satisfies
\begin{align*}
\left|E_2\right|&\leq C\ep^{\frac{3}{2}}\ltr^{\al-2\ga_0-\frac{1}{2}}
\|\sqrt{a}\U \|_{H_{\Psi}^{\frac{22}{3},0}}\|\sqrt{a}\zet\|_{H_{\Psi}^{\frac{23}{3},0}}\\
&\leq C\ep\ltr^{2\al-2\ga_0-\frac{1}{2}}\tad
\|\sqrt{a}\U \|_{H_{\Psi}^{\frac{22}{3},0}}\|\sqrt{a}\zet\|_{H_{\Psi}^{\frac{23}{3},0}}\\
&\leq C\ep\tad
\|\sqrt{a}\U \|_{H_{\Psi}^{\frac{22}{3},0}}\|\sqrt{a}\zet\|_{H_{\Psi}^{\frac{23}{3},0}}\\
&\leq C\ep\tad
\|\sqrt{a}\U \|^2_{H_{\Psi}^{\frac{22}{3},0}}+
C\ep\tad\|\sqrt{a}\zet\|^2_{H_{\Psi}^{\frac{23}{3},0}},
\end{align*}
here we use $\al\leq\ga_0+\frac{1}{4}$. For $E_3$,
\begin{align*}
\left|E_3\right|&\leq C\la\|\py b\|_{L_{\vm}^2(H_{\h}^{\frac{3}{2}+})}\|\sqrt{a}\int_{y}^{\ty}\U dz\|_{L_{\vm,\Psi}^{\ty}(H_{\h}^{\frac{22}{3}})}\|\sqrt{a}\zet\|_{H_{\Psi}^{\frac{23}{3},0}}\\
&\leq C\ep\la\ltr^{-\ga_0-\frac{1}{4}}\|\sqrt{a}\U\|_{H_{\Psi}^{\frac{22}{3},0}}
\|\sqrt{a}\zet\|_{H_{\Psi}^{\frac{23}{3},0}}\\
&\leq C\ep^{\frac{1}{2}}\la\ltr^{\al-\ga_0-\frac{1}{4}}\tad\|\sqrt{a}\U\|_{H_{\Psi}^{\frac{22}{3},0}}
\|\sqrt{a}\zet\|_{H_{\Psi}^{\frac{23}{3},0}}\\
&\leq C\ep^{\frac{1}{2}}\la\tad\|\sqrt{a}\U\|_{H_{\Psi}^{\frac{22}{3},0}}
\|\sqrt{a}\zet\|_{H_{\Psi}^{\frac{23}{3},0}}\\
&\leq C\ep^{\frac{1}{2}}\la\tad\|\sqrt{a}\U\|^2_{H_{\Psi}^{\frac{22}{3},0}}
+C\ep^{\frac{1}{2}}\la\tad\|\sqrt{a}\zet\|^2_{H_{\Psi}^{\frac{23}{3},0}},
\end{align*}
here we also use $\al\leq\ga_0+\frac{1}{4}$. For $E_4$, by
\eqref{theta}, \eqref{ty}, lemma \ref{gu2} and lemma \ref{ub}, it's easy to know
\begin{align*}
\left|E_4\right|&\leq C\ep^{-\frac{1}{2}}\ltr^{\al}\|\py b\|_{L_{\vm}^2(H_{\h}^{\frac{3}{2}+})}\|u\|_{L_{\vm}^{\ty}(H_{\h}^{\frac{3}{2}+})}
\|\sqrt{a}\int_{y}^{\ty}\U dz\|_{L_{\vm,\Psi}^{\ty}(H_{\h}^{\frac{22}{3}})}\|\sqrt{a}\zet\|_{H_{\Psi}^{\frac{23}{3},0}}\\
&\leq C\ep\ltr^{2\al-2\ga_0-\frac{1}{2}}\tad\|\sqrt{a}\U\|_{H_{\Psi}^{\frac{22}{3},0}}
\|\sqrt{a}\zet\|_{H_{\Psi}^{\frac{23}{3},0}}\\
&\leq C\ep\tad\|\sqrt{a}\U\|^2_{H_{\Psi}^{\frac{22}{3},0}}
+C\ep\tad\|\sqrt{a}\zet\|^2_{H_{\Psi}^{\frac{23}{3},0}}.
\end{align*}
For $E_5$, first we notice
\begin{align}\label{jiaohuan}
\left[\T_{\py u};\T_v\py\right]=\left[\T_{\py u};\T_v\right]\py-\T_v\T_{\py^2 u}.
\end{align}
Applying \eqref{jiaohuan}, lemma \ref{gu1}, lemma \ref{gu2} and lemma \ref{ub}, we can get the estimate of $E_5$,
\begin{align*}
\left|E_5\right|&\leq \left|a\frac{1}{\tad}\left(\left[\T_{\py b};\T_v\right]\U,\zet\right)_{H_{\Psi}^{\frac{22}{3},0}}\right|+
\left|a\frac{1}{\tad}\left(\T_v\T_{\py^2 u}\int_y^{\ty}\U dz,\zet\right)_{H_{\Psi}^{\frac{22}{3},0}}\right|\\
&\leq C\ep\ltr^{2\al-2\ga_0-\frac{1}{2}}\tad\|\sqrt{a}\U\|_{H_{\Psi}^{\frac{22}{3},0}}
\|\sqrt{a}\zet\|_{H_{\Psi}^{\frac{23}{3},0}}\\
&\leq C\ep\tad\|\sqrt{a}\U\|^2_{H_{\Psi}^{\frac{22}{3},0}}
+C\ep\tad\|\sqrt{a}\zet\|^2_{H_{\Psi}^{\frac{23}{3},0}}.
\end{align*}
For $E_6$, using a similar calculation as $E_4$, we obtain
\begin{align*}
\left|E_6\right|\leq C\ep\tad\|\sqrt{a}\U\|^2_{H_{\Psi}^{\frac{22}{3},0}}
+C\ep\tad\|\sqrt{a}\zet\|^2_{H_{\Psi}^{\frac{23}{3},0}}.
\end{align*}
For $E_7$, it's obvious
\begin{align*}
\left|E_7\right|&\leq C\ep^{-\frac{1}{2}}\ltr^{\al}\|\py^2 b\|_{L_{\vm}^{\ty}(H_{\h}^{\frac{1}{2}+})} \|\sqrt{a}\U\|_{H_{\Psi}^{\frac{22}{3},0}}
\|\sqrt{a}\zet\|_{H_{\Psi}^{\frac{23}{3},0}}\\
&\leq C\ltr^{2\al-\ga_0-\frac{5}{4}}\tad \|\sqrt{a}\U\|_{H_{\Psi}^{\frac{22}{3},0}}
\|\sqrt{a}\zet\|_{H_{\Psi}^{\frac{23}{3},0}}\\
&\leq C\tad\|\sqrt{a}\U\|^2_{H_{\Psi}^{\frac{22}{3},0}}
+C\tad\|\sqrt{a}\zet\|^2_{H_{\Psi}^{\frac{23}{3},0}},
\end{align*}
here we use the fact that $\al\leq\frac{1}{2}\ga_0+\frac{5}{8}.$

Combining $E_1-E_7$, we can deduce that when $\al\leq\min\{\ga_0+\frac{1}{4},\frac{1}{2}\ga_0+\frac{5}{8}\}$, $F_4$ satisfies
\begin{align*}
\left|a\left(\left[\frac{1}{\tad}\T_{\py b};\Lk\right]\int_{y}^{\ty}\U dz,\zet\right)_{H_{\Psi}^{\frac{22}{3},0}}\right|\leq C(1+\ep^{\frac{1}{2}}\la)\tad\|\sqrt{a}\U\|^2_{H_{\Psi}^{\frac{22}{3},0}}
+C(1+\ep^{\frac{1}{2}}\la)\tad\|\sqrt{a}\zet\|^2_{H_{\Psi}^{\frac{23}{3},0}}.
\end{align*}
For $F_5$, referring to the estimates of $F_3$, we can gain
\begin{align*}
\left| a\left(\frac{2(1-\ka)}{3\tad}\de(t)\T_{\py D_x b}Q(D_x)\py\U,\zet\right)_{H_{\Psi}^{\frac{22}{3},0}}\right|\
\leq C\eta^{-1}\ep^{\frac{1}{2}}\tad\|\sqrt{a}\zet\|^2_{H_{\Psi}^{\frac{23}{3},0}}
+\frac{1}{100}\eta\|\sqrt{a}\py\U\|^2_{H_{\Psi}^{7,0}}.
\end{align*}
For $F_6$, similar to $F_4$, we obtain by direct calculation
\begin{align*}
&\left[\frac{2\de(t)}{3\tad}\T_{\py D_x b}Q(D_x);\Lk\right]\\
=
&-\frac{2}{3}\pt\left(\frac{\de(t)}{\tad}\right)\T_{\py D_x b}Q(D_x)
-\frac{2\de(t)}{3\tad}\T_{\pt\py D_x b-\ka\py^3 D_x b}Q(D_x)\\
&+\frac{2\de(t)}{3}\la\left[\T_{\py D_x b}Q(D_x);[D_x]^{\frac{2}{3}}\right]+\frac{2\de(t)}{3\tad}\left[\T_{\py D_x b}Q(D_x);\T_u \px\right]\\
&+\frac{2\de(t)}{3\tad}\left[\T_{\py D_x b}Q(D_x);\T_v \py\right]+\frac{4\de^2(t)}{9\tad}\left[\T_{\py D_x b}Q(D_x);\T_{D_x u}Q(D_x)\px\right]\\
&+\frac{4\ka\de(t)}{3\tad}\T_{\py^2 D_x b}Q(D_x)\py.
\end{align*}
Noticing \eqref{phi} and \eqref{theta}, we derive
$$0<\de(t)\leq\de<C,$$
$$\pt\left(\frac{\de(t)}{\tad}\right)=-\la+
\de(t)\al\ep^{-\frac{1}{2}}\ltr^{\al-1}.$$
So similar to the estimates of $F_4$, we can obtain the estimates of $F_6$
\begin{align*}
&~~~~~\left|a\left(\left[\frac{2\de(t)}{3\tad}\T_{\py D_x b}Q(D_x);\Lk\right]\int_{y}^{\ty}\U dz,\zet\right)_{H_{\Psi}^{\frac{22}{3},0}}\right|\\
&\leq C(1+\ep^{\frac{1}{2}}\la)\tad\|\sqrt{a}\U\|^2_{H_{\Psi}^{\frac{22}{3},0}}
+C(1+\ep^{\frac{1}{2}}\la)\tad\|\sqrt{a}\zet\|^2_{H_{\Psi}^{\frac{23}{3},0}}.
\end{align*}

Now start computing the left side of \eqref{zetzet}, similar to \eqref{ptpy}, we can get
\begin{align}\label{ptpylk}
\begin{split}
\left(\pt \zet-\ka\py^2\zet,\zet\right)_{H_{\Psi}^{\frac{22}{3},0}}
\geq\frac{1}{2}\dfrac{d}{dt}\|\zet(t)\|^2_{H_{\Psi}^{\frac{22}{3},0}}
+2l_{\ka}\|\py\zet(t)\|^2_{H_{\Psi}^{\frac{22}{3},0}}.
\end{split}
\end{align}
Then
\begin{align}\label{zetzetlk}
\begin{split}
a(\Lk\zet,\zet)_{H_{\Psi}^{\frac{22}{3},0}}&\geq\frac{1}{2}\dfrac{d}{dt}
\|\sqrt{a}\zet(t)\|^2_{H_{\Psi}^{\frac{22}{3},0}}-\frac{1}{2}\|\sqrt{a'}\zet(t)\|^2_{H_{\Psi}^{\frac{22}{3},0}}
+\la\tad\|\sqrt{a}\zet(t)\|^2_{H_{\Psi}^{\frac{22}{3},0}}\\
&+2l_{\ka}\|\sqrt{a}\py\zet(t)\|^2_{H_{\Psi}^{\frac{22}{3},0}}+
a\left(\T_u\px\zet,\zet\right)_{H_{\Psi}^{\frac{22}{3},0}}+a\left(\T_v\py\zet,\zet\right)_{H_{\Psi}^{\frac{22}{3},0}}\\
&+a\left(\frac{2}{3}\de(t)
\T_{D_x u}Q(D_x)\px\zet,\zet\right)_{H_{\Psi}^{\frac{22}{3},0}}.
\end{split}
\end{align}
Using the same method of estimates as $a\left(\Ll\U,\U\right)_{H_{\Psi}^{7,0}}$, we can deduce
$$\left|a\left(\T_u\px\zet,\zet\right)_{H_{\Psi}^{\frac{22}{3},0}}\right|\leq
C\ep^{\frac{1}{2}}\tad\|\sqrt{a}\zet\|^2_{H_{\Psi}^{\frac{23}{3},0}},$$
$$\left|a\left(\T_v\py\zet,\zet\right)_{H_{\Psi}^{\frac{22}{3},0}}\right|\leq
C\eta^{-1}\ep^{\frac{3}{2}}\tad\|\sqrt{a}\zet\|^2_{H_{\Psi}^{\frac{23}{3},0}}
+\frac{1}{100}\eta\|\sqrt{a}\py\zet\|^2_{H_{\Psi}^{\frac{22}{3},0}},$$
$$\left|a\left(\frac{2}{3}\de(t)
\T_{D_x u}Q(D_x)\px\zet,\zet\right)_{H_{\Psi}^{\frac{22}{3},0}}\right|\leq
C\ep^{\frac{1}{2}}\tad\|\sqrt{a}\zet\|^2_{H_{\Psi}^{\frac{23}{3},0}}.$$
Combining all the above estimates and integrating over $[0,t]$, noting that \eqref{zetchubianzhi}, we can obtain the following estimates
\begin{align}\label{turanhaomafang}
\begin{aligned}
&~~~~~\|\sqrt{a}\zet(t)\|^2_{H_{\Psi}^{\frac{22}{3},0}}-\int_0^t \|\sqrt{a'}\zet\|^2_{H_{\Psi}^{\frac{22}{3},0}}ds+(4l_{\ka}-\frac{1}{50}\eta)\int_0^t \|\sqrt{a}\py\zet\|^2_{H_{\Psi}^{\frac{22}{3},0}}ds\\
&~~~~~
+2\left(\la-C(1+\la\ep^{\frac{1}{2}}+\eta^{-1}\ep^{\frac{1}{2}})\right)\int_0^t
\tad\|\sqrt{a}\zet\|^2_{H_{\Psi}^{\frac{23}{3},0}}ds\\
&\leq
\|\sqrt{a}(0)\bp(0)\|^2_{H_{\Psi}^{\frac{22}{3},0}}
+\frac{1}{25}\eta\int_0^t
\left(\|\sqrt{a}\py\U\|^2_{H_{\Psi}^{7,0}}
+\|\sqrt{a}(\py\uph,\py\bp)\|^2_{H_{\Psi}^{7,0}}\right)ds
\\
&~~~~~+C(1+\ep^{\frac{1}{2}}\la)\int_0^t\tad
\left(\|\sqrt{a}\U\|^2_{H_{\Psi}^{\frac{22}{3},0}}+\|\sqrt{a}\Np\|^2_{H_{\Psi}^{7,0}}
+\|\sqrt{a}(\uph,\bp)\|^2_{H_{\Psi}^{\frac{22}{3},0}}\right)ds.
\end{aligned}
\end{align}
Finally we come back to give the estimates of
$\|\sqrt{a}(\py\uph,\py\bp)\|^2_{H_{\Psi}^{7,0}}$
and $\|\sqrt{a}(\uph,\bp)\|^2_{H_{\Psi}^{\frac{22}{3},0}}$.

By \eqref{ze}, we have
\begin{align}\label{upbp}
\begin{split}
\uph=\ze+\frac{1}{\tad}\T_{\py u}\int_y^{\ty}\U dz
+\frac{2\de(t)}{3\tad}\T_{\py D_x u}Q(D_x)\int_y^{\ty}\U dz,
\\
\bp=\zet+\frac{1}{\tad}\T_{\py b}\int_y^{\ty}\U dz
+\frac{2\de(t)}{3\tad}\T_{\py D_x b}Q(D_x)\int_y^{\ty}\U dz.
\end{split}
\end{align}
Using \eqref{theta}, \eqref{ty}, lemma \ref{gu1} and lemma \ref{ub}, it yields
\begin{align}\label{Uub}
\begin{aligned}
&\|\frac{1}{\tad}\T_{\py u}\int_y^{\ty}\U dz\|_{H_{\Psi}^{\frac{22}{3},0}}\leq
C\ep^{\frac{1}{2}}\ltr^{\al-\ga_0-\frac{1}{4}}\|\U\|_{H_{\Psi}^{\frac{22}{3},0}}\leq
C\ep^{\frac{1}{2}}\|\U\|_{H_{\Psi}^{\frac{22}{3},0}},\\
&\|\frac{2\de(t)}{3\tad}\T_{\py D_x u}Q(D_x)\int_y^{\ty}\U dz\|_{H_{\Psi}^{\frac{22}{3},0}}\leq
C\ep^{\frac{1}{2}}\ltr^{\al-\ga_0-\frac{1}{4}}\|\U\|_{H_{\Psi}^{\frac{22}{3},0}}\leq
C\ep^{\frac{1}{2}}\|\U\|_{H_{\Psi}^{\frac{22}{3},0}},\\
&\|\frac{1}{\tad}\T_{\py b}\int_y^{\ty}\U dz\|_{H_{\Psi}^{\frac{22}{3},0}}\leq
C\ep^{\frac{1}{2}}\ltr^{\al-\ga_0-\frac{1}{4}}\|\U\|_{H_{\Psi}^{\frac{22}{3},0}}\leq
C\ep^{\frac{1}{2}}\|\U\|_{H_{\Psi}^{\frac{22}{3},0}},\\
&\|\frac{2\de(t)}{3\tad}\T_{\py D_x b}Q(D_x)\int_y^{\ty}\U dz\|_{H_{\Psi}^{\frac{22}{3},0}}\leq
C\ep^{\frac{1}{2}}\ltr^{\al-\ga_0-\frac{1}{4}}\|\U\|_{H_{\Psi}^{\frac{22}{3},0}}\leq
C\ep^{\frac{1}{2}}\|\U\|_{H_{\Psi}^{\frac{22}{3},0}}.
\end{aligned}
\end{align}
Combing \eqref{upbp} and \eqref{Uub}, we know
\begin{align}\label{ubdekongzhi}
\begin{split}
\|\sqrt{a}(\uph,\bp)\|_{H_{\Psi}^{\frac{22}{3},0}}&\leq \|\sqrt{a}(\ze,\zet)\|_{H_{\Psi}^{\frac{22}{3},0}}
+C\ep^{\frac{1}{2}}\|\sqrt{a}\U\|_{H_{\Psi}^{\frac{22}{3},0}}\\
&\leq \|\sqrt{a}(\ze,\zet)\|_{H_{\Psi}^{\frac{23}{3},0}}
+C\ep^{\frac{1}{2}}\|\sqrt{a}\U\|_{H_{\Psi}^{\frac{22}{3},0}}.
\end{split}
\end{align}
With $\py$ acting on \eqref{upbp}, we can get
\begin{align}\label{pyupbp}
\begin{split}
\py\uph&=\py\ze+\frac{1}{\tad}\T_{\py^2 u}\int_y^{\ty}\U dz
+\frac{2\de(t)}{3\tad}\T_{\py^2 D_x u}Q(D_x)\int_y^{\ty}\U dz\\
&-\frac{1}{\tad}\T_{\py u}\U
-\frac{2\de(t)}{3\tad}\T_{\py D_x u}Q(D_x)\U,
\\
\py\bp&=\py\zet+\frac{1}{\tad}\T_{\py^2 b}\int_y^{\ty}\U dz
+\frac{2\de(t)}{3\tad}\T_{\py^2 D_x b}Q(D_x)\int_y^{\ty}\U dz\\
&-\frac{1}{\tad}\T_{\py b}\U
-\frac{2\de(t)}{3\tad}\T_{\py D_x b}Q(D_x)\U.
\end{split}
\end{align}
Similarly by \eqref{theta}, \eqref{ty}, lemma \ref{gu1} and
lemma \ref{ub}, we can deduce
\begin{align}\label{pyubdekongzhi}
\begin{split}
\|\sqrt{a}(\py\uph,\py\bp)\|_{H_{\Psi}^{7,0}}\leq \|\sqrt{a}(\py\ze,\py\zet)\|_{H_{\Psi}^{\frac{22}{3},0}}
+C\ep^{\frac{1}{2}}\ltr^{-\frac{1}{2}}\|\sqrt{a}\U\|_{H_{\Psi}^{7,0}}.
\end{split}
\end{align}
At this time using lemma \ref{gu6}, we can obtain
\begin{align}\label{pyubdekongzhi2}
\begin{split}
\|\sqrt{a}(\py\uph,\py\bp)\|_{H_{\Psi}^{7,0}}\leq \|\sqrt{a}(\py\ze,\py\zet)\|_{H_{\Psi}^{\frac{22}{3},0}}
+C\ep^{\frac{1}{2}}\|\sqrt{a}\py\U\|_{H_{\Psi}^{7,0}}.
\end{split}
\end{align}

Next, choose $\ep_*$ sufficiently small so that when $\ep<\ep_*$, $C\ep^{\frac{1}{2}}<1$ holds. At this point, combining \eqref{turanhaomafang}, \eqref{ubdekongzhi} and \eqref{pyubdekongzhi2},
we can get
\begin{align}\label{turanhaomafang2}
\begin{aligned}
&~~~~~\|\sqrt{a}\zet(t)\|^2_{H_{\Psi}^{\frac{22}{3},0}}-\int_0^t \|\sqrt{a'}\zet\|^2_{H_{\Psi}^{\frac{22}{3},0}}ds+(4l_{\ka}-\frac{1}{50}\eta)\int_0^t \|\sqrt{a}\py\zet\|^2_{H_{\Psi}^{\frac{22}{3},0}}ds\\
&~~~~~
+2\left(\la-C(1+\la\ep^{\frac{1}{2}}+\eta^{-1}\ep^{\frac{1}{2}})\right)\int_0^t
\tad\|\sqrt{a}\zet\|^2_{H_{\Psi}^{\frac{23}{3},0}}ds\\
&\leq
\|\sqrt{a}(0)(\uph(0),\bp(0))\|^2_{H_{\Psi}^{\frac{22}{3},0}}
+\frac{2}{25}\eta\int_0^t
\left(\|\sqrt{a}\py\U\|^2_{H_{\Psi}^{7,0}}
+\|\sqrt{a}(\py\ze,\py\zet)\|^2_{H_{\Psi}^{\frac{22}{3},0}}\right)ds
\\
&~~~~~+C(1+\ep^{\frac{1}{2}}\la)\int_0^t\tad
\left(\|\sqrt{a}\U\|^2_{H_{\Psi}^{\frac{22}{3},0}}+\|\sqrt{a}(\Pp,\Np)\|^2_{H_{\Psi}^{7,0}}
+\|\sqrt{a}(\ze,\zet)\|^2_{H_{\Psi}^{\frac{23}{3},0}}\right)ds.
\end{aligned}
\end{align}
This completes the estimates of $\zet$, and the estimates of $\ze$ follow in the same way. After sorting out the coefficients, we complete the proof of proposition \ref{guze}.
\end{proof}

Applying proposition \ref{guze}, we can prove proposition \ref{propositionze}.

\textbf{Proof of Proposition \ref{propositionze}:}
Take $a(t)=\ltr^{2l_{\ka}-2\eta}$ in \eqref{zedekongzhi}. Then by lemma \ref{gu6} we can get
\begin{align*}
&-(2l_{\ka}-2\eta)\int_0^t \|\lsr^{l_{\ka}-\eta-\frac{1}{2}}
(\ze,\zet)\|^2_{H_{\Psi}^{\frac{22}{3},0}}ds
+(4l_{\ka}-\frac{1}{10}\eta)\int_0^t \|\lsr^{l_{\ka}-\eta}(\py\ze,\py\zet)\|^2_{H_{\Psi}^{\frac{22}{3},0}}ds\\
\geq&-(4l_{\ka}-4\eta)\int_0^t \|\lsr^{l_{\ka}-\eta}(\py\ze,\py\zet)\|^2_{H_{\Psi}^{\frac{22}{3},0}}ds
+(4l_{\ka}-\frac{1}{10}\eta)\int_0^t \|\lsr^{l_{\ka}-\eta}(\py\ze,\py\zet)\|^2_{H_{\Psi}^{\frac{22}{3},0}}ds\\
\geq&\eta\int_0^t \|\lsr^{l_{\ka}-\eta}(\py\ze,\py\zet)\|^2_{H_{\Psi}^{\frac{22}{3},0}}ds,
\end{align*}
thus we have
\begin{align*}
&~~~~~\|\ltr^{l_{\ka}-\eta}(\ze(t),\zet(t))\|^2_{H_{\Psi}^{\frac{22}{3},0}}
+\eta\int_0^t \|\lsr^{l_{\ka}-\eta}(\py\ze,\py\zet)\|^2_{H_{\Psi}^{\frac{22}{3},0}}ds\\
&~~~~~
+2\left(\la-C(1+\la\ep^{\frac{1}{2}}+\eta^{-1}\ep^{\frac{1}{2}})\right)\int_0^t
\tad\|\lsr^{l_{\ka}-\eta}(\ze,\zet)\|^2_{H_{\Psi}^{\frac{23}{3},0}}ds\\
&\leq
\|(\uph(0),\bp(0))\|^2_{H_{\Psi}^{\frac{22}{3},0}}
+\frac{2}{25}\eta\int_0^t
\|\lsr^{l_{\ka}-\eta}\py\U\|^2_{H_{\Psi}^{7,0}}
ds
\\
&~~~~~+C(1+\ep^{\frac{1}{2}}\la)\int_0^t\tad
\left(\|\lsr^{l_{\ka}-\eta}
\U\|^2_{H_{\Psi}^{\frac{22}{3},0}}+\|\lsr^{l_{\ka}-\eta}
(\Pp,\Np)\|^2_{H_{\Psi}^{7,0}}
\right)ds.
\end{align*}
Taking $\ep_2$ sufficiently small, such that when $\ep<\ep_2<\min\{\ep_*,\ep_1\}$, we have $1+\eta^{-1}\ep^{\frac{1}{2}}<2,C\ep^{\frac{1}{2}}<\frac{1}{10}. $
Let $\la_2$ be sufficiently large, and $\la$ satisfies
$\la>\la_2>\max\{10C,\la_1\}. $ So the proposition \ref{propositionze} holds.

\hfill $\square$
\section{The Gevrey Estimates of $\uph$ and $\bp$}\label{gujiub}
In this section we will give the estimates of $(\uph,\bp)$. In fact with the help of proposition \ref{guU} and proposition \ref{guze}, we can directly obtain the estimates of $(\uph,\bp)$ with the following results:
\begin{Proposition}\label{ubxianyanguji}
Let $\ka\in(0,2)$ be a given constant. Let $(u,b)$ be a sufficiently smooth solution of the equation \eqref{mhd} on $[0,T_*]$ and $(u,b)$ decays rapidly to $0$ as $y$ tends to $+\ty$. Let $a(t)$ be a non-negative and non-decreasing function on $\R_+$. Then when $\al\leq\min\{\ga_0+\frac{1}{4},\frac{2}{3}\ga_0+\frac{1}{2},\frac{1}{2}\ga_0+\frac{5}{8}\}$ and $\ep<\ep_*$, for any $t<T_*$ and sufficiently small $\eta >0$ the following inequality holds,
\begin{align}\label{ubxianyanguji2}
\begin{aligned}
&~~~~~\|\sqrt{a}(\uph(t),\bp(t))\|^2_{H_{\Psi}^{7,0}}
-2\int_0^t\left(\|\sqrt{a'}\U\|^2_{H_{\Psi}^{7,0}}+
\|\sqrt{a'}(\ze,\zet)\|^2_{H_{\Psi}^{\frac{22}{3},0}}\right)ds\\
&~~~~~+2(4l_{\ka}-\frac{1}{10}\eta)\int_0^t \left(\|\sqrt{a}\py\U\|^2_{H_{\Psi}^{7,0}}+\|\sqrt{a}(\py\ze,\py\zet)\|^2_{H_{\Psi}^{\frac{22}{3},0}}\right)ds\\
&~~~~~
+2\left(\la-C(1+\la\ep^{\frac{1}{2}}+\eta^{-1}\ep^{\frac{1}{2}})\right)\int_0^t
\tad\|\sqrt{a}(\uph,\bp)\|^2_{H_{\Psi}^{\frac{22}{3},0}}ds\\
&\leq
2\|\sqrt{a}(0)(\uph(0),\bp(0))\|^2_{H_{\Psi}^{\frac{22}{3},0}}
+C(1+\ep^{\frac{1}{2}}\la)\int_0^t\tad
\|\sqrt{a}(\Pp,\Np)\|^2_{H_{\Psi}^{7,0}}
ds,
\end{aligned}
\end{align}
where $l_{\ka}=\frac{\ka(2-\ka)}{4}\in(0,\frac{1}{4}].$
\end{Proposition}
\begin{proof}
By \eqref{ubdekongzhi}, when $\ep<\ep_*$, we have
\begin{align}\label{zhendewuliaoa}
\begin{split}
\|\sqrt{a}(\uph,\bp)\|_{H_{\Psi}^{\frac{22}{3},0}}
\leq \|\sqrt{a}(\ze,\zet)\|_{H_{\Psi}^{\frac{23}{3},0}}
+\|\sqrt{a}\U\|_{H_{\Psi}^{\frac{22}{3},0}}.
\end{split}
\end{align}
Similarly,
\begin{align}\label{wuliaoa}
\begin{split}
\|\sqrt{a}(\uph,\bp)\|_{H_{\Psi}^{7,0}}
\leq \|\sqrt{a}(\ze,\zet)\|_{H_{\Psi}^{\frac{22}{3},0}}
+\|\sqrt{a}\U\|_{H_{\Psi}^{7,0}}.
\end{split}
\end{align}
It's easy to know that when $0\leq s\leq q+t$, $s^2\leq 2q^2+2t^2$ holds.
Combining \eqref{Udekongzhi},
\eqref{zedekongzhi}, \eqref{zhendewuliaoa} and \eqref{wuliaoa}, we can complete the proof of proposition \ref{ubxianyanguji}.
\end{proof}

Now we can prove proposition \ref{propositionub} by using proposition \ref{ubxianyanguji}.

\textbf{Proof of Proposition \ref{propositionub}:}
Take $a(t)=\ltr^{2l_{\ka}-2\eta}$ in \eqref{ubxianyanguji2}. By
lemma \ref{gu6}, we can derive
\begin{align*}
&-2(2l_{\ka}-2\eta)\int_0^t
\left(\|\lsr^{l_{\ka}-\eta-\frac{1}{2}}\U\|^2_{H_{\Psi}^{7,0}}+
\|\lsr^{l_{\ka}-\eta-\frac{1}{2}}
(\ze,\zet)\|^2_{H_{\Psi}^{\frac{22}{3},0}}\right)ds\\
&+2(4l_{\ka}-\frac{1}{10}\eta)\int_0^t \left(\|\lsr^{l_{\ka}-\eta}\py\U\|^2_{H_{\Psi}^{7,0}}
+\|\lsr^{l_{\ka}-\eta}
(\py\ze,\py\zet)\|^2_{H_{\Psi}^{\frac{22}{3},0}}\right)ds\\
\geq&
-4(2l_{\ka}-2\eta)\int_0^t \left(\|\lsr^{l_{\ka}-\eta}\py\U\|^2_{H_{\Psi}^{7,0}}
+\|\lsr^{l_{\ka}-\eta}
(\py\ze,\py\zet)\|^2_{H_{\Psi}^{\frac{22}{3},0}}\right)ds\\
&+2(4l_{\ka}-\frac{1}{10}\eta)\int_0^t \left(\|\lsr^{l_{\ka}-\eta}\py\U\|^2_{H_{\Psi}^{7,0}}
+\|\lsr^{l_{\ka}-\eta}
(\py\ze,\py\zet)\|^2_{H_{\Psi}^{\frac{22}{3},0}}\right)ds\\
\geq&2\eta\int_0^t \left(\|\lsr^{l_{\ka}-\eta}\py\U\|^2_{H_{\Psi}^{7,0}}
+\|\lsr^{l_{\ka}-\eta}
(\py\ze,\py\zet)\|^2_{H_{\Psi}^{\frac{22}{3},0}}\right)ds\\
\geq&\eta\int_0^t \|\lsr^{l_{\ka}-\eta}
(\py\uph,\py\bp)\|^2_{H_{\Psi}^{7,0}}ds
\end{align*}
here we use \eqref{pyubdekongzhi2} and $\ep<\ep_2<\ep_*.$
Thus when $\ep<\ep_2$ and $\la>\la_2$, we have
\begin{align*}
&\|\ltr^{l_{\ka}-\eta}(\uph(t),\bp(t))\|^2_{H_{\Psi}^{7,0}}
+\la\int_0^t
\tad\|\lsr^{l_{\ka}-\eta}(\uph,\bp)\|^2_{H_{\Psi}^{\frac{22}{3},0}}ds\\
&+\eta\int_0^t \|\lsr^{l_{\ka}-\eta}
(\py\uph,\py\bp)\|^2_{H_{\Psi}^{7,0}}ds\\
\leq&
2\|(\uph(0),\bp(0))\|^2_{H_{\Psi}^{\frac{22}{3},0}}
+\frac{1}{5}\la\int_0^t\tad
\|\lsr^{l_{\ka}-\eta}(\Pp,\Np)\|^2_{H_{\Psi}^{7,0}}
ds,
\end{align*}
so proposition \ref{propositionub} holds.
\hfill $\square$
\section{The Gevrey Estimates of $\Pp$ and $\Np$}\label{lalalademaxiya}
In this section we give a priori estimates for $(\Pp,\Np)$. First, from \eqref{cb},\eqref{mhd} and \eqref{PN} we can obtain that $(P,N)$ satisfies the following equation
\begin{align}\label{PNdefangcheng}
\left\{\begin{array}{l}
\left(\partial_{t}+u \partial_{x}+v \partial_{y}- \partial_{y}^{2}\right) N=-\frac{\ddot{\ta}}{\tad}N+b\px P+h\py P
+H
\\
\left(\partial_{t}+u \partial_{x}+v \partial_{y}- \ka \partial_{y}^{2}\right) P=-\frac{\ddot{\ta}}{\tad}P+b\px N+h\py N
+\tilde{H}.
\end{array}\right.
\end{align}
where $(H,\tilde{H})$ is defined as follows
\begin{align}\label{HH}
\begin{split}
&H=\frac{2}{\tad}(\px b\py^2 u-\py b\py\px u)
+\frac{\ka-1}{\tad}(\px u\py^2 b-\py u\py\px b),\\
&\tilde{H}=\frac{2\ka}{\tad}(\px b\py^2 b-\py b\py\px b).
\end{split}
\end{align}
Next, we use $e^{\Phi(t,D_x)}$ to act on \eqref{PNdefangcheng}, and do Bony's decomposition of all nonlinear terms using \eqref{bony}. Similar to \eqref{mhdp}, we write the equation in the following form:
\begin{align}\label{PNp}
\left\{\begin{array}{l}
\Ll \Np=-\frac{\ddot{\theta}}{\tad}\Np+\T_{b}\px\Pp
+\frac{2}{3}\de(t)\T_{D_x b}Q(D_x)\px\Pp\\
~~~~~~~~~~~~~~~~~~~~~~~~~~~~~~~~~~~~~~~~+\T_{\py P}\hp+\T_{h}\py\Pp
-\T_{\py N}\vp-S+\Hp
\\
\Lk \Pp=-\frac{\ddot{\theta}}{\tad}\Pp+\T_{b}\px\Np
+\frac{2}{3}\de(t)\T_{D_x b}Q(D_x)\px\Np\\
~~~~~~~~~~~~~~~~~~~~~~~~~~~~~~~~~~~~~~~~+\T_{\py N}\hp+\T_{h}\py\Np
-\T_{\py P}\vp-\tilde{S}+\Htp
\\
 (\Pp,\Np)|_{y=0}=(0,0),~(\Pp,\Np)|_{y\rightarrow +\ty}=(0,0),
\\
(\Pp,\Np)|_{t=0}=\left(\ep^{-\frac{1}{2}}e^{\de[D_x]^{\frac{2}{3}}}(b_0\px b_0+h_0\py b_0),
\ep^{-\frac{1}{2}}e^{\de[D_x]^{\frac{2}{3}}}(b_0\px u_0+h_0\py u_0)\right).
\end{array}\right.
\end{align}
where $(S,\tilde{S})$ is defined as
\begin{align}\label{S}
\begin{split}
S=&\left(T_{u}^{\mathrm{h}} \partial_{x} N\right)_{\Phi}-T_{u}^{\mathrm{h}} \partial_{x} N_{\Phi}-\frac{2}{3}\delta(t) T_{D_{x} u}^{\mathrm{h}} Q(D_{x}) \partial_{x} N_{\Phi}\\
&+\left(T_{\py N}^{\mathrm{h}} v\right)_{\Phi}-T_{\py N}^{\mathrm{h}} v_{\Phi}+\left(\T_v \py N\right)_{\Phi}-\T_v\py\Np\\
&+\left(\T_{\px N}u+R^{\mathrm{h}}(u,\px N)+R^{\mathrm{h}}(v,\py N)\right)_{\Phi}\\
&-\left(T_{\px P} b+R^{\h}(b,\px P)+R^{\h}(h,\py P)\right)_{\Phi}\\
&-(\T_b \px P)_{\Phi}+\T_{b}\px\Pp+\frac{2}{3}\de(t)\T_{D_x b}Q(D_x)\px\Pp\\
&-(\T_{\py P} h)_{\Phi}+\T_{\py P}\hp-(\T_{h} \py P)_{\Phi}+\T_{h}\py\Pp
,
\end{split}
\end{align}
\begin{align}\label{St}
\begin{split}
\tilde{S}=&\left(T_{u}^{\mathrm{h}} \partial_{x} P\right)_{\Phi}-T_{u}^{\mathrm{h}} \partial_{x} P_{\Phi}-\frac{2}{3}\delta(t) T_{D_{x} u}^{\mathrm{h}} Q(D_{x}) \partial_{x} P_{\Phi}\\
&+\left(T_{\py P}^{\mathrm{h}} v\right)_{\Phi}-T_{\py P}^{\mathrm{h}} v_{\Phi}+\left(\T_v \py P\right)_{\Phi}-\T_v\py\Pp\\
&+\left(\T_{\px P}u+R^{\mathrm{h}}(u,\px P)+R^{\mathrm{h}}(v,\py P)\right)_{\Phi}\\
&-\left(T_{\px N} b+R^{\h}(b,\px N)+R^{\h}(h,\py N)\right)_{\Phi}\\
&-(\T_b \px N)_{\Phi}+\T_{b}\px\Np+\frac{2}{3}\de(t)\T_{D_x b}Q(D_x)\px\Np\\
&-(\T_{\py N} h)_{\Phi}+\T_{\py N}\hp-(\T_{h} \py N)_{\Phi}+\T_{h}\py\Np.
\end{split}
\end{align}
First, we give the estimates of $(S,\tilde{S})$.
\begin{Lemma}\label{Sdekongzhi}
Let $(S,\tilde{S})$ be given by \eqref{S} and \eqref{St}, respectively, and $\ka\in(0,2)$. Then for any $s>0,$ we have:
\begin{align*}
\|(S,\tilde{S})\|_{H_{\Psi}^{s,0}}\leq
&C \ep\ltr^{-\ga_0-\frac{1}{4}}\|(\Pp,\Np)\|_{H_{\Psi}^{s+\frac{1}{3},0}}
 +C\ep^{\frac{3}{2}}\ltr^{\al-2\ga_0-\frac{1}{2}}
\|(\uph,\bp)\|_{H_{\Psi}^{s+\frac{2}{3},0}}\\
&+ C \ep\ltr^{-\ga_0+\frac{1}{4}}\|(\py\Pp,\py\Np)\|_{H_{\Psi}^{s,0}}
.
\end{align*}
\end{Lemma}
\begin{proof}
First, by lemma \ref{gu5} and lemma \ref{ub}, we get
\begin{align*}
\left\|\left(T_{u}^{\mathrm{h}} \partial_{x} N\right)_{\Phi}-T_{u}^{\mathrm{h}} \partial_{x} N_{\Phi}-\frac{2}{3}\delta(t) T_{D_{x} u}^{\mathrm{h}} Q(D_{x}) \partial_{x} N_{\Phi}\right\|_{H_{\Psi}^{s,0}}
&\leq C
\|\uph\|_{L_{\vm}^{\ty}(H_{\h}^{\frac{5}{2}+})}\|\Np\|_{H_{\Psi}^{s+\frac{1}{3},0}}
\\
&\leq C \ep\ltr^{-\ga_0-\frac{1}{4}}\|\Np\|_{H_{\Psi}^{s+\frac{1}{3},0}}.
\end{align*}
Using \ref{gu4}, it yields
\begin{align*}
\left\|\left(T_{\py N}^{\mathrm{h}} v\right)_{\Phi}-T_{\py N}^{\mathrm{h}} v_{\Phi}\right\|_{H_{\Psi}^{s,0}}&\leq C
\|\py \Np\|_{L_{\vm}^{2}(H_{\h}^{\frac{3}{2}+})}
\|\vp\|_{L_{\vm,\Psi}^{\ty}(H_{\h}^{s-\frac{1}{3}})}
\\
&\leq C\ltr^{\frac{1}{4}}\|\py \Np\|_{L_{\vm}^{2}(H_{\h}^{\frac{3}{2}+})}\|\uph\|_{H_{\Psi}^{s+\frac{2}{3},0}}.
\end{align*}
And applying \eqref{PN} and \eqref{theta}, we have
\begin{align*}
\py \Np&=\ep^{-\frac{1}{2}}\ltr^{\al}(\py b\px u+b\py\px u+\py h\py u +h \py^2 u)_{\Phi}\\
&=\ep^{-\frac{1}{2}}\ltr^{\al}(\py b\px u+b\py\px u-\px b\py u +h \py^2 u)_{\Phi}
\end{align*}
Thus by \eqref{bony} and lemma \ref{ub}, we get
\begin{align}\label{pyNdexiaoxing}
\begin{aligned}
\|\py \Np\|_{L_{\vm}^{2}(H_{\h}^{\frac{3}{2}+})}
&=\ep^{-\frac{1}{2}}\ltr^{\al}\|(\py b\px u+b\py\px u+\py h\py u +h \py^2 u)_{\Phi}\|_{L_{\vm}^{2}(H_{\h}^{\frac{3}{2}+})}\\
&\leq \ep^{-\frac{1}{2}}\ltr^{\al} C \ep^2\ltr^{-2\ga_0-\frac{3}{4}}\\
&= C \ep^{\frac{3}{2}}\ltr^{\al-2\ga_0-\frac{3}{4}}.
\end{aligned}
\end{align}
Then we can deduce that
\begin{align*}
\left\|\left(T_{\py N}^{\mathrm{h}} v\right)_{\Phi}-T_{\py N}^{\mathrm{h}} v_{\Phi}\right\|_{H_{\Psi}^{s,0}}\leq  C\ep^{\frac{3}{2}}\ltr^{\al-2\ga_0-\frac{1}{2}}
\|\uph\|_{H_{\Psi}^{s+\frac{2}{3},0}}.
\end{align*}
Next, using lemma \ref{gu4} and lemma \ref{ub}, it gives
\begin{align*}
\left\|\left(\T_v \py N\right)_{\Phi}-\T_v\py\Np\right\|_{H_{\Psi}^{s,0}}&\leq C \|\vp\|_{L_{\vm}^{\ty}(H_{\h}^{\frac{3}{2}+})}\|\py \Np\|_{H_{\Psi}^{s-\frac{1}{3},0}}\\
&\leq C \ep\ltr^{-\ga_0+\frac{1}{4}}\|\py \Np\|_{H_{\Psi}^{s-\frac{1}{3},0}}\\
&\leq C \ep\ltr^{-\ga_0+\frac{1}{4}}\|\py \Np\|_{H_{\Psi}^{s,0}}.
\end{align*}
Similarly, by lemma \ref{gu1} and lemma \ref{ub}, we can obtain
\begin{align*}
\left\|R^{\h}(v,\py N)_{\Psi}\right\|_{H_{\Psi}^{s,0}}&\leq C \|\vp\|_{L_{\vm}^{\ty}(H_{\h}^{\frac{1}{2}+})}\|\py \Np\|_{H_{\Psi}^{s,0}}\\
&\leq C \ep\ltr^{-\ga_0+\frac{1}{4}}\|\py \Np\|_{H_{\Psi}^{s,0}}.
\end{align*}
We can derive from lemma \ref{gu1} and lemma \ref{ub}
\begin{align*}
\left\|\left(\T_{\px N}u+R^{\mathrm{h}}(u,\px N)\right)_{\Phi}\right\|_{H_{\Psi}^{s,0}}\leq
C \|\Np\|_{L_{\vm}^{\ty}(H_{\h}^{\frac{3}{2}+})}
\|\uph\|_{H_{\Psi}^{s,0}},
\end{align*}
Combining \eqref{PN}, \eqref{bony}, \eqref{theta} and lemma \ref{ub}, we find
\begin{align}\label{Ndexiaoxing}
\begin{aligned}
\|\Np\|_{L_{\vm}^{\ty}(H_{\h}^{\frac{3}{2}+})}
&=\ep^{-\frac{1}{2}}\ltr^{\al}
\|(b\px u+h\py u)_{\Phi}\|_{L_{\vm}^{\ty}(H_{\h}^{\frac{3}{2}+})}\\
&\leq C \ep^{\frac{3}{2}}\ltr^{\al-2\ga_0-\frac{1}{2}},
\end{aligned}
\end{align}
Thus we can get
\begin{align*}
\left\|\left(\T_{\px N}u+R^{\mathrm{h}}(u,\px N)\right)_{\Phi}\right\|_{H_{\Psi}^{s,0}}\leq
C \ep^{\frac{3}{2}}\ltr^{\al-2\ga_0-\frac{1}{2}}
\|\uph\|_{H_{\Psi}^{s,0}}.
\end{align*}
Looking closely at \eqref{S}, we can find that the remaining terms are similar, so we can get the estimates of $S$.
\begin{align*}
\|S\|_{H_{\Psi}^{s,0}}\leq
&C \ep\ltr^{-\ga_0-\frac{1}{4}}\|(\Pp,\Np)\|_{H_{\Psi}^{s+\frac{1}{3},0}}
 +C\ep^{\frac{3}{2}}\ltr^{\al-2\ga_0-\frac{1}{2}}
\|(\uph,\bp)\|_{H_{\Psi}^{s+\frac{2}{3},0}}\\
&+ C \ep\ltr^{-\ga_0+\frac{1}{4}}\|(\py\Pp,\py\Np)\|_{H_{\Psi}^{s,0}}
.
\end{align*}
On the other hand the estimates of $\tilde{S}$ are similar, so we complete the proof of lemma \ref{Sdekongzhi}.
\end{proof}
Now we give the a priori estimates of $(\Pp,\Np)$, and the exact results are shown below.
\begin{Proposition}\label{PNdexianyanguji}
Let $\ka\in(0,2)$ be a given constant. Let $(P,N)$ be defined by \eqref{PN} and $a(t)$ be a non-negative and non-decreasing function on $\R_+$. Then when $\al\leq\min\{\ga_0+\frac{1}{4},\frac{2}{3}\ga_0+\frac{1}{2},\frac{1}{2}\ga_0+\frac{5}{8}\}$  and $\ep<\ep_*$, for any $t<T_*$ and sufficiently small $\eta >0$ the following equation holds:
\begin{align}\label{PNdexianyanguji2}
\begin{aligned}
&~~~~~\|\sqrt{a}(\Pp(t),\Np(t))\|^2_{H_{\Psi}^{\frac{20}{3},0}}
-\int_0^t\|\sqrt{a'}(\Pp,\Np)\|^2_{H_{\Psi}^{\frac{20}{3},0}}ds\\
&~~~~~+(4l_{\ka}-\frac{1}{10}\eta-C\ep^{\frac{1}{2}})\int_0^t \|\sqrt{a}(\py\Pp,\py\Np)\|^2_{H_{\Psi}^{\frac{20}{3},0}}ds\\
&~~~~~
+2\left(\la-C(1+\la\ep^{\frac{1}{2}}+\eta^{-1}\ep^{\frac{1}{2}})\right)\int_0^t
\tad\|\sqrt{a}(\Pp,\Np)\|^2_{H_{\Psi}^{7,0}}ds\\
&\leq
\|\sqrt{a}(0)(\Pp(0),\Np(0))\|^2_{H_{\Psi}^{\frac{20}{3},0}}
+C\int_0^t\tad
\left(\|\sqrt{a}\U\|^2_{H_{\Psi}^{\frac{22}{3},0}}
+\|\sqrt{a}(\uph,\bp)\|^2_{H_{\Psi}^{\frac{22}{3},0}}\right)
ds\\
&~~~~~+(\frac{3}{25}\eta+C\ep^{\frac{1}{2}})\int_0^t
\left(\|\sqrt{a}\py\U\|^2_{H_{\Psi}^{7,0}}
+\|\sqrt{a}(\py\ze,\py\zet)\|^2_{H_{\Psi}^{\frac{22}{3},0}}\right)ds,
\end{aligned}
\end{align}
where $l_{\ka}=\frac{\ka(2-\ka)}{4}\in(0,\frac{1}{4}].$
\end{Proposition}
\begin{proof}
Take $H_{\Psi}^{\frac{20}{3},0}$ inner product of $\eqref{PNp}_{1}$ and $\eqref{PNp}_{2}$ with $a(t)\Np$ and $a(t)\Pp$ respectively,
and add them together.
\begin{align}\label{PNzuoneiji}
\begin{aligned}
&a(\Ll \Np,\Np)_{H_{\Psi}^{\frac{20}{3},0}}
+a(\Lk \Pp,\Pp)_{H_{\Psi}^{\frac{20}{3},0}}\\
=&-a(\frac{\ddot{\theta}}{\tad}\Np,\Np)_{H_{\Psi}^{\frac{20}{3},0}}
-a(\frac{\ddot{\theta}}{\tad}\Pp,\Pp)_{H_{\Psi}^{\frac{20}{3},0}}
+a(\T_{b}\px\Pp,\Np)_{H_{\Psi}^{\frac{20}{3},0}}
+a(\T_{b}\px\Np,\Pp)_{H_{\Psi}^{\frac{20}{3},0}}\\
&+a(\frac{2}{3}\de(t)\T_{D_x b}Q(D_x)\px\Pp,\Np)_{H_{\Psi}^{\frac{20}{3},0}}
+a(\frac{2}{3}\de(t)\T_{D_x b}Q(D_x)\px\Np,\Pp)_{H_{\Psi}^{\frac{20}{3},0}}\\
&+a(\T_{\py P}\hp,\Np)_{H_{\Psi}^{\frac{20}{3},0}}
+a(\T_{\py N}\hp,\Pp)_{H_{\Psi}^{\frac{20}{3},0}}\\
&+a(\T_{h}\py\Pp,\Np)_{H_{\Psi}^{\frac{20}{3},0}}
+a(\T_{h}\py\Np,\Pp)_{H_{\Psi}^{\frac{20}{3},0}}\\
&-a(\T_{\py N}\vp,\Np)_{H_{\Psi}^{\frac{20}{3},0}}
-a(\T_{\py P}\vp,\Pp)_{H_{\Psi}^{\frac{20}{3},0}}\\
&-a(S,\Np)_{H_{\Psi}^{\frac{20}{3},0}}
-a(\tilde{S},\Pp)_{H_{\Psi}^{\frac{20}{3},0}}
+a(\Hp,\Np)_{H_{\Psi}^{\frac{20}{3},0}}
+a(\Htp,\Pp)_{H_{\Psi}^{\frac{20}{3},0}}\\
=&\sum_{i=1}^{16} M_i.
\end{aligned}
\end{align}

For $M_1$, by \eqref{PN}, \eqref{bony} and \eqref{theta}, we have
\begin{align*}
M_1&=\al\ep^{-\frac{1}{2}}\ltr^{\al-1}a\left(
(b\px u+h\py u)_{\Phi},\Np\right)_{H_{\Psi}^{\frac{20}{3},0}}\\
&=\al\ep^{-\frac{1}{2}}\ltr^{\al-1}a\left(
(\T_{b}\px u)_{\Phi}+(\T_{\px u}b)_{\Phi}+R^{\h}(b,\px u)_{\Phi}\right.\\
&~~~~~\left.+(\T_{h}\py u)_{\Phi}+(\T_{\py u}h)_{\Phi}+R^{\h}(h,\py u)_{\Phi},\Np\right)_{H_{\Psi}^{\frac{20}{3},0}}\\
&=\sum_{i=1}^{6}K_i.
\end{align*}
For $K_1$, $K_2$ and $K_3$, it's obvious
\begin{align*}
|K_1+K_2+K_3|&\leq C\al\ep^{-\frac{1}{2}}\ltr^{\al-1}
\|(\uph,\bp)\|_{L^{\ty}_{\vm}(H_{\h}^{\frac{3}{2}+})}
\|\sqrt{a}(\uph,\bp)\|_{H_{\Psi}^{\frac{22}{3},0}}
\|\sqrt{a}\Np\|_{H_{\Psi}^{7,0}}\\
&\leq C\ltr^{2\al-\ga_0-\frac{5}{4}}\tad
\|\sqrt{a}(\uph,\bp)\|_{H_{\Psi}^{\frac{22}{3},0}}
\|\sqrt{a}\Np\|_{H_{\Psi}^{7,0}}\\
&\leq C\tad
\|\sqrt{a}(\uph,\bp)\|^2_{H_{\Psi}^{\frac{22}{3},0}}
+C\tad
\|\sqrt{a}\Np\|^2_{H_{\Psi}^{7,0}}.
\end{align*}
Similarly, for $K_5$ and $K_6$, we have
\begin{align*}
|K_5+K_6|&\leq C\al\ep^{-\frac{1}{2}}\ltr^{\al-1}
\|\py \uph\|_{L^{2}_{\vm}(H_{\h}^{\frac{1}{2}+})}
\|\sqrt{a}\hp\|_{L^{\ty}_{\vm}(H_{\h}^{\frac{19}{3}})}
\|\sqrt{a}\Np\|_{H_{\Psi}^{7,0}}\\
&\leq C\ltr^{2\al-\ga_0-\frac{5}{4}}\tad
\|\sqrt{a}\bp\|_{H_{\Psi}^{\frac{22}{3},0}}
\|\sqrt{a}\Np\|_{H_{\Psi}^{7,0}}\\
&\leq C\tad
\|\sqrt{a}\bp\|^2_{H_{\Psi}^{\frac{22}{3},0}}
+C\tad
\|\sqrt{a}\Np\|^2_{H_{\Psi}^{7,0}}.
\end{align*}
For $K_4$, when $\ep<\ep_*$, $C\ep^{\frac{1}{2}}<1$ holds. Thus applying \eqref{pyubdekongzhi2}, we get
\begin{align}\label{zaizheliyongyong}
\begin{split}
\|\sqrt{a}\py\uph\|_{H_{\Psi}^{7,0}}\leq \|\sqrt{a}(\py\ze,\py\zet)\|_{H_{\Psi}^{\frac{22}{3},0}}
+\|\sqrt{a}\py\U\|_{H_{\Psi}^{7,0}}.
\end{split}
\end{align}
Then we deduce
\begin{align*}
|K_4|&\leq C\al\ep^{-\frac{1}{2}}\ltr^{\al-1}
\|\hp\|_{L^{\ty}_{\vm}(H_{\h}^{\frac{1}{2}+})}
\|\sqrt{a}\py\uph\|_{H_{\Psi}^{7,0}}
\|\sqrt{a}\Np\|_{H_{\Psi}^{7,0}}\\
&\leq C\ep^{\frac{1}{2}}\ltr^{\al-\ga_0-\frac{3}{4}}
\left(\|\sqrt{a}(\py\ze,\py\zet)\|_{H_{\Psi}^{\frac{22}{3},0}}
+\|\sqrt{a}\py\U\|_{H_{\Psi}^{7,0}}\right)
\|\sqrt{a}\Np\|_{H_{\Psi}^{7,0}}\\
&\leq \frac{1}{100}\eta\|\sqrt{a}(\py\ze,\py\zet)\|^2_{H_{\Psi}^{\frac{22}{3},0}}
+\frac{1}{100}\eta\|\sqrt{a}\py\U\|^2_{H_{\Psi}^{7,0}}
+C\eta^{-1}\ep^{\frac{1}{2}}\ltr^{3\al-2\ga_0-\frac{3}{2}}\tad
\|\sqrt{a}\Np\|^2_{H_{\Psi}^{7,0}}\\
&\leq \frac{1}{100}\eta\|\sqrt{a}(\py\ze,\py\zet)\|^2_{H_{\Psi}^{\frac{22}{3},0}}
+\frac{1}{100}\eta\|\sqrt{a}\py\U\|^2_{H_{\Psi}^{7,0}}
+C\eta^{-1}\ep^{\frac{1}{2}}\tad
\|\sqrt{a}\Np\|^2_{H_{\Psi}^{7,0}}.
\end{align*}
Combining $K_1-K_6$, we can obtain
\begin{align*}
|M_1|\leq &C\tad
\|\sqrt{a}(\uph,\bp)\|^2_{H_{\Psi}^{\frac{22}{3},0}}
+C(1+\eta^{-1}\ep^{\frac{1}{2}})\tad
\|\sqrt{a}\Np\|^2_{H_{\Psi}^{7,0}}\\
&+\frac{1}{100}\eta\|\sqrt{a}(\py\ze,\py\zet)\|^2_{H_{\Psi}^{\frac{22}{3},0}}
+\frac{1}{100}\eta\|\sqrt{a}\py\U\|^2_{H_{\Psi}^{7,0}}.
\end{align*}
For $M_2$, it's similar that
\begin{align*}
|M_2|\leq &C\tad
\|\sqrt{a}(\uph,\bp)\|^2_{H_{\Psi}^{\frac{22}{3},0}}
+C(1+\eta^{-1}\ep^{\frac{1}{2}})\tad
\|\sqrt{a}\Pp\|^2_{H_{\Psi}^{7,0}}\\
&+\frac{1}{100}\eta\|\sqrt{a}(\py\ze,\py\zet)\|^2_{H_{\Psi}^{\frac{22}{3},0}}
+\frac{1}{100}\eta\|\sqrt{a}\py\U\|^2_{H_{\Psi}^{7,0}}.
\end{align*}
For $M_3$ and $M_4$, using lemma \ref{gu3}, we get
\begin{align*}
|M_3+M_4|&\leq C\|b\|_{L^{\ty}_{\vm}(H_{\h}^{\frac{3}{2}+})}
\|\sqrt{a}\Pp\|_{H_{\Psi}^{\frac{20}{3},0}}
\|\sqrt{a}\Np\|_{H_{\Psi}^{\frac{20}{3},0}}\\
&\leq C\ep\ltr^{-\ga_0-\frac{1}{4}}
\|\sqrt{a}\Pp\|_{H_{\Psi}^{\frac{20}{3},0}}
\|\sqrt{a}\Np\|_{H_{\Psi}^{\frac{20}{3},0}}\\
&\leq C\ep^{\frac{1}{2}}\ltr^{\al-\ga_0-\frac{1}{4}}
\tad\|\sqrt{a}(\Pp,\Np)\|^2_{H_{\Psi}^{7,0}}\\
&\leq C\ep^{\frac{1}{2}}
\tad\|\sqrt{a}(\Pp,\Np)\|^2_{H_{\Psi}^{7,0}},
\end{align*}
here we use the fact that $\al\leq\ga_0+\frac{1}{4}.$ For $M_5$ and $M_6$, we can deduce along the same line
\begin{align*}
|M_5+M_6|\leq C\ep^{\frac{1}{2}}
\tad\|\sqrt{a}(\Pp,\Np)\|^2_{H_{\Psi}^{7,0}}.
\end{align*}
For $M_7$, first it's easy to prove that \eqref{pyNdexiaoxing} also holds for $\py P$, then we have
\begin{align*}
|M_7|&\leq C\|\py P\|_{L^{2}_{\vm}(H_{\h}^{\frac{1}{2}+})}
\|\sqrt{a}\hp\|_{L^{\ty}_{\vm,\Psi}(H_{\h}^{\frac{19}{3}})}
\|\sqrt{a}\Np\|_{H_{\Psi}^{7,0}}\\
&\leq C \ep^{\frac{3}{2}}\ltr^{\al-2\ga_0-\frac{3}{4}}
\ltr^{\frac{1}{4}}\|\sqrt{a}\bp\|_{H_{\Psi}^{\frac{22}{3},0}}
\|\sqrt{a}\Np\|_{H_{\Psi}^{7,0}}\\
&\leq C \ep\ltr^{2\al-2\ga_0-\frac{1}{2}}
\tad\|\sqrt{a}\bp\|_{H_{\Psi}^{\frac{22}{3},0}}
\|\sqrt{a}\Np\|_{H_{\Psi}^{7,0}}\\
&\leq C \ep\tad\|\sqrt{a}\bp\|^2_{H_{\Psi}^{\frac{22}{3},0}}
+C \ep\tad\|\sqrt{a}\Np\|^2_{H_{\Psi}^{7,0}}.
\end{align*}
For $M_8$, $M_{11}$ and $M_{12}$, from a similar derivation we can find
\begin{align*}
|M_8+M_{11}+M_{12}|
\leq C \ep\tad\|\sqrt{a}(\uph,\bp)\|^2_{H_{\Psi}^{\frac{22}{3},0}}
+C \ep\tad\|\sqrt{a}(\Pp,\Np)\|^2_{H_{\Psi}^{7,0}}.
\end{align*}
For $M_9$, we have
\begin{align*}
|M_9|
&\leq C \|h\|_{L^{\ty}_{\vm}(H_{\h}^{\frac{1}{2}+})}
\|\sqrt{a}\py\Pp\|_{H_{\Psi}^{\frac{20}{3},0}}
\|\sqrt{a}\Np\|_{H_{\Psi}^{7,0}}\\
&\leq C\ep\ltr^{-\ga_0+\frac{1}{4}}
\|\sqrt{a}\py\Pp\|_{H_{\Psi}^{\frac{20}{3},0}}
\|\sqrt{a}\Np\|_{H_{\Psi}^{7,0}}\\
&\leq C\eta^{-1}\ep^{2}\ltr^{-2\ga_0+\frac{1}{2}}
\|\sqrt{a}\Np\|^2_{H_{\Psi}^{7,0}}
+\frac{1}{100}\eta\|\sqrt{a}\py\Pp\|^2_{H_{\Psi}^{\frac{20}{3},0}}\\
&\leq C\eta^{-1}\ep^{\frac{3}{2}}\ltr^{\al-2\ga_0+\frac{1}{2}}
\tad\|\sqrt{a}\Np\|^2_{H_{\Psi}^{7,0}}
+\frac{1}{100}\eta\|\sqrt{a}\py\Pp\|^2_{H_{\Psi}^{\frac{20}{3},0}}\\
&\leq C\eta^{-1}\ep^{\frac{3}{2}}
\tad\|\sqrt{a}\Np\|^2_{H_{\Psi}^{7,0}}
+\frac{1}{100}\eta\|\sqrt{a}\py\Pp\|^2_{H_{\Psi}^{\frac{20}{3},0}},
\end{align*}
here we use $\al\leq\ga_0+\frac{1}{4}$ and $\ga_0>1.$
For $M_{10}$, it's similar that
\begin{align*}
|M_{10}|
\leq C\eta^{-1}\ep^{\frac{3}{2}}
\tad\|\sqrt{a}\Pp\|^2_{H_{\Psi}^{7,0}}
+\frac{1}{100}\eta\|\sqrt{a}\py\Np\|^2_{H_{\Psi}^{\frac{20}{3},0}}.
\end{align*}
For $M_{13},$ by lemma \ref{Sdekongzhi}, we can know
\begin{align*}
|M_{13}|&\leq
\|\sqrt{a}S\|_{H_{\Psi}^{\frac{19}{3},0}}\|\sqrt{a}\Np\|_{H_{\Psi}^{7,0}}\\
&\leq C
 \left(\ep\ltr^{-\ga_0-\frac{1}{4}}\|\sqrt{a}(\Pp,\Np)\|_{H_{\Psi}^{\frac{20}{3},0}}
 +\ep^{\frac{3}{2}}\ltr^{\al-2\ga_0-\frac{1}{2}}
\|\sqrt{a}(\uph,\bp)\|_{H_{\Psi}^{7,0}}\right.\\
&~~~~~\left.+ \ep\ltr^{-\ga_0+\frac{1}{4}}\|\sqrt{a}(\py\Pp,\py\Np)\|_{H_{\Psi}^{\frac{19}{3},0}}\right)
\|\sqrt{a}\Np\|_{H_{\Psi}^{7,0}}\\
&\leq C\ep^{\frac{1}{2}}\ltr^{\al-\ga_0-\frac{1}{4}}\tad\|\sqrt{a}(\Pp,\Np)\|^2_{H_{\Psi}^{7,0}}
+C\ep\ltr^{2\al-2\ga_0-\frac{1}{2}}\tad\|\sqrt{a}(\uph,\bp)\|_{H_{\Psi}^{\frac{22}{3},0}}
\|\sqrt{a}\Np\|_{H_{\Psi}^{7,0}}\\
&~~~~~+C\ep\ltr^{-\ga_0+\frac{1}{4}}\|\sqrt{a}(\py\Pp,\py\Np)\|_{H_{\Psi}^{\frac{20}{3},0}}
\|\sqrt{a}\Np\|_{H_{\Psi}^{7,0}}\\
&\leq C\ep^{\frac{1}{2}}\tad\|\sqrt{a}(\Pp,\Np)\|^2_{H_{\Psi}^{7,0}}
+C\ep\tad\|\sqrt{a}(\uph,\bp)\|^2_{H_{\Psi}^{\frac{22}{3},0}}
+C\ep\tad\|\sqrt{a}\Np\|^2_{H_{\Psi}^{7,0}}\\
&~~~~~+\frac{1}{100}\eta\|\sqrt{a}(\py\Pp,\py\Np)\|^2_{H_{\Psi}^{\frac{20}{3},0}}
+C\eta^{-1}\ep^{\frac{3}{2}}\ltr^{\al-2\ga_0+\frac{1}{2}}\tad\|\sqrt{a}\Np\|^2_{H_{\Psi}^{7,0}}\\
&\leq C(\ep^{\frac{1}{2}}+\eta^{-1}\ep^{\frac{3}{2}})\tad\|\sqrt{a}(\Pp,\Np)\|^2_{H_{\Psi}^{7,0}}
+C\ep\tad\|\sqrt{a}(\uph,\bp)\|^2_{H_{\Psi}^{\frac{22}{3},0}}\\
&~~~~~+\frac{1}{100}\eta\|\sqrt{a}(\py\Pp,\py\Np)\|^2_{H_{\Psi}^{\frac{20}{3},0}}.
\end{align*}
For $M_{14}$, along the same line we have
\begin{align*}
|M_{14}|
&\leq C(\ep^{\frac{1}{2}}+\eta^{-1}\ep^{\frac{3}{2}})\tad\|\sqrt{a}(\Pp,\Np)\|^2_{H_{\Psi}^{7,0}}
+C\ep\tad\|\sqrt{a}(\uph,\bp)\|^2_{H_{\Psi}^{\frac{22}{3},0}}\\
&~~~~~+\frac{1}{100}\eta\|\sqrt{a}(\py\Pp,\py\Np)\|^2_{H_{\Psi}^{\frac{20}{3},0}}.
\end{align*}
For $M_{15}$ and $M_{16}$, they can be estimated by similar derivations. We only give the estimate procedure for $M_{16}$ here. Notice that
\begin{align*}
\tilde{H}=\frac{2\ka}{\tad}(\px b\py^2 b-\py b\py\px b)
=2\ka\ep^{-\frac{1}{2}}\ltr^{\al}(\px b\py^2 b-\py b\py\px b),
\end{align*}
thus
\begin{align*}
\Htp=2\ka\ep^{-\frac{1}{2}}\ltr^{\al}(\px b\py^2 b)_{\Phi}-
2\ka\ep^{-\frac{1}{2}}\ltr^{\al}(\py b\py\px b)_{\Phi}.
\end{align*}
By \eqref{bony}, we have
\begin{align*}
(\py b\py\px b)_{\Phi}
=\left(\T_{\py b}\py\px b+\T_{\py\px b}\py b+R^{\h}(\py b,\py\px b)\right)_{\Phi}.
\end{align*}
Therefore
\begin{align*}
&\left|\left(2\ka\ep^{-\frac{1}{2}}\ltr^{\al}a(\py b\py\px b)_{\Phi},\Pp\right)_{H_{\Psi}^{\frac{20}{3},0}}\right|\\
=&\left|2\ka\ep^{-\frac{1}{2}}\ltr^{\al}a\left((\T_{\py b}\py\px b)_{\Phi}+(\T_{\py\px b}\py b)_{\Phi}+R^{\h}(\py b,\py\px b)_{\Phi},\Pp\right)_{H_{\Psi}^{\frac{20}{3},0}}\right|\\
\leq&2\ka\ep^{-\frac{1}{2}}\ltr^{\al}a
\|(\T_{\py b}\py\px b)_{\Phi}+(\T_{\py\px b}\py b)_{\Phi}+R^{\h}(\py b,\py\px b)_{\Phi}\|_{H_{\Psi}^{\frac{19}{3},0}}
\|\Pp\|_{H_{\Psi}^{7,0}}\\
\leq&2\ka\ep^{-\frac{1}{2}}\ltr^{\al}
\|\py \bp\|_{L^{\ty}_{\vm}(H_{\h}^{\frac{3}{2}+})}
\|\sqrt{a}\py \bp\|_{H_{\Psi}^{\frac{22}{3},0}}
\|\sqrt{a}\Pp\|_{H_{\Psi}^{7,0}}\\
\leq&C\ep^{\frac{1}{2}}\ltr^{\al-\ga_0-\frac{3}{4}}
\|\sqrt{a}\py \bp\|_{H_{\Psi}^{\frac{22}{3},0}}
\|\sqrt{a}\Pp\|_{H_{\Psi}^{7,0}}.
\end{align*}
Next using \eqref{pyupbp}, similar to \eqref{pyubdekongzhi}, we can get
\begin{align*}
\|\sqrt{a}\py \bp\|_{H_{\Psi}^{\frac{22}{3},0}}\leq
\|\sqrt{a}(\py\ze,\py\zet)\|_{H_{\Psi}^{\frac{22}{3},0}}
+C\ep^{\frac{1}{2}}\ltr^{-\frac{1}{2}}\|\sqrt{a}\U\|_{H_{\Psi}^{\frac{22}{3},0}}.
\end{align*}
Let $\ep<\ep_*$, then $C\ep^{\frac{1}{2}}<1$, thus we can get
\begin{align*}
\|\sqrt{a}\py \bp\|_{H_{\Psi}^{\frac{22}{3},0}}\leq
\|\sqrt{a}(\py\ze,\py\zet)\|_{H_{\Psi}^{\frac{22}{3},0}}
+\ltr^{-\frac{1}{2}}\|\sqrt{a}\U\|_{H_{\Psi}^{\frac{22}{3},0}}.
\end{align*}
So
\begin{align*}
&\left|\left(2\ka\ep^{-\frac{1}{2}}a(\py b\py\px b)_{\Phi},\Pp\right)_{H_{\Psi}^{\frac{20}{3},0}}\right|\\
\leq&C\ep^{\frac{1}{2}}\ltr^{\al-\ga_0-\frac{3}{4}}
\|\sqrt{a}(\py\ze,\py\zet)\|_{H_{\Psi}^{\frac{22}{3},0}}
\|\sqrt{a}\Pp\|_{H_{\Psi}^{7,0}}
+C\ep^{\frac{1}{2}}\ltr^{\al-\ga_0-\frac{5}{4}}
\|\sqrt{a}\U\|_{H_{\Psi}^{\frac{22}{3},0}}
\|\sqrt{a}\Pp\|_{H_{\Psi}^{7,0}}\\
\leq&
\frac{1}{100}\eta\|\sqrt{a}(\py\ze,\py\zet)\|^2_{H_{\Psi}^{\frac{22}{3},0}}
+
C\eta^{-1}\ep^{\frac{1}{2}}\ltr^{3\al-2\ga_0-\frac{3}{2}}
\tad\|\sqrt{a}\Pp\|^2_{H_{\Psi}^{7,0}}\\
&+
C\ltr^{2\al-\ga_0-\frac{5}{4}}
\tad\|\sqrt{a}\U\|^2_{H_{\Psi}^{\frac{22}{3},0}}
+
C\ltr^{2\al-\ga_0-\frac{5}{4}}
\tad\|\sqrt{a}\Pp\|^2_{H_{\Psi}^{7,0}}\\
\leq&
\frac{1}{100}\eta\|\sqrt{a}(\py\ze,\py\zet)\|^2_{H_{\Psi}^{\frac{22}{3},0}}
+C(1+\eta^{-1}\ep^{\frac{1}{2}})
\tad\|\sqrt{a}\Pp\|^2_{H_{\Psi}^{7,0}}
+C\tad\|\sqrt{a}\U\|^2_{H_{\Psi}^{\frac{22}{3},0}},
\end{align*}
here we use the fact that $\al\leq\min\{\frac{2}{3}\ga_0+\frac{1}{2},\frac{1}{2}\ga_0+\frac{5}{8}\}$.
By \eqref{bony}, we have
\begin{align*}
2\ka\ep^{-\frac{1}{2}}\ltr^{\al}(\px b\py^2 b)_{\Phi}
=2\ka\ep^{-\frac{1}{2}}\ltr^{\al}\left((\T_{\px b}\py^2 b)_{\Phi}
+(\T_{\py^2 b}\px b)_{\Phi}+R^{\h}(\px b,\py^2 b)_{\Phi}\right).
\end{align*}
Applying lemma \ref{gu1} and lemma \ref{ub}, we can derive
\begin{align*}
&\left|\left(2\ka\ep^{-\frac{1}{2}}\ltr^{\al}a\left(
(\T_{\py^2 b}\px b)_{\Phi}+R^{\h}(\px b,\py^2 b)_{\Phi}\right),\Pp\right)_{H_{\Psi}^{\frac{20}{3},0}}\right|\\
\leq&C \ep^{-\frac{1}{2}}\ltr^{\al}
\|\py^2 \bp\|_{L^{\ty}_{\vm}(H_{\h}^{\frac{1}{2}+})}
\|\sqrt{a}\bp\|_{H_{\Psi}^{\frac{22}{3},0}}
\|\sqrt{a}\Pp\|_{H_{\Psi}^{7,0}}\\
\leq&C\ltr^{2\al-\ga_0-\frac{5}{4}}\tad
\|\sqrt{a}\bp\|_{H_{\Psi}^{\frac{22}{3},0}}
\|\sqrt{a}\Pp\|_{H_{\Psi}^{7,0}}\\
\leq&C\tad
\|\sqrt{a}\bp\|^2_{H_{\Psi}^{\frac{22}{3},0}}
+C\tad\|\sqrt{a}\Pp\|^2_{H_{\Psi}^{7,0}}.
\end{align*}
Noticing the boundary condition $\Pp|_{y=0}=0$, applying integrating by parts, we can gain
\begin{align*}
&\left|\left(2\ka\ep^{-\frac{1}{2}}\ltr^{\al}a(\T_{\px b}\py^2 b)_{\Phi},\Pp\right)_{H_{\Psi}^{\frac{20}{3},0}}\right|\\
\leq&C\ep^{-\frac{1}{2}}\ltr^{\al}a \left|\left((\T_{\py\px b}\py b)_{\Phi},\Pp\right)_{H_{\Psi}^{\frac{20}{3},0}}\right|
+C\ep^{-\frac{1}{2}}\ltr^{\al}a \left|\left((\T_{\px b}\py b)_{\Phi},\py\Pp\right)_{H_{\Psi}^{\frac{20}{3},0}}\right|\\
&+C\ep^{-\frac{1}{2}}\ltr^{\al}a \left|\left((\T_{\px b}\py b)_{\Phi},\frac{y}{2\ltr}\Pp\right)_{H_{\Psi}^{\frac{20}{3},0}}\right|\\
=&I_1+I_2+I_3.
\end{align*}
For $I_1$, by \eqref{zaizheliyongyong}, we know
\begin{align*}
I_1&\leq C\ep^{-\frac{1}{2}}\ltr^{\al}
\|\py \bp\|_{L^{\ty}_{\vm}(H_{\h}^{\frac{3}{2}+})}
\|\sqrt{a}\py \bp\|_{H_{\Psi}^{7,0}}
\|\sqrt{a}\Pp\|_{H_{\Psi}^{7,0}}\\
&\leq C\ep^{\frac{1}{2}}\ltr^{\al-\ga_0-\frac{3}{4}}
\|\sqrt{a}(\py\ze,\py\zet)\|_{H_{\Psi}^{\frac{22}{3},0}}
\|\sqrt{a}\Pp\|_{H_{\Psi}^{7,0}}\\
&~~~~~+C\ep^{\frac{1}{2}}\ltr^{\al-\ga_0-\frac{3}{4}}
\|\sqrt{a}\py\U\|_{H_{\Psi}^{7,0}}
\|\sqrt{a}\Pp\|_{H_{\Psi}^{7,0}}\\
&\leq\frac{1}{100}\eta
\|\sqrt{a}(\py\ze,\py\zet)\|^2_{H_{\Psi}^{\frac{22}{3},0}}
+\frac{1}{100}\eta
\|\sqrt{a}\py\U\|^2_{H_{\Psi}^{7,0}}
+C\eta^{-1}\ep^{\frac{1}{2}}\ltr^{3\al-2\ga_0-\frac{3}{2}}\tad
\|\sqrt{a}\Pp\|^2_{H_{\Psi}^{7,0}}\\
&\leq\frac{1}{100}\eta
\|\sqrt{a}(\py\ze,\py\zet)\|^2_{H_{\Psi}^{\frac{22}{3},0}}
+\frac{1}{100}\eta\|\sqrt{a}\py\U\|^2_{H_{\Psi}^{7,0}}
+C\eta^{-1}\ep^{\frac{1}{2}}\tad
\|\sqrt{a}\Pp\|^2_{H_{\Psi}^{7,0}}.
\end{align*}
For $I_2$, from a similar derivation, we obtain
\begin{align*}
I_2&\leq C\ep^{-\frac{1}{2}}\ltr^{\al}
\|\bp\|_{L^{\ty}_{\vm}(H_{\h}^{\frac{3}{2}+})}
\|\sqrt{a}\py \bp\|_{H_{\Psi}^{7,0}}
\|\sqrt{a}\py\Pp\|_{H_{\Psi}^{\frac{20}{3},0}}\\
&\leq C\ep^{\frac{1}{2}}\ltr^{\al-\ga_0-\frac{1}{4}}
\|\sqrt{a}(\py\ze,\py\zet)\|_{H_{\Psi}^{\frac{22}{3},0}}
\|\sqrt{a}\py\Pp\|_{H_{\Psi}^{\frac{20}{3},0}}\\
&~~~~~+C\ep^{\frac{1}{2}}\ltr^{\al-\ga_0-\frac{1}{4}}
\|\sqrt{a}\py\U\|_{H_{\Psi}^{7,0}}
\|\sqrt{a}\py\Pp\|_{H_{\Psi}^{\frac{20}{3},0}}\\
&\leq C\ep^{\frac{1}{2}}
\|\sqrt{a}(\py\ze,\py\zet)\|^2_{H_{\Psi}^{\frac{22}{3},0}}
+C\ep^{\frac{1}{2}}
\|\sqrt{a}\py\U\|^2_{H_{\Psi}^{7,0}}
+C\ep^{\frac{1}{2}}
\|\sqrt{a}\py\Pp\|^2_{H_{\Psi}^{\frac{20}{3},0}}.
\end{align*}
For $I_3,$ it's similar,
\begin{align*}
I_3&\leq C\ep^{\frac{1}{2}}
\|\sqrt{a}(\py\ze,\py\zet)\|^2_{H_{\Psi}^{\frac{22}{3},0}}
+C\ep^{\frac{1}{2}}
\|\sqrt{a}\py\U\|^2_{H_{\Psi}^{7,0}}
+C\ep^{\frac{1}{2}}
\|\sqrt{a}\frac{y}{2\ltr}\Pp\|^2_{H_{\Psi}^{\frac{20}{3},0}}.
\end{align*}
On the other hand, by lemma \ref{gu6}, we have
\begin{align*}
I_3&\leq C\ep^{\frac{1}{2}}
\|\sqrt{a}(\py\ze,\py\zet)\|^2_{H_{\Psi}^{\frac{22}{3},0}}
+C\ep^{\frac{1}{2}}
\|\sqrt{a}\py\U\|^2_{H_{\Psi}^{7,0}}
+C\ep^{\frac{1}{2}}
\|\sqrt{a}\py\Pp\|^2_{H_{\Psi}^{\frac{20}{3},0}}.
\end{align*}
Combining the above estimates, we can obtain
\begin{align*}
|M_{16}|\leq&
C(1+\eta^{-1}\ep^{\frac{1}{2}})
\tad\|\sqrt{a}\Pp\|^2_{H_{\Psi}^{7,0}}
+C\tad\|\sqrt{a}\U\|^2_{H_{\Psi}^{\frac{22}{3},0}}
+C\ep^{\frac{1}{2}}\tad
\|\sqrt{a}\bp\|^2_{H_{\Psi}^{\frac{22}{3},0}}\\
&+(\frac{1}{50}\eta+C\ep^{\frac{1}{2}})
\|\sqrt{a}(\py\ze,\py\zet)\|^2_{H_{\Psi}^{\frac{22}{3},0}}
+(\frac{1}{100}\eta+C\ep^{\frac{1}{2}})
\|\sqrt{a}\py\U\|^2_{H_{\Psi}^{7,0}}\\
&+C\ep^{\frac{1}{2}}
\|\sqrt{a}\py\Pp\|^2_{H_{\Psi}^{\frac{20}{3},0}}.
\end{align*}
For $M_{15}$, it's similar to $M_{16}$.
\begin{align*}
|M_{15}|\leq&
C(1+\eta^{-1}\ep^{\frac{1}{2}})
\tad\|\sqrt{a}\Np\|^2_{H_{\Psi}^{7,0}}
+C\tad\|\sqrt{a}\U\|^2_{H_{\Psi}^{\frac{22}{3},0}}
+C\ep^{\frac{1}{2}}\tad
\|\sqrt{a}(\uph,\bp)\|^2_{H_{\Psi}^{\frac{22}{3},0}}\\
&+(\frac{1}{50}\eta+C\ep^{\frac{1}{2}})
\|\sqrt{a}(\py\ze,\py\zet)\|^2_{H_{\Psi}^{\frac{22}{3},0}}
+(\frac{1}{100}\eta+C\ep^{\frac{1}{2}})
\|\sqrt{a}\py\U\|^2_{H_{\Psi}^{7,0}}\\
&+C\ep^{\frac{1}{2}}
\|\sqrt{a}\py\Np\|^2_{H_{\Psi}^{\frac{20}{3},0}}.
\end{align*}
Next we need to estimate the left side of \eqref{PNzuoneiji}. Similar to \eqref{UU}, we have
\begin{align}\label{PPNNlk}
\begin{split}
&a(\Ll \Np,\Np)_{H_{\Psi}^{\frac{20}{3},0}}
+a(\Lk \Pp,\Pp)_{H_{\Psi}^{\frac{20}{3},0}}\\
\geq&\frac{1}{2}\dfrac{d}{dt}
\|\sqrt{a}(\Pp(t),\Np(t))\|^2_{H_{\Psi}^{\frac{20}{3},0}}
-\frac{1}{2}\|\sqrt{a'}(\Pp(t),\Np(t))\|^2_{H_{\Psi}^{\frac{20}{3},0}}\\
&+\la\tad\|\sqrt{a}(\Pp(t),\Np(t))\|^2_{H_{\Psi}^{\frac{20}{3},0}}
+2l_{\ka}\|\sqrt{a}(\py\Pp(t),\py\Np(t))\|^2_{H_{\Psi}^{\frac{20}{3},0}}\\
&+a\left(\T_u\px\Pp,\Pp\right)_{H_{\Psi}^{\frac{20}{3},0}}
+a\left(\T_u\px\Np,\Np\right)_{H_{\Psi}^{\frac{20}{3},0}}\\
&+a\left(\T_v\py\Pp,\Pp\right)_{H_{\Psi}^{\frac{20}{3},0}}
+a\left(\T_v\py\Np,\Np\right)_{H_{\Psi}^{\frac{20}{3},0}}
\\
&+a\left(\frac{2}{3}\de(t)
\T_{D_x u}Q(D_x)\px\Pp,\Pp\right)_{H_{\Psi}^{\frac{20}{3},0}}
+a\left(\frac{2}{3}\de(t)
\T_{D_x u}Q(D_x)\px\Np,\Np\right)_{H_{\Psi}^{\frac{20}{3},0}}\\
=&\sum_{i=1}^{10} J_i.
\end{split}
\end{align}
By a similar derivation of $a\left(\Ll\U,\U\right)_{H_{\Psi}^{7,0}}$, we can get
\begin{align*}
|\sum_{i=5}^{10} J_i|\leq
C(\ep^{\frac{1}{2}}+\eta^{-1}\ep^{\frac{3}{2}})\tad
\|\sqrt{a}(\Pp,\Np)\|^2_{H_{\Psi}^{7,0}}
+\frac{1}{100}\eta\|\sqrt{a}(\py\Pp,\py\Np)\|^2_{H_{\Psi}^{\frac{20}{3},0}}.
\end{align*}
Combining all the above estimates and integrating over $[0,t]$, we can obtain the following estimates.
\begin{align*}
&~~~~~\|\sqrt{a}(\Pp(t),\Np(t))\|^2_{H_{\Psi}^{\frac{20}{3},0}}
-\int_0^t\|\sqrt{a'}(\Pp,\Np)\|^2_{H_{\Psi}^{\frac{20}{3},0}}ds\\
&~~~~~+(4l_{\ka}-\frac{1}{10}\eta-C\ep^{\frac{1}{2}})\int_0^t \|\sqrt{a}(\py\Pp,\py\Np)\|^2_{H_{\Psi}^{\frac{20}{3},0}}ds\\
&~~~~~
+2\left(\la-C(1+\la\ep^{\frac{1}{2}}+\eta^{-1}\ep^{\frac{1}{2}})\right)\int_0^t
\tad\|\sqrt{a}(\Pp,\Np)\|^2_{H_{\Psi}^{7,0}}ds\\
&\leq
\|\sqrt{a}(0)(\Pp(0),\Np(0))\|^2_{H_{\Psi}^{\frac{20}{3},0}}
+C\int_0^t\tad
\left(\|\sqrt{a}\U\|^2_{H_{\Psi}^{\frac{22}{3},0}}
+\|\sqrt{a}(\uph,\bp)\|^2_{H_{\Psi}^{\frac{22}{3},0}}\right)
ds\\
&~~~~~+(\frac{3}{25}\eta+C\ep^{\frac{1}{2}})\int_0^t
\left(\|\sqrt{a}\py\U\|^2_{H_{\Psi}^{7,0}}
+\|\sqrt{a}(\py\ze,\py\zet)\|^2_{H_{\Psi}^{\frac{22}{3},0}}\right)ds.
\end{align*}
We have thus completed the proof of proposition \ref{PNdexianyanguji}.
\end{proof}

Now we can prove proposition \ref{propositionPN} by applying proposition \ref{PNdexianyanguji}.

\textbf{Proof of Proposition \ref{propositionPN}:}
By taking $a(t)=\ltr^{2l_{\ka}-2\eta}$ in \eqref{PNdexianyanguji2}, and using lemma \ref{gu6}, we deduce
\begin{align*}
&-(2l_{\ka}-2\eta)\int_0^t\|\lsr^{l_{\ka}-\eta-\frac{1}{2}}
(\Pp,\Np)\|^2_{H_{\Psi}^{\frac{20}{3},0}}ds\\
&+(4l_{\ka}-\frac{1}{10}\eta-C\ep^{\frac{1}{2}})\int_0^t \|\lsr^{l_{\ka}-\eta}
(\py\Pp,\py\Np)\|^2_{H_{\Psi}^{\frac{20}{3},0}}ds\\
\geq&
-2(2l_{\ka}-2\eta)\int_0^t \|\lsr^{l_{\ka}-\eta}
(\py\Pp,\py\Np)\|^2_{H_{\Psi}^{\frac{20}{3},0}}ds\\
&+(4l_{\ka}-\frac{1}{10}\eta-C\ep^{\frac{1}{2}})\int_0^t \|\lsr^{l_{\ka}-\eta}
(\py\Pp,\py\Np)\|^2_{H_{\Psi}^{\frac{20}{3},0}}ds\\
\geq&(\frac{39}{10}\eta-C\ep^{\frac{1}{2}})\int_0^t \|\lsr^{l_{\ka}-\eta}
(\py\Pp,\py\Np)\|^2_{H_{\Psi}^{\frac{20}{3},0}}ds,
\end{align*}
So we have
\begin{align*}
&\|\ltr^{l_{\ka}-\eta}
(\Pp(t),\Np(t))\|^2_{H_{\Psi}^{\frac{20}{3},0}}
+(\frac{39}{10}\eta-C\ep^{\frac{1}{2}})\int_0^t \|\lsr^{l_{\ka}-\eta}
(\py\Pp,\py\Np)\|^2_{H_{\Psi}^{\frac{20}{3},0}}ds\\
&
+2\left(\la-C(1+\la\ep^{\frac{1}{2}}+\eta^{-1}\ep^{\frac{1}{2}})\right)\int_0^t
\tad\|\lsr^{l_{\ka}-\eta}(\Pp,\Np)\|^2_{H_{\Psi}^{7,0}}ds\\
\leq&
\|(\Pp(0),\Np(0))\|^2_{H_{\Psi}^{\frac{20}{3},0}}
+C\int_0^t\tad
\left(\|\lsr^{l_{\ka}-\eta}\U\|^2_{H_{\Psi}^{\frac{22}{3},0}}
+\|\lsr^{l_{\ka}-\eta}
(\uph,\bp)\|^2_{H_{\Psi}^{\frac{22}{3},0}}\right)
ds\\
&+(\frac{3}{25}\eta+C\ep^{\frac{1}{2}})\int_0^t
\left(\|\lsr^{l_{\ka}-\eta}\py\U\|^2_{H_{\Psi}^{7,0}}
+\|\lsr^{l_{\ka}-\eta}
(\py\ze,\py\zet)\|^2_{H_{\Psi}^{\frac{22}{3},0}}\right)ds,
\end{align*}
Let $\ep_3$ be sufficiently small such that when $\ep<\ep_3<\ep_2$, we have $C\ep^{\frac{1}{2}}<\frac{1}{100}\eta.$ And $\la$ satisfies $\la>\la_2.$ Then proposition \ref{propositionPN} holds.

\hfill $\square$
\section{The Gevrey Estimates of $\Gp$ and $\Gtp$}\label{gujiG}
Starting from this section, we will give the estimates for the ``good function'' $(G,\Gt)$. For the sake of brevity in this paper, we give the relevant calculations for $G$, while the part of $\Gt$ is obtained in the same way. By acting the operator $e^{\Phi(t,D_x)}$ on \eqref{gfe} and doing the Bony's decomposition of all the nonlinear terms, we obtain
\begin{align}\label{gfephi}
\begin{aligned}
&~~~~~\pt\Gp+\la\tad(t)[D_x]^{\frac{2}{3}}\Gp-\py^2\Gp+\ltr^{-1}\Gp+\T_u\px\Gp
+\T_{\py G}\vp\\
&-\frac{1}{2\ltr}\T_{\py(y\psi)}\vp+\frac{y}{\ltr}\int_y^{\ty}\T_{\py u}\vp dz-\tad\Pp+\frac{y}{2\ltr}\int_y^{\ty}\tad\Pp dz=Z,
\end{aligned}
\end{align}
here $Z$ is defined as follows
\begin{align}\label{Z}
\begin{aligned}
-Z \stackrel{\mathrm{def}}{=}&\left(T_u^{\mathrm{h}} \partial_x G\right)_{\Phi}-T_u^{\mathrm{h}} \partial_x G_{\Phi}+\left(T_{\partial_x G}^{\mathrm{h}} u\right)_{\Phi}+\left(R^{\mathrm{h}}\left(u, \partial_x G\right)\right)_{\Phi}\\
&+\left(T_{\partial_y G}^{\mathrm{h}} v\right)_{\Phi}-T_{\partial_y G}^{\mathrm{h}} v_{\Phi}+\left(T_v^{\mathrm{h}} \partial_y G\right)_{\Phi}+\left(R^{\mathrm{h}}\left(v, \partial_y G\right)\right)_{\Phi}, \\
&-\frac{1}{2\langle t\rangle}\left(\left(T_{\partial_y(y \psi)}^{\mathrm{h}} v\right)_{\Phi}-T_{\partial_y(y \psi)}^{\mathrm{h}} v_{\Phi}\right)-\frac{1}{2\langle t\rangle}\left(T_v^{\mathrm{h}} \partial_y(y \psi)+R^{\mathrm{h}}\left(\partial_y(y \psi), v\right)\right)_{\Phi} \\
&+\frac{y}{\langle t\rangle} \int_y^{\infty}\left(\left(T_{\partial_y u}^{\mathrm{h}} v\right)_{\Phi}-T_{\partial_y u}^{\mathrm{h}} v_{\Phi}\right) dz+\frac{y}{\langle t\rangle} \int_y^{\infty}\left(T_v^{\mathrm{h}} \partial_y u+R^{\mathrm{h}}\left(v, \partial_y u\right)\right)_{\Phi} dz .
\end{aligned}
\end{align}

We first give the estimates of source term $Z$.
\begin{Lemma}\label{Zdeguji}
The source term $Z$ satisfies the following.
\begin{align*}
\|Z\|_{H_{\Psi}^{s,0}}\leq C\ep\left(
\ltr^{-\ga_0-\frac{1}{4}}\|\Gp\|_{H_{\Psi}^{s+\frac{2}{3},0}}
+\ltr^{-\ga_0+\frac{1}{4}}\|\py\Gp\|_{H_{\Psi}^{s,0}}\right).
\end{align*}
\end{Lemma}
The proof of this lemma can be found in lemma 6.1 of \cite{CW}, just notice the difference of lemma \ref{gu4} in this paper. For the sake of brevity, we will not give the proof here.

With this lemma, we can give the a priori estimates of $\Gp$, as follows:
\begin{Proposition}\label{Gpdexianyanguji}
Let $\ka\in(0,2)$ be a given constant. Let $(G,\Gt)$ be defined by \eqref{gf}, and $a(t)$ be a non-negative and non-decreasing function on $\R_+$. Then when $\al\leq\ga_0+\frac{5}{36}$, the following inequality holds for any $t<T_*$ and sufficiently small $\eta>0$.
\begin{align}\label{Gpdexianyanguji2}
\begin{aligned}
&\|\sqrt{a}(\Gp(t),\Gtp(t))\|^2_{H_{\Psi}^{4,0}}
-\int_0^t\|\sqrt{a'}(\Gp,\Gtp)\|^2_{H_{\Psi}^{4,0}}ds
+2\int_0^t\lsr^{-1}\|\sqrt{a}(\Gp,\Gtp)\|^2_{H_{\Psi}^{4,0}}ds\\
&+2\left(\la-C(1+\eta^{-1}\ep^{\frac{3}{2}})\right)
\int_0^t\tad(s)\|\sqrt{a}(\Gp,\Gtp)\|^2_{H_{\Psi}^{\frac{13}{3},0}}ds\\
&+(4l_{\ka}-\frac{2}{25}\eta)\int_0^t\|\sqrt{a}(\py\Gp,\py\Gtp)\|^2_{H_{\Psi}^{4,0}}ds\\
\leq&\|\sqrt{a}(0)(\Gp(0),\Gtp(0))\|^2_{H_{\Psi}^{4,0}}
+C\ep^{\frac{1}{2}}\int_0^t\tad
\|\lsr^{-1}\sqrt{a}(\uph,\bp)\|^{2}_{H_{\Psi}^{\frac{22}{3},0}}
ds
.
\end{aligned}
\end{align}
where $l_{\ka}=\frac{\ka(2-\ka)}{4}\in(0,\frac{1}{4}].$
\end{Proposition}
\begin{proof}
By taking $H_{\Psi}^{4,0}$ inner product of \eqref{gfephi} with
$a(t)\Gp$, and using \eqref{ptpy} we can get
\begin{align}\label{zaijianchiyixia}
\begin{aligned}
\frac{1}{2}&\frac{d}{dt}\|\sqrt{a}\Gp\|^2_{H_{\Psi}^{4,0}}
-\frac{1}{2}\|\sqrt{a'}\Gp\|^2_{H_{\Psi}^{4,0}}
+\ltr^{-1}\|\sqrt{a}\Gp\|^2_{H_{\Psi}^{4,0}}\\
&+\la\tad(t)\|\sqrt{a}\Gp\|^2_{H_{\Psi}^{\frac{13}{3},0}}
+\frac{1}{2}\|\sqrt{a}\py\Gp\|^2_{H_{\Psi}^{4,0}}\\
\leq&a\left|(\T_u\px \Gp,\Gp)_{H_{\Psi}^{4,0}}\right|
+a\left|(\T_{\py G}\vp,\Gp)_{H_{\Psi}^{4,0}}\right|
+\frac{a}{2\ltr}\left|(\T_{\py(y\psi)}\vp,\Gp)_{H_{\Psi}^{4,0}}\right|\\
&+a\left|(\frac{y}{\ltr}\int_y^{\ty}\T_{\py u}\vp dz,\Gp)_{H_{\Psi}^{4,0}}\right|
+a\tad\left|(\Pp,\Gp)_{H_{\Psi}^{4,0}}\right|\\
&+a\tad\left|(\frac{y}{2\ltr}\int_y^{\ty}\Pp dz,\Gp)_{H_{\Psi}^{4,0}}\right|
+a\left|(Z,\Gp)_{H_{\Psi}^{4,0}}\right|\\
=&\sum_{i=1}^{7} K_i.
\end{aligned}
\end{align}
For $K_1$, by lemma \ref{gu3}, we have
\begin{align*}
K_1\leq&C\|u\|_{L^{\ty}_{\vm}(H_{\h}^{\frac{3}{2}+})}
\|\sqrt{a}\Gp\|^2_{H_{\Psi}^{4,0}}\\
\leq& C\ep^{\frac{1}{2}}\ltr^{\al-\ga_0-\frac{1}{4}}\tad
\|\sqrt{a}\Gp\|^2_{H_{\Psi}^{4,0}}\\
\leq& C\tad
\|\sqrt{a}\Gp\|^2_{H_{\Psi}^{4,0}}.
\end{align*}
Next since $\py v=-\px u$, combining \eqref{ty} and lemma \ref{gu7}, we have that for any $s\in\R$, the following holds.
\begin{align}\label{vuG}
\begin{aligned}
\|\vp\|_{L^{\ty}_{\vm,\frac{\Psi}{2}}(H_{\h}^s)}\leq
C\ltr^{\frac{1}{4}}\|\uph\|_{H_{\frac{\Psi}{2}}^{s+1,0}}
\leq C\ltr^{\frac{1}{4}}\|\Gp\|_{H_{\Psi}^{s+1,0}}.
\end{aligned}
\end{align}
Then for $K_2$, by \eqref{T**} we deduce
\begin{align*}
K_2\leq &C\|\py G\|_{L^{2}_{\vm,\frac{\Psi}{2}}(H_{\h}^{\frac{1}{2}+})}
\|\sqrt{a}\vp\|_{L^{\ty}_{\vm,\frac{\Psi}{2}}(H_{\h}^{\frac{11}{3}})}
\|\sqrt{a}\Gp\|_{H_{\Psi}^{\frac{13}{3},0}}\\
\leq& C\ep\ltr^{-\ga_0-\frac{1}{4}}
\|\sqrt{a}\uph\|_{H_{\frac{\Psi}{2}}^{\frac{14}{3},0}}
\|\sqrt{a}\Gp\|_{H_{\Psi}^{\frac{13}{3},0}}.
\end{align*}
Using the anisotropic Sobolev inequality and the lemma \ref{gu7}, we can obtain
\begin{align}\label{chazhi}
\begin{aligned}
\|\sqrt{a}\uph\|_{H_{\frac{\Psi}{2}}^{\frac{14}{3},0}}
\leq &C\|\sqrt{a}\uph\|^{\frac{1}{9}}_{H_{\frac{\Psi}{2}}^{\frac{22}{3},0}}
\|\sqrt{a}\uph\|^{\frac{8}{9}}_{H_{\frac{\Psi}{2}}^{\frac{13}{3},0}}\\
\leq &C\ltr^{\frac{1}{9}}\|\ltr^{-1}\sqrt{a}\uph
\|^{\frac{1}{9}}_{H_{\Psi}^{\frac{22}{3},0}}
\|\sqrt{a}\Gp\|^{\frac{8}{9}}_{H_{\Psi}^{\frac{13}{3},0}}.
\end{aligned}
\end{align}
Thus we can derive
\begin{align*}
K_2\leq &C\ep\ltr^{-\ga_0-\frac{5}{36}}
\|\ltr^{-1}\sqrt{a}\uph
\|^{\frac{1}{9}}_{H_{\Psi}^{\frac{22}{3},0}}
\|\sqrt{a}\Gp\|^{\frac{17}{9}}_{H_{\Psi}^{\frac{13}{3},0}}\\
\leq &C\ep^{\frac{1}{2}}\ltr^{\al-\ga_0-\frac{5}{36}}\tad
\left(\|\ltr^{-1}\sqrt{a}\uph
\|^{2}_{H_{\Psi}^{\frac{22}{3},0}}
+\|\sqrt{a}\Gp\|^{2}_{H_{\Psi}^{\frac{13}{3},0}}\right)\\
\leq &C\ep^{\frac{1}{2}}\tad
\left(\|\ltr^{-1}\sqrt{a}\uph
\|^{2}_{H_{\Psi}^{\frac{22}{3},0}}
+\|\sqrt{a}\Gp\|^{2}_{H_{\Psi}^{\frac{13}{3},0}}\right),
\end{align*}
here we use $\al\leq\ga_0+\frac{5}{36}.$ For $K_3$, it follows from lemma \ref{gu7} that
\begin{align*}
K_3\leq&C
\|\ltr^{-1}\py(y\psi)\|_{L^2_{\vm,\frac{\Psi}{2}}(H_{\h}^{\frac{1}{2}+})}
\|\sqrt{a}\vp\|_{L^{\ty}_{\vm,\frac{\Psi}{2}}(H_{\h}^{\frac{11}{3}})}
\|\sqrt{a}\Gp\|_{H_{\Psi}^{\frac{13}{3},0}}\\
\leq&C\|\py G\|_{L^2_{\vm,\frac{\Psi}{2}}(H_{\h}^{\frac{1}{2}+})}
\|\sqrt{a}\vp\|_{L^{\ty}_{\vm,\frac{\Psi}{2}}(H_{\h}^{\frac{11}{3}})}
\|\sqrt{a}\Gp\|_{H_{\Psi}^{\frac{13}{3},0}},
\end{align*}
By the estimates of $K_2$, we get
\begin{align*}
K_3\leq C\ep^{\frac{1}{2}}\tad
\left(\|\ltr^{-1}\sqrt{a}\uph
\|^{2}_{H_{\Psi}^{\frac{22}{3},0}}
+\|\sqrt{a}\Gp\|^{2}_{H_{\Psi}^{\frac{13}{3},0}}\right).
\end{align*}
For $K_4,$ first for any $s\in\R$, we have
\begin{align}\label{hensuibian}
\begin{aligned}
\left\|\int_{y}^{\ty}\T_{\py u}\vp dz\right\|_{L^{\ty}_{\vm}(H_{\h}^{s})}
&\leq\left\|\int_{y}^{\ty}\|\py u\|_{H_{\h}^{\frac{1}{2}+}}\|\vp\|_{H_{\h}^s}dz\right\|_{L^{\ty}_{\vm}}\\
&\leq C \|\py u\|_{L^2_{\vm,\frac{3}{4}\Psi}(H_{\h}^{\frac{1}{2}+})}
\|e^{\frac{3}{4}\Psi}\vp\|_{L^{\ty}_{\vm}(H_{\h}^s)}
\left(\int_y^{\ty}e^{-3\Psi}dz\right)^{\frac{1}{2}}\\
&\leq C\ep\ltr^{-\ga_0}e^{-\frac{3}{2}\Psi}
\|\uph\|_{H_{\frac{3}{4}\Psi}^{1+s,0}},
\end{aligned}
\end{align}
Applying \eqref{chazhi}, we deduce
\begin{align*}
K_4&\leq a\|\frac{y}{\ltr}\int_{y}^{\ty}\T_{\py u}\vp dz\|_{H_{\Psi}^{\frac{11}{3},0}}
\|\Gp\|_{H_{\Psi}^{\frac{13}{3},0}}\\
&\leq C\ep^{\frac{1}{2}}\ltr^{\al-\ga_0-\frac{1}{2}}\tad
\|\frac{y}{\ltr^{\frac{1}{2}}}e^{-\frac{1}{2}\Psi}\|_{L^2_{\vm}}
\|\sqrt{a}\uph\|_{H_{\frac{3}{4}\Psi}^{\frac{14}{3},0}}
\|\sqrt{a}\Gp\|_{H_{\Psi}^{\frac{13}{3},0}}\\
&\leq C\ep^{\frac{1}{2}}\ltr^{\al-\ga_0-\frac{7}{18}}\tad
\left(\|\ltr^{-1}\sqrt{a}\uph
\|^{2}_{H_{\Psi}^{\frac{22}{3},0}}
+\|\sqrt{a}\Gp\|^{2}_{H_{\Psi}^{\frac{13}{3},0}}\right)\\
&\leq C\ep^{\frac{1}{2}}\tad
\left(\|\ltr^{-1}\sqrt{a}\uph
\|^{2}_{H_{\Psi}^{\frac{22}{3},0}}
+\|\sqrt{a}\Gp\|^{2}_{H_{\Psi}^{\frac{13}{3},0}}\right),
\end{align*}
here we use $\al\leq\ga_0+\frac{5}{36}.$

For $K_5$, it follows from \eqref{PN}, \eqref{bony}, \eqref{ty}, lemma \ref{gu1} and lemma \ref{ub} that for any $s\in\R$,
\begin{align}\label{haonana}
\begin{aligned}
\|\tad\Pp\|_{H_{\Psi}^{s,0}}
=&\|(b\px b+h\py b)_{\Phi}\|_{H_{\Psi}^{s,0}}\\
\leq&\|\bp\|_{L^{\ty}_{\vm,\frac{\Psi}{2}}(H_{\h}^{\frac{3}{2}+})}
\|\bp\|_{H_{\frac{\Psi}{2}}^{s+1,0}}
+\|\hp\|_{L^{\ty}_{\vm,\frac{\Psi}{2}}(H_{\h}^{\frac{1}{2}+})}
\|\py\bp\|_{H_{\frac{\Psi}{2}}^{s,0}}\\
&+\|\py\bp\|_{L^{2}_{\vm,\frac{\Psi}{2}}(H_{\h}^{\frac{1}{2}+})}
\|\hp\|_{L^{\ty}_{\vm,\frac{\Psi}{2}}(H_{\h}^{s})}\\
\leq&C\ep\ltr^{-\ga_0-\frac{1}{4}}\|\bp\|_{H_{\frac{\Psi}{2}}^{s+1,0}}
+C\ep\ltr^{-\ga_0+\frac{1}{4}}\|\py\bp\|_{H_{\frac{\Psi}{2}}^{s,0}}.
\end{aligned}
\end{align}
Thus for $K_5$, we have
\begin{align*}
K_5\leq &C a\|\tad\Pp\|_{H_{\Psi}^{\frac{11}{3},0}}
\|\Gp\|_{H_{\Psi}^{\frac{13}{3},0}}\\
\leq&C\ep\ltr^{-\ga_0-\frac{1}{4}}a\|\bp\|_{H_{\frac{\Psi}{2}}^{\frac{14}{3},0}}
\|\Gp\|_{H_{\Psi}^{\frac{13}{3},0}}
+C\ep\ltr^{-\ga_0+\frac{1}{4}}a\|\py\bp\|_{H_{\frac{\Psi}{2}}^{\frac{11}{3},0}}
\|\Gp\|_{H_{\Psi}^{\frac{13}{3},0}}\\
=&K_{51}+K_{52}.
\end{align*}
Using the estimates of $K_2$, we deduce
\begin{align*}
K_{51}\leq C\ep^{\frac{1}{2}}\tad
\left(\|\ltr^{-1}\sqrt{a}\bp
\|^{2}_{H_{\Psi}^{\frac{22}{3},0}}
+\|\sqrt{a}(\Gp,\Gtp)\|^{2}_{H_{\Psi}^{\frac{13}{3},0}}\right).
\end{align*}
On the other hand, by using lemma \ref{gu7}, it yields
\begin{align*}
K_{52}\leq &C\ep\ltr^{-\ga_0+\frac{1}{4}}
\|\sqrt{a}\py\Gtp\|_{H_{\Psi}^{\frac{11}{3},0}}
\|\sqrt{a}\Gp\|_{H_{\Psi}^{\frac{13}{3},0}}\\
\leq&C\eta^{-1}\ep^{\frac{3}{2}}\ltr^{\al-2\ga_0+\frac{1}{2}}\tad
\|\sqrt{a}\Gp\|^2_{H_{\Psi}^{\frac{13}{3},0}}
+\frac{1}{100}\eta\|\sqrt{a}\py\Gtp\|^2_{H_{\Psi}^{4,0}}\\
\leq&C\eta^{-1}\ep^{\frac{3}{2}}\tad
\|\sqrt{a}\Gp\|^2_{H_{\Psi}^{\frac{13}{3},0}}
+\frac{1}{100}\eta\|\sqrt{a}\py\Gtp\|^2_{H_{\Psi}^{4,0}}.
\end{align*}
here we use $\al\leq\ga_0+\frac{5}{36}.$ Combining $K_{51}$ and $K_{52}$, we get
\begin{align*}
K_5
\leq&C(\ep^{\frac{1}{2}}+\eta^{-1}\ep^{\frac{3}{2}})\tad
\|\sqrt{a}(\Gp,\Gtp)\|^{2}_{H_{\Psi}^{\frac{13}{3},0}}
+\frac{1}{100}\eta\|\sqrt{a}\py\Gtp\|^2_{H_{\Psi}^{4,0}}
+C\ep^{\frac{1}{2}}\tad
\|\ltr^{-1}\sqrt{a}\bp
\|^{2}_{H_{\Psi}^{\frac{22}{3},0}}.
\end{align*}

For $K_6$, first by applying lemma \ref{gu6}, we know that for any $s\in\R$,
\begin{align}\label{xinqingbucuo}
\begin{aligned}
\|\frac{y}{2\ltr}\int^{\ty}_{y}f dz\|_{H_{\Psi}^{s,0}}
\leq C\|f\|_{H_{\Psi}^{s,0}},
\end{aligned}
\end{align}
Thus similar to $K_5$, we have
\begin{align*}
K_6
\leq&C(\ep^{\frac{1}{2}}+\eta^{-1}\ep^{\frac{3}{2}})\tad
\|\sqrt{a}(\Gp,\Gtp)\|^{2}_{H_{\Psi}^{\frac{13}{3},0}}
+\frac{1}{100}\eta\|\sqrt{a}\py\Gtp\|^2_{H_{\Psi}^{4,0}}
+C\ep^{\frac{1}{2}}\tad
\|\ltr^{-1}\sqrt{a}\bp
\|^{2}_{H_{\Psi}^{\frac{22}{3},0}}.
\end{align*}
For $K_7$, by lemma \ref{Zdeguji} we derive that
\begin{align*}
K_7\leq&C\ep\ltr^{-\ga_0-\frac{1}{4}}
\|\sqrt{a}\Gp\|^2_{H_{\Psi}^{\frac{13}{3},0}}
+C\ep\ltr^{-\ga_0+\frac{1}{4}}
\|\sqrt{a}\py\Gp\|_{H_{\Psi}^{\frac{11}{3},0}}
\|\sqrt{a}\Gp\|_{H_{\Psi}^{\frac{13}{3},0}}\\
\leq&C(\ep^{\frac{1}{2}}\ltr^{\al-\ga_0-\frac{1}{4}}
+\eta^{-1}\ep^{\frac{3}{2}}\ltr^{\al-2\ga_0+\frac{1}{2}})\tad
\|\sqrt{a}\Gp\|^2_{H_{\Psi}^{\frac{13}{3},0}}
+\frac{1}{100}\eta\|\sqrt{a}\py\Gp\|^2_{H_{\Psi}^{\frac{11}{3},0}}\\
\leq&C(\ep^{\frac{1}{2}}
+\eta^{-1}\ep^{\frac{3}{2}})\tad
\|\sqrt{a}\Gp\|^2_{H_{\Psi}^{\frac{13}{3},0}}
+\frac{1}{100}\eta\|\sqrt{a}\py\Gp\|^2_{H_{\Psi}^{\frac{11}{3},0}}.
\end{align*}
By taking $K_1-K_7$ into \eqref{zaijianchiyixia} and integrating over $[0,t]$, we can obtain
\begin{align*}
&\|\sqrt{a}\Gp(t)\|^2_{H_{\Psi}^{4,0}}
-\int_0^t\|\sqrt{a'}\Gp\|^2_{H_{\Psi}^{4,0}}ds
+\int_0^t\ltr^{-1}\|\sqrt{a}\Gp\|^2_{H_{\Psi}^{4,0}}ds\\
&+2\left(\la-C(1+\eta^{-1}\ep^{\frac{3}{2}})\right)
\int_0^t\tad(s)\|\sqrt{a}\Gp\|^2_{H_{\Psi}^{\frac{13}{3},0}}ds\\
&+(1-\frac{1}{25}\eta)\int_0^t\|\sqrt{a}\py\Gp\|^2_{H_{\Psi}^{4,0}}ds\\
\leq&\|\sqrt{a}(0)\Gp(0)\|^2_{H_{\Psi}^{4,0}}
+C(1+\eta^{-1}\ep^{\frac{3}{2}})
\int_0^t\tad(s)\|\sqrt{a}\Gtp\|^2_{H_{\Psi}^{\frac{13}{3},0}}ds\\
&+C\ep^{\frac{1}{2}}\int_0^t\tad
\|\lsr^{-1}\sqrt{a}(\uph,\bp)\|^{2}_{H_{\Psi}^{\frac{22}{3},0}}
ds
+\frac{1}{25}\eta\int_0^t\|\sqrt{a}\py\Gtp\|^2_{H_{\Psi}^{4,0}}ds.
\end{align*}
The estimates for $\Gtp$ follow in the same way, and we have thus completed the proof of proposition \ref{Gpdexianyanguji}.
\end{proof}

Now we can prove proposition \ref{propositionG} by applying proposition \ref{Gpdexianyanguji}.

\textbf{Proof of Proposition \ref{propositionG}:}
Take $a(t)=\ltr^{2+2l_{\ka}-2\eta}$ in \eqref{Gpdexianyanguji2}. By lemma \ref{gu6} we deduce
\begin{align*}
&-(2+2l_{\ka}-2\eta)\int_0^t\|\lsr^{1+l_{\ka}-\eta-\frac{1}{2}}(\Gp,\Gtp)\|^2_{H_{\Psi}^{4,0}}ds
+2\int_0^t\|\lsr^{1+l_{\ka}-\eta-\frac{1}{2}}(\Gp,\Gtp)\|^2_{H_{\Psi}^{4,0}}ds\\
&+(4l_{\ka}-\frac{2}{25}\eta)\int_0^t\|
\lsr^{1+l_{\ka}-\eta}(\py\Gp,\py\Gtp)\|^2_{H_{\Psi}^{4,0}}ds\\
\geq&\frac{98}{25}\eta
\int_0^t\|
\lsr^{1+l_{\ka}-\eta}(\py\Gp,\py\Gtp)\|^2_{H_{\Psi}^{4,0}}ds,
\end{align*}
then we get
\begin{align*}
&\|\ltr^{1+l_{\ka}-\eta}(\Gp(t),\Gtp(t))\|^2_{H_{\Psi}^{4,0}}
+\frac{98}{25}\eta
\int_0^t\|
\lsr^{1+l_{\ka}-\eta}(\py\Gp,\py\Gtp)\|^2_{H_{\Psi}^{4,0}}ds\\
&+2\left(\la-C(1+\eta^{-1}\ep^{\frac{3}{2}})\right)
\int_0^t\tad(s)\|\lsr^{1+l_{\ka}-\eta}(\Gp,\Gtp)\|^2_{H_{\Psi}^{\frac{13}{3},0}}ds\\
\leq&\|(\Gp(0),\Gtp(0))\|^2_{H_{\Psi}^{4,0}}
+C\ep^{\frac{1}{2}}\int_0^t\tad
\|\lsr^{l_{\ka}-\eta}\sqrt{a}(\uph,\bp)\|^{2}_{H_{\Psi}^{\frac{22}{3},0}}
ds
.
\end{align*}
Thus when $\ep<\ep_3$ and $\la>\la_2$, proposition \ref{propositionG} holds.

\hfill $\square$

\section{The Gevrey Estimates of $\py\Gp$ and $\py\Gtp$}\label{gujipyG}
In this section, we will give the estimates of $(\py G,\py\Gt)$, for which we act on \eqref{gfephi} with $\py$,
\begin{align}\label{pygfephi}
\begin{aligned}
&~~~~~\pt\py\Gp+\la\tad(t)[D_x]^{\frac{2}{3}}\py\Gp-\py^3\Gp
+\ltr^{-1}\py\Gp+\py\left(\T_u\px\Gp
+\T_{\py G}\vp\right.\\
&\left.-\frac{1}{2\ltr}\T_{\py(y\psi)}\vp+\frac{y}{\ltr}\int_y^{\ty}\T_{\py u}\vp dz-\tad\Pp+\frac{y}{2\ltr}\int_y^{\ty}\tad\Pp dz\right)=\py Z,
\end{aligned}
\end{align}
where $Z$ is defined by \eqref{Z}.
\begin{Proposition}\label{pyGdexianyanguji}
Let $\ka\in(0,2)$ be a given constant. Let $(G,\Gt)$ be defined by \eqref{gf}, and $a(t)$ be a non-negative and non-decreasing function on $\R_+$. Then when $\al\leq\ga_0+\frac{5}{36}$, the following inequality holds for any $t<T_*$ and sufficiently small $\eta>0$.
\begin{align}\label{pyGdexianyanguji2}
\begin{aligned}
&\|\sqrt{a}(\py\Gp(t),\py\Gtp(t))\|^2_{H_{\Psi}^{3,0}}
-\int_0^t\|\sqrt{a'}(\py\Gp,\py\Gtp)\|^2_{H_{\Psi}^{3,0}}ds\\
&+2\int_0^t\lsr^{-1}\|\sqrt{a}(\py\Gp,\py\Gtp)\|^2_{H_{\Psi}^{3,0}}ds
+(4l_{\ka}-\frac{7}{25}\eta)
\int_0^t\|\sqrt{a}(\py^2\Gp,\py^2\Gtp)\|^2_{H_{\Psi}^{3,0}}ds\\
&+2(\la-C\eta^{-1}\ep^{\frac{3}{2}})
\int_0^t\tad\|\sqrt{a}(\py\Gp,
\py\Gtp)\|^2_{H_{\Psi}^{\frac{10}{3},0}}ds\\
\leq&\|\sqrt{a}(0)(\py\Gp(0),\py\Gtp(0))\|^2_{H_{\Psi}^{3,0}}
+C\eta^{-1}\ep^{\frac{3}{2}}\int_0^t\tad
\|\lsr^{-\frac{1}{2}}\sqrt{a}(\Gp,\Gtp)
\|^2_{H_{\Psi}^{\frac{13}{3},0}}ds,
\end{aligned}
\end{align}
where $l_{\ka}=\frac{\ka(2-\ka)}{4}\in(0,\frac{1}{4}].$
\end{Proposition}
\begin{proof}
First noticing $u|_{y=0}=v|_{y=0}=P|_{y=0}=G|_{y=0}$, combining
\eqref{gfe} and \eqref{Z},  we have
\begin{align}\label{bianjietiaojian}
\begin{aligned}
&~~~~~~~~~~~~~~~~~~~~~~~~~~~
\py^2 \Gp|_{y=0}=\py^2 G|_{y=0}=0,~~
\Pp|_{y=0}=0,~~ Z|_{y=0}=0\\
&\left(\T_u\px\Gp
+\T_{\py G}\vp
-\frac{1}{2\ltr}\T_{\py(y\psi)}\vp+\frac{y}{\ltr}\int_y^{\ty}\T_{\py u}\vp dz+\frac{y}{2\ltr}\int_y^{\ty}\tad\Pp dz\right)|_{y=0}=0
\end{aligned}
\end{align}
These boundary conditions ensure that no boundary term arises in integrating by parts.

Taking $H_{\Psi}^{3,0}$ inner product of \eqref{pygfephi} with $a(t)\py\Gp$ and integrating by parts, we derive
\begin{align}\label{zaijianchiliangxia}
\begin{aligned}
\frac{1}{2}&\frac{d}{dt}\|\sqrt{a}\py\Gp\|^2_{H_{\Psi}^{3,0}}
-\frac{1}{2}\|\sqrt{a'}\py\Gp\|^2_{H_{\Psi}^{3,0}}
+\ltr^{-1}\|\sqrt{a}\py\Gp\|^2_{H_{\Psi}^{3,0}}\\
&+\la\tad(t)\|\sqrt{a}\py\Gp\|^2_{H_{\Psi}^{\frac{10}{3},0}}
+\frac{1}{2}\|\sqrt{a}\py^2\Gp\|^2_{H_{\Psi}^{3,0}}\\
\leq&a\left|(\T_u\px \Gp,\py^2\Gp)_{H_{\Psi}^{3,0}}\right|
+a\left|(\T_{\py G}\vp,\py^2\Gp)_{H_{\Psi}^{3,0}}\right|\\
&+\frac{a}{2\ltr}\left|(\T_{\py(y\psi)}\vp,\py^2\Gp
)_{H_{\Psi}^{3,0}}\right|
+a\left|(\frac{y}{\ltr}\int_y^{\ty}\T_{\py u}\vp dz,\py^2\Gp)_{H_{\Psi}^{3,0}}\right|\\
&+a\tad\left|(\Pp,\py^2\Gp)_{H_{\Psi}^{3,0}}\right|
+a\tad\left|(\frac{y}{2\ltr}\int_y^{\ty}\Pp dz,\py^2\Gp)_{H_{\Psi}^{3,0}}\right|\\
&+a\left|(\T_u\px \Gp,\frac{y}{2\ltr}\py\Gp)_{H_{\Psi}^{3,0}}\right|
+a\left|(\T_{\py G}\vp,\frac{y}{2\ltr}\py\Gp)_{H_{\Psi}^{3,0}}\right|\\
&+\frac{a}{2\ltr}\left|(\T_{\py(y\psi)}\vp,\frac{y}{2\ltr}\py\Gp
)_{H_{\Psi}^{3,0}}\right|
+a\left|(\frac{y}{\ltr}\int_y^{\ty}\T_{\py u}\vp dz,\frac{y}{2\ltr}\py\Gp)_{H_{\Psi}^{3,0}}\right|\\
&+a\tad\left|(\Pp,\frac{y}{2\ltr}\py\Gp)_{H_{\Psi}^{3,0}}\right|
+a\tad\left|(\frac{y}{2\ltr}\int_y^{\ty}\Pp dz,\frac{y}{2\ltr}\py\Gp)_{H_{\Psi}^{3,0}}\right|\\
&+a\left|(Z,\py^2\Gp)_{H_{\Psi}^{3,0}}\right|
+a\left|(Z,\frac{y}{2\ltr}\py\Gp)_{H_{\Psi}^{3,0}}\right|
\\
=&\sum_{i=1}^{14} L_i.
\end{aligned}
\end{align}
For $L_1,$ it follows from lemma\ref{gu1} and lemma \ref{ub} that
\begin{align*}
L_1\leq&C\|\sqrt{a}\T_u \px \Gp\|_{H_{\Psi}^{3,0}}
\|\sqrt{a}\py^2\Gp\|_{H_{\Psi}^{3,0}}\\
\leq&C\|u\|_{L^{\ty}_{\vm}(H_{\h}^{\frac{1}{2}+})}
\|\sqrt{a}\Gp\|_{H_{\Psi}^{4,0}}
\|\sqrt{a}\py^2\Gp\|_{H_{\Psi}^{3,0}}\\
\leq&C\ep\ltr^{-\ga_0+\frac{1}{4}}
\|\ltr^{-\frac{1}{2}}\sqrt{a}\Gp\|_{H_{\Psi}^{4,0}}
\|\sqrt{a}\py^2\Gp\|_{H_{\Psi}^{3,0}}\\
\leq&\frac{1}{100}\eta\|\sqrt{a}\py^2\Gp\|^2_{H_{\Psi}^{3,0}}
+C\eta^{-1}\ep^{\frac{3}{2}}\ltr^{\al-2\ga_0+\frac{1}{2}}\tad
\|\ltr^{-\frac{1}{2}}\sqrt{a}\Gp\|^2_{H_{\Psi}^{4,0}}\\
\leq&\frac{1}{100}\eta\|\sqrt{a}\py^2\Gp\|^2_{H_{\Psi}^{3,0}}
+C\eta^{-1}\ep^{\frac{3}{2}}\tad
\|\ltr^{-\frac{1}{2}}\sqrt{a}\Gp\|^2_{H_{\Psi}^{\frac{13}{3},0}},
\end{align*}
here we use $\al\leq\ga_0+\frac{5}{36}$ and $\ga_0>1.$ For $L_2$, applying lemma \ref{gu1}, \eqref{T**} and \eqref{vuG}, we deduce
\begin{align*}
L_2\leq&C\|\py G\|_{L^2_{\vm,\frac{1}{2}\Psi}(H_{\h}^{\frac{1}{2}+})}
\|\sqrt{a}\vp\|_{L^{\ty}_{\vm,\frac{1}{2}\Psi}(H_{\h}^{3})}
\|\sqrt{a}\py^2\Gp\|_{H_{\Psi}^{3,0}}\\
\leq &C\ep\ltr^{-\ga_0-\frac{1}{4}}
\|\sqrt{a}\Gp\|_{H_{\Psi}^{4,0}}
\|\sqrt{a}\py^2\Gp\|_{H_{\Psi}^{3,0}}.
\end{align*}
Then similar to$L_1$, we have
\begin{align*}
L_2
\leq\frac{1}{100}\eta\|\sqrt{a}\py^2\Gp\|^2_{H_{\Psi}^{3,0}}
+C\eta^{-1}\ep^{\frac{3}{2}}\tad
\|\ltr^{-\frac{1}{2}}\sqrt{a}\Gp\|^2_{H_{\Psi}^{\frac{13}{3},0}}.
\end{align*}
For $L_3$, noticing lemma \ref{gu7}, we can use a similar derivation of $L_2$ to get
\begin{align*}
L_3
\leq\frac{1}{100}\eta\|\sqrt{a}\py^2\Gp\|^2_{H_{\Psi}^{3,0}}
+C\eta^{-1}\ep^{\frac{3}{2}}\tad
\|\ltr^{-\frac{1}{2}}\sqrt{a}\Gp\|^2_{H_{\Psi}^{\frac{13}{3},0}}.
\end{align*}
For $L_4,$ by \eqref{hensuibian}, we can gain
\begin{align*}
L_4\leq&C a\|\frac{y}{\ltr}\int_y^{\ty}\T_{\py u}\vp dz\|_{H_{\Psi}^{3,0}}
\|\py^2\Gp\|_{H_{\Psi}^{3,0}}\\
\leq&C\ep\ltr^{-\ga_0-\frac{1}{2}}
\|\frac{y}{\ltr^{\frac{1}{2}}}e^{-\frac{1}{2}\Psi}\|_{L^2_{\vm}}
\|\sqrt{a}\Gp\|_{H_{\Psi}^{4,0}}
\|\sqrt{a}\py^2\Gp\|_{H_{\Psi}^{3,0}}\\
\leq&C\ep\ltr^{-\ga_0-\frac{1}{4}}
\|\sqrt{a}\Gp\|_{H_{\Psi}^{4,0}}
\|\sqrt{a}\py^2\Gp\|_{H_{\Psi}^{3,0}}.
\end{align*}
Therefore similarly we have
\begin{align*}
L_4
\leq\frac{1}{100}\eta\|\sqrt{a}\py^2\Gp\|^2_{H_{\Psi}^{3,0}}
+C\eta^{-1}\ep^{\frac{3}{2}}\tad
\|\ltr^{-\frac{1}{2}}\sqrt{a}\Gp\|^2_{H_{\Psi}^{\frac{13}{3},0}}.
\end{align*}
For $L_5,$ by lemma \ref{gu7} and \eqref{haonana} we obtain
\begin{align*}
L_5\leq&\left(C\ep\ltr^{-\ga_0-\frac{1}{4}}
\|\sqrt{a}\bp\|_{H_{\frac{1}{2}\Psi}^{4,0}}
+C\ep\ltr^{-\ga_0+\frac{1}{4}}
\|\sqrt{a}\py\bp\|_{H_{\frac{1}{2}\Psi}^{3,0}}
\right)\|\sqrt{a}\py^2\Gp\|_{H_{\Psi}^{3,0}}\\
\leq&\left(C\ep\ltr^{-\ga_0-\frac{1}{4}}
\|\sqrt{a}\Gtp\|_{H_{\Psi}^{4,0}}
+C\ep\ltr^{-\ga_0+\frac{1}{4}}
\|\sqrt{a}\py\Gtp\|_{H_{\Psi}^{3,0}}
\right)\|\sqrt{a}\py^2\Gp\|_{H_{\Psi}^{3,0}}\\
\leq&\frac{1}{100}\eta \|\sqrt{a}\py^2\Gp\|^2_{H_{\Psi}^{3,0}}
+C\eta^{-1}\ep^{\frac{3}{2}}\ltr^{\al-2\ga_0+\frac{1}{2}}\tad
\|\ltr^{-\frac{1}{2}}\sqrt{a}\Gtp\|^2_{H_{\Psi}^{\frac{13}{3},0}}\\
&+C\eta^{-1}\ep^{\frac{3}{2}}\ltr^{\al-2\ga_0+\frac{1}{2}}\tad
\|\sqrt{a}\py\Gtp\|^2_{H_{\Psi}^{\frac{10}{3},0}}\\
\leq&\frac{1}{100}\eta \|\sqrt{a}\py^2\Gp\|^2_{H_{\Psi}^{3,0}}
+C\eta^{-1}\ep^{\frac{3}{2}}\tad
\|\ltr^{-\frac{1}{2}}\sqrt{a}\Gtp\|^2_{H_{\Psi}^{\frac{13}{3},0}}
+C\eta^{-1}\ep^{\frac{3}{2}}\tad
\|\sqrt{a}\py\Gtp\|^2_{H_{\Psi}^{\frac{10}{3},0}}.
\end{align*}
For $L_6$, by using \eqref{xinqingbucuo} and the estimates of $L_5$, we have
\begin{align*}
L_6
\leq\frac{1}{100}\eta \|\sqrt{a}\py^2\Gp\|^2_{H_{\Psi}^{3,0}}
+C\eta^{-1}\ep^{\frac{3}{2}}\tad
\|\ltr^{-\frac{1}{2}}\sqrt{a}\Gtp\|^2_{H_{\Psi}^{\frac{13}{3},0}}
+C\eta^{-1}\ep^{\frac{3}{2}}\tad
\|\sqrt{a}\py\Gtp\|^2_{H_{\Psi}^{\frac{10}{3},0}}.
\end{align*}
For $L_7-L_{12}$, using \eqref{xinqingbucuo} and the estimates of $L_1-L_6$, we can directly get
\begin{align*}
\sum_{i=7}^{12} L_i\leq&
\frac{3}{50}\eta \|\sqrt{a}\py^2\Gp\|^2_{H_{\Psi}^{3,0}}
+C\eta^{-1}\ep^{\frac{3}{2}}\tad
\|\ltr^{-\frac{1}{2}}\sqrt{a}(\Gp,\Gtp)\|^2_{H_{\Psi}^{\frac{13}{3},0}}\\
&+C\eta^{-1}\ep^{\frac{3}{2}}\tad
\|\sqrt{a}\py\Gtp\|^2_{H_{\Psi}^{\frac{10}{3},0}}.
\end{align*}
For $L_{13},$ by applying lemma \ref{Zdeguji}, we can derive
\begin{align*}
L_{13}\leq&\|\sqrt{a}Z\|_{H_{\Psi}^{3,0}}
\|\sqrt{a}\py^2\Gp\|_{H_{\Psi}^{3,0}}\\
\leq&C\ep\left(
\ltr^{-\ga_0-\frac{1}{4}}\|\sqrt{a}\Gp\|_{H_{\Psi}^{\frac{11}{3},0}}
+\ltr^{-\ga_0+\frac{1}{4}}\|\sqrt{a}\py\Gp\|_{H_{\Psi}^{3,0}}\right)
\|\sqrt{a}\py^2\Gp\|_{H_{\Psi}^{3,0}}\\
\leq&\frac{1}{100}\eta\|\sqrt{a}\py^2\Gp\|^2_{H_{\Psi}^{3,0}}
+C\eta^{-1}\ep^{\frac{3}{2}}\ltr^{\al-2\ga_0+\frac{1}{2}}\tad
\left(
\|\ltr^{-\frac{1}{2}}\sqrt{a}\Gp\|^2_{H_{\Psi}^{4,0}}
+\|\sqrt{a}\py\Gp\|^2_{H_{\Psi}^{3,0}}\right)\\
\leq&\frac{1}{100}\eta\|\sqrt{a}\py^2\Gp\|^2_{H_{\Psi}^{3,0}}
+C\eta^{-1}\ep^{\frac{3}{2}}\tad
\left(
\|\ltr^{-\frac{1}{2}}\sqrt{a}\Gp\|^2_{H_{\Psi}^{\frac{13}{3},0}}
+\|\sqrt{a}\py\Gp\|^2_{H_{\Psi}^{\frac{10}{3},0}}\right).
\end{align*}
here we use $\al\leq\ga_0+\frac{5}{36}$ and $\ga_0>1.$

For $L_{14}$, similarly by applying \eqref{xinqingbucuo}, we can deduce
\begin{align*}
L_{14}
\leq\frac{1}{100}\eta\|\sqrt{a}\py^2\Gp\|^2_{H_{\Psi}^{3,0}}
+C\eta^{-1}\ep^{\frac{3}{2}}\tad
\left(
\|\ltr^{-\frac{1}{2}}\sqrt{a}\Gp\|^2_{H_{\Psi}^{\frac{13}{3},0}}
+\|\sqrt{a}\py\Gp\|^2_{H_{\Psi}^{\frac{10}{3},0}}\right).
\end{align*}

Taking $L_1-L_{14}$ into \eqref{zaijianchiliangxia}, and integrating over $[0,t]$, we can get the estimates of $\py\Gp$.
\begin{align*}
&\|\sqrt{a}\py\Gp(t)\|^2_{H_{\Psi}^{3,0}}
-\int_0^t\|\sqrt{a'}\py\Gp\|^2_{H_{\Psi}^{3,0}}ds
+2\int_0^t\lsr^{-1}\|\sqrt{a}\py\Gp\|^2_{H_{\Psi}^{3,0}}ds\\
&+2\la\int_0^t\tad\|\sqrt{a}\py\Gp\|^2_{H_{\Psi}^{\frac{10}{3},0}}ds
+\int_0^t\|\sqrt{a}\py^2\Gp\|^2_{H_{\Psi}^{3,0}}ds\\
\leq&\|\sqrt{a}(0)\py\Gp(0)\|^2_{H_{\Psi}^{3,0}}
+C\eta^{-1}\ep^{\frac{3}{2}}
\int_0^t\tad
\|\sqrt{a}(\py\Gp,\py\Gtp)\|^2_{H_{\Psi}^{\frac{10}{3},0}}ds\\
&+\frac{7}{25}\eta\int_0^t\|\sqrt{a}\py^2\Gp\|^2_{H_{\Psi}^{3,0}}ds
+C\eta^{-1}\ep^{\frac{3}{2}}\int_0^t\tad
\|\lsr^{-\frac{1}{2}}\sqrt{a}(\Gp,\Gtp)
\|^2_{H_{\Psi}^{\frac{13}{3},0}}ds.
\end{align*}
This completes the estimates of $\py\Gp$. And the estimates of $\py\Gtp$ can be obtained in a similar way. Thus we have completed the proof of proposition \ref{pyGdexianyanguji}.
\end{proof}

Now we can prove proposition \ref{propositionpyG}.

\textbf{Proof of Proposition \ref{propositionpyG}:}
Taking $a(t)=\ltr^{3+2l_{\ka}-2\eta}$ in \eqref{pyGdexianyanguji2}, we have
\begin{align*}
&\|\ltr^{\frac{3}{2}+l_{\ka}-\eta}(\py\Gp(t),\py\Gtp(t))\|^2_{H_{\Psi}^{3,0}}
+(4l_{\ka}-\frac{7}{25}\eta)
\int_0^t\|\lsr^{\frac{3}{2}+l_{\ka}-\eta}
(\py^2\Gp,\py^2\Gtp)\|^2_{H_{\Psi}^{3,0}}ds\\
&+2(\la-C\eta^{-1}\ep^{\frac{3}{2}})
\int_0^t\tad\|\lsr^{\frac{3}{2}+l_{\ka}-\eta}(\py\Gp,
\py\Gtp)\|^2_{H_{\Psi}^{\frac{10}{3},0}}ds\\
\leq&\|(\py\Gp(0),\py\Gtp(0))\|^2_{H_{\Psi}^{3,0}}
+C\eta^{-1}\ep^{\frac{3}{2}}\int_0^t\tad
\|\lsr^{1+l_{\ka}-\eta}(\Gp,\Gtp)
\|^2_{H_{\Psi}^{\frac{13}{3},0}}ds\\
&+(1+2\l_{\ka}-2\eta)\int_0^t\|\lsr^{1+l_{\ka}-\eta}
(\py\Gp,\py\Gtp)\|^2_{H_{\Psi}^{3,0}}ds.
\end{align*}
And using \eqref{propositionG1} we can deduce
\begin{align*}
&\|\ltr^{\frac{3}{2}+l_{\ka}-\eta}(\py\Gp(t),\py\Gtp(t))\|^2_{H_{\Psi}^{3,0}}
+(4l_{\ka}-\frac{7}{25}\eta)
\int_0^t\|\lsr^{\frac{3}{2}+l_{\ka}-\eta}
(\py^2\Gp,\py^2\Gtp)\|^2_{H_{\Psi}^{3,0}}ds\\
&+2(\la-C\eta^{-1}\ep^{\frac{3}{2}})
\int_0^t\tad\|\lsr^{\frac{3}{2}+l_{\ka}-\eta}(\py\Gp,
\py\Gtp)\|^2_{H_{\Psi}^{\frac{10}{3},0}}ds\\
\leq&C\eta^{-1}\left(\|(\py\Gp(0),\py\Gtp(0))\|^2_{H_{\Psi}^{3,0}}
+\|(\Gp(0),\Gtp(0))\|^2_{H_{\Psi}^{4,0}}\right)\\
&+C\eta^{-1}\int_0^t\tad
\|\lsr^{l_{\ka}-\eta}\sqrt{a}(\uph,\bp)\|^{2}_{H_{\Psi}^{\frac{22}{3},0}}
ds,
\end{align*}
thus when $\ep<\ep_3$ and $\la>\la_2$, proposition \ref{propositionpyG} holds.

\hfill $\square$
\section{The Gevrey Estimates of $\py^2\Gp$ and $\py^2\Gtp$}\label{jiusuibianbianhao}
Next, we estimate $(\py^2 \Gp,\py^2\Gtp)$. Unlike \cite{CW}, we need to do Gevrey norm estimates, whereas in \cite{CW} they only need Sobolev norm estimates. This is because the $(P,N)$ term due to the magnetic field leads us to estimate Gevrey norm of $(\py^2 \uph,\py^2 \bp)$, which can be found in the estimates of $M_{16}$ in section \ref{lalalademaxiya}.

Acting $\py$ on \eqref{pygfephi}, we can get
\begin{align}\label{pypygfephi}
\begin{aligned}
&~~~~~\pt\py^2\Gp+\la\tad(t)[D_x]^{\frac{2}{3}}\py^2\Gp-\py^4\Gp
+\ltr^{-1}\py^2\Gp+\py^2\left(\T_u\px\Gp
+\T_{\py G}\vp\right.\\
&\left.-\frac{1}{2\ltr}\T_{\py(y\psi)}\vp+\frac{y}{\ltr}\int_y^{\ty}\T_{\py u}\vp dz-\tad\Pp+\frac{y}{2\ltr}\int_y^{\ty}\tad\Pp dz\right)=\py^2 Z,
\end{aligned}
\end{align}
and the boundary condition is
\begin{align}\label{pypyg}
\begin{aligned}
\py^2\Gp|_{y=0}=0.
\end{aligned}
\end{align}
First we give the estimates of $\py^2 Z$.
\begin{Lemma}\label{pypyZdeguji}
The source term $\py^2 Z$ satisfise
\begin{align*}
\|\py^2 Z\|_{H_{\Psi}^{s,0}}
\leq&
C\ep\ltr^{-\ga_0-\frac{5}{4}}
\|\Gp\|_{H_{\Psi}^{s+\frac{2}{3},0}}
+C\ep\ltr^{-\ga_0-\frac{3}{4}}
\|\py\Gp\|_{H_{\Psi}^{s+\frac{2}{3},0}}\\
&+C\ep\ltr^{-\ga_0-\frac{1}{4}}
\|\py^2\Gp\|_{H_{\Psi}^{s+\frac{2}{3},0}}
+C\ep\ltr^{-\ga_0+\frac{1}{4}}
\|\py^3\Gp\|_{H_{\Psi}^{s,0}}.
\end{align*}
\end{Lemma}
\begin{proof}
We estimate them term by term. First use lemma \ref{gu4} and lemma \ref{ub}, we have
\begin{align*}
&\|\py^2\left((\T_u\px G)_{\Phi}-\T_u\px\Gp\right)\|_{H_{\Psi}^{s,0}}\\
\leq&
C\|\py^2 u\|_{L^{\ty}_{\vm}(H_{\h}^{\frac{3}{2}+})}
\|\Gp\|_{H_{\Psi}^{s+\frac{2}{3},0}}
+C\|\py u\|_{L^{\ty}_{\vm}(H_{\h}^{\frac{3}{2}+})}
\|\py\Gp\|_{H_{\Psi}^{s+\frac{2}{3},0}}\\
&+C\|u\|_{L^{\ty}_{\vm}(H_{\h}^{\frac{3}{2}+})}
\|\py^2\Gp\|_{H_{\Psi}^{s+\frac{2}{3},0}}\\
\leq&
C\ep\ltr^{-\ga_0-\frac{5}{4}}
\|\Gp\|_{H_{\Psi}^{s+\frac{2}{3},0}}
+C\ep\ltr^{-\ga_0-\frac{3}{4}}
\|\py\Gp\|_{H_{\Psi}^{s+\frac{2}{3},0}}\\
&+C\ep\ltr^{-\ga_0-\frac{1}{4}}
\|\py^2\Gp\|_{H_{\Psi}^{s+\frac{2}{3},0}}.
\end{align*}
Noticing \eqref{T**} and \eqref{vuG}, we can similarly obtain
\begin{align*}
&\|\py^2\left((\T_{\py G}v)_{\Phi}-\T_{\py G}\vp\right)\|_{H_{\Psi}^{s,0}}\\
\leq&
C\ep\ltr^{-\ga_0-\frac{5}{4}}
\|\Gp\|_{H_{\Psi}^{s+\frac{2}{3},0}}
+C\ep\ltr^{-\ga_0-\frac{3}{4}}
\|\py\Gp\|_{H_{\Psi}^{s+\frac{2}{3},0}}\\
&+C\ep\ltr^{-\ga_0-\frac{1}{4}}
\|\py^2\Gp\|_{H_{\Psi}^{s+\frac{2}{3},0}}.
\end{align*}
By using lemma \ref{gu7}, from a similar derivation it yields
\begin{align*}
&\|\frac{1}{2\ltr}\py^2\left((\T_{\py (y\psi)}v)_{\Phi}-\T_{\py (y\psi)}\vp\right)\|_{H_{\Psi}^{s,0}}\\
\leq&
C\ep\ltr^{-\ga_0-\frac{5}{4}}
\|\Gp\|_{H_{\Psi}^{s+\frac{2}{3},0}}
+C\ep\ltr^{-\ga_0-\frac{3}{4}}
\|\py\Gp\|_{H_{\Psi}^{s+\frac{2}{3},0}}\\
&+C\ep\ltr^{-\ga_0-\frac{1}{4}}
\|\py^2\Gp\|_{H_{\Psi}^{s+\frac{2}{3},0}}.
\end{align*}
On the other hand,
\begin{align*}
&\py^2\left(\frac{y}{\ltr}\int_y^{\ty}
\left((\T_{\py u}v)_{\Phi}-\T_{\py u}\vp\right)dz\right)\\
=&-\frac{2}{\ltr}
\left((\T_{\py u}v)_{\Phi}-\T_{\py u}\vp\right)
-\frac{y}{\ltr}
\py\left((\T_{\py u}v)_{\Phi}-\T_{\py u}\vp\right)
\end{align*}
It follows from lemma \ref{gu6} that
\begin{align}\label{huisushele}
\begin{aligned}
&\|\frac{1}{\ltr}f\|_{H_{\Psi}^{s,0}}
\leq \|\frac{1}{\ltr^{\frac{1}{2}}}\py f\|_{H_{\Psi}^{s,0}}
\leq \|\py^2 f\|_{H_{\Psi}^{s,0}},\\
&~~~~~~~~~~~~~~~~~~~~~~~~
\|\frac{y}{\ltr}f\|_{H_{\Psi}^{s,0}}\leq \|\py f\|_{H_{\Psi}^{s,0}}.
\end{aligned}
\end{align}
By \eqref{huisushele}, from a similar derivation we have
\begin{align*}
&\|\py^2\left(\frac{y}{\ltr}\int_y^{\ty}
\left((\T_{\py u}v)_{\Phi}-\T_{\py u}\vp\right)dz\right)\|_{H_{\Psi}^{s,0}}\\
\leq&
C\ep\ltr^{-\ga_0-\frac{5}{4}}
\|\Gp\|_{H_{\Psi}^{s+\frac{2}{3},0}}
+C\ep\ltr^{-\ga_0-\frac{3}{4}}
\|\py\Gp\|_{H_{\Psi}^{s+\frac{2}{3},0}}\\
&+C\ep\ltr^{-\ga_0-\frac{1}{4}}
\|\py^2\Gp\|_{H_{\Psi}^{s+\frac{2}{3},0}}.
\end{align*}

Notice that if $\lim_{y\rightarrow +\ty}f=0$, then $f=-\int_y^{\ty}\py f dz$. Using \eqref{ty},  we can get
\begin{align}\label{aiyoubucuoou}
\begin{aligned}
\|f\|_{L^{\ty}_{\vm,\Psi}(H_{\h}^{s})}\leq
C\ltr^{\frac{1}{4}}\|\py f\|_{H_{\Psi}^{s,0}}.
\end{aligned}
\end{align}
Then by using lemma \ref{gu1}, lemma \ref{gu7}, \eqref{T**} and
\eqref{aiyoubucuoou}, we can obtain
\begin{align*}
&\|\py^2\left((\T_{\px G}u)_{\Phi}+R^{\h}(u,\px G)_{\Phi}
\right)\|_{H_{\Psi}^{s,0}}\\
\leq&
C\|\py^2 G\|_{L^{\ty}_{\vm,\frac{1}{2}\Psi}(H_{\h}^{\frac{3}{2}+})}
\|\uph\|_{H_{\frac{1}{2}\Psi}^{s,0}}
+C\|\py G\|_{L^{\ty}_{\vm,\frac{1}{2}\Psi}(H_{\h}^{\frac{3}{2}+})}
\|\py\uph\|_{H_{\frac{1}{2}\Psi}^{s,0}}\\
&+C\|G\|_{L^{\ty}_{\vm,\frac{1}{2}\Psi}(H_{\h}^{\frac{3}{2}+})}
\|\py^2\uph\|_{H_{\frac{1}{2}\Psi}^{s,0}}\\
\leq&
C\ltr^{\frac{1}{4}}\|\py^3 G\|_{H_{\Psi}^{\frac{3}{2}+,0}}
\|\Gp\|_{H_{\Psi}^{s,0}}
+C\ltr^{\frac{1}{4}}\|\py^2 G\|_{H_{\Psi}^{\frac{3}{2}+,0}}
\|\py\Gp\|_{H_{\Psi}^{s,0}}\\
&+C\ltr^{\frac{1}{4}}\|\py G\|_{H_{\Psi}^{\frac{3}{2}+,0}}
\|\py^2\Gp\|_{H_{\Psi}^{s,0}}\\
\leq&
C\ep\ltr^{-\ga_0-\frac{5}{4}}
\|\Gp\|_{H_{\Psi}^{s+\frac{2}{3},0}}
+C\ep\ltr^{-\ga_0-\frac{3}{4}}
\|\py\Gp\|_{H_{\Psi}^{s+\frac{2}{3},0}}\\
&+C\ep\ltr^{-\ga_0-\frac{1}{4}}
\|\py^2\Gp\|_{H_{\Psi}^{s+\frac{2}{3},0}}.
\end{align*}
Similarly , by \eqref{vuG} we deduce
\begin{align*}
&\|\py^2\left((\T_{v}\py G)_{\Phi}+R^{\h}(v,\py G)_{\Phi}
\right)\|_{H_{\Psi}^{s,0}}\\
\leq&
C\|\py^2 v\|_{L^{\ty}_{\vm,\frac{1}{2}\Psi}(H_{\h}^{\frac{1}{2}+})}
\|\py\Gp\|_{H_{\frac{1}{2}\Psi}^{s,0}}
+C\|\py v\|_{L^{\ty}_{\vm,\frac{1}{2}\Psi}(H_{\h}^{\frac{1}{2}+})}
\|\py^2\Gp\|_{H_{\frac{1}{2}\Psi}^{s,0}}\\
&+C\|v\|_{L^{\ty}_{\vm,\frac{1}{2}\Psi}(H_{\h}^{\frac{1}{2}+})}
\|\py^3\Gp\|_{H_{\frac{1}{2}\Psi}^{s,0}}\\
\leq&
C\ltr^{\frac{1}{4}}\|\py^2 G\|_{H_{\Psi}^{\frac{3}{2}+,0}}
\|\py\Gp\|_{H_{\Psi}^{s,0}}
+C\ltr^{\frac{1}{4}}\|\py G\|_{H_{\Psi}^{\frac{3}{2}+,0}}
\|\py^2\Gp\|_{H_{\Psi}^{s,0}}\\
&+C\ltr^{\frac{1}{4}}\| G\|_{H_{\Psi}^{\frac{3}{2}+,0}}
\|\py^3\Gp\|_{H_{\Psi}^{s,0}}\\
\leq&
C\ep\ltr^{-\ga_0-\frac{3}{4}}
\|\py\Gp\|_{H_{\Psi}^{s+\frac{2}{3},0}}
+C\ep\ltr^{-\ga_0-\frac{1}{4}}
\|\py^2\Gp\|_{H_{\Psi}^{s+\frac{2}{3},0}}\\
&+C\ep\ltr^{-\ga_0+\frac{1}{4}}
\|\py^3\Gp\|_{H_{\Psi}^{s,0}}.
\end{align*}
Finally, using lemma \ref{gu7} and \eqref{huisushele}, combined with the above estimates we can obtain
\begin{align*}
&\|\frac{1}{2\ltr}\py^2\left((\T_{v}\py (y\psi))_{\Phi}+R^{\h}(v,\py (y\psi))_{\Phi}
\right)\|_{H_{\Psi}^{s,0}}\\
\leq&
C\ep\ltr^{-\ga_0-\frac{3}{4}}
\|\py\Gp\|_{H_{\Psi}^{s+\frac{2}{3},0}}
+C\ep\ltr^{-\ga_0-\frac{1}{4}}
\|\py^2\Gp\|_{H_{\Psi}^{s+\frac{2}{3},0}}\\
&+C\ep\ltr^{-\ga_0+\frac{1}{4}}
\|\py^3\Gp\|_{H_{\Psi}^{s,0}}.
\end{align*}
And
\begin{align*}
&\|\py^2\left(\frac{y}{\ltr}\int_y^{\ty}\left((\T_{v}\py u)_{\Phi}+R^{\h}(v,\py u)_{\Phi}
\right)dz\right)\|_{H_{\Psi}^{s,0}}\\
\leq&
C\ep\ltr^{-\ga_0-\frac{3}{4}}
\|\py\Gp\|_{H_{\Psi}^{s+\frac{2}{3},0}}
+C\ep\ltr^{-\ga_0-\frac{1}{4}}
\|\py^2\Gp\|_{H_{\Psi}^{s+\frac{2}{3},0}}\\
&+C\ep\ltr^{-\ga_0+\frac{1}{4}}
\|\py^3\Gp\|_{H_{\Psi}^{s,0}}.
\end{align*}
By combining all the above estimates, we have completed the proof of lemma \ref{pypyZdeguji}.
\end{proof}

Next is the estimates of $(\py^2 \Gp,\py^2\Gtp)$, and the proposition is shown below:
\begin{Proposition}\label{pypyGdexianyanguji}
Let $\ka\in(0,2)$ be a given constant. Let $(G,\Gt)$ be defined by \eqref{gf}, and $a(t)$ be a non-negative and non-decreasing function on $\R_+$. Then when $\al\leq\ga_0+\frac{5}{36}$, the following inequality holds for any $t<T_*$ and sufficiently small $\eta>0$.
\begin{align}\label{pypyGdexianyanguji2}
\begin{aligned}
&\|\sqrt{a}(\py^2\Gp(t),\py^2\Gtp(t))\|^2_{H_{\Psi}^{3,0}}
-\int_0^t\|\sqrt{a'}(\py^2\Gp,\py^2\Gtp)\|^2_{H_{\Psi}^{3,0}}ds\\
&+2\int_0^t\lsr^{-1}\|\sqrt{a}(\py^2\Gp,\py^2\Gtp)\|^2_{H_{\Psi}^{3,0}}ds
+(4l_{\ka}-\frac{3}{25}\eta-C\ep)\int_0^t\|\sqrt{a}(\py^3\Gp,\py^3\Gtp)\|^2_{H_{\Psi}^{3,0}}ds\\
&+2\left(\la-C(\ep^{\frac{1}{2}}+\eta^{-1}\ep^{\frac{3}{2}})
\right)\int_0^t\tad\|\sqrt{a}(\py^2\Gp,\py^2\Gtp)\|^2_{H_{\Psi}^{\frac{10}{3},0}}ds
\\
\leq&\|\sqrt{a}(0)(\py^2\Gp(0),\py^2\Gtp(0))\|^2_{H_{\Psi}^{3,0}}
+C(\ep^{\frac{1}{2}}+\eta^{-1}\ep^{\frac{3}{2}})\int_0^t\tad
\|\lsr^{-1}\sqrt{a}(\Gp,\Gtp)
\|^2_{H_{\Psi}^{\frac{13}{3},0}}ds\\
&+C\ep^{\frac{1}{2}}\int_0^t\tad
\|\lsr^{-\frac{1}{2}}\sqrt{a}(\py\Gp,\py\Gtp)
\|^2_{H_{\Psi}^{\frac{10}{3},0}}ds
+C\ep\int_0^t
\|\lsr^{-1}\sqrt{a}(\py\Gp,\py\Gtp)
\|^2_{H_{\Psi}^{4,0}}ds,
\end{aligned}
\end{align}
where $l_{\ka}=\frac{\ka(2-\ka)}{4}\in(0,\frac{1}{4}].$
\end{Proposition}
\begin{proof}
By taking $H_{\Psi}^{3,0}$ inner product of \eqref{pypygfephi} with, and using \eqref{ptpy}, we deduce
\begin{align}\label{zaijianchisanxia}
\begin{aligned}
\frac{1}{2}&\frac{d}{dt}\|\sqrt{a}\py^2\Gp\|^2_{H_{\Psi}^{3,0}}
-\frac{1}{2}\|\sqrt{a'}\py^2\Gp\|^2_{H_{\Psi}^{3,0}}
+\ltr^{-1}\|\sqrt{a}\py^2\Gp\|^2_{H_{\Psi}^{3,0}}\\
&+\la\tad(t)\|\sqrt{a}\py^2\Gp\|^2_{H_{\Psi}^{\frac{10}{3},0}}
+\frac{1}{2}\|\sqrt{a}\py^3\Gp\|^2_{H_{\Psi}^{3,0}}\\
\leq&a\left|(\py^2\T_u\px \Gp,\py^2\Gp)_{H_{\Psi}^{3,0}}\right|
+a\left|(\py^2\T_{\py G}\vp,\py^2\Gp)_{H_{\Psi}^{3,0}}\right|\\
&+\frac{a}{2\ltr}\left|(\py^2\T_{\py(y\psi)}\vp,\py^2\Gp
)_{H_{\Psi}^{3,0}}\right|
+a\left|\left(\py^2(\frac{y}{\ltr}\int_y^{\ty}\T_{\py u}\vp dz),\py^2\Gp\right)_{H_{\Psi}^{3,0}}\right|\\
&+a\tad\left|(\py^2\Pp,\py^2\Gp)_{H_{\Psi}^{3,0}}\right|
+a\tad\left|\left(\py^2(\frac{y}{2\ltr}\int_y^{\ty}\Pp dz),\py^2\Gp\right)_{H_{\Psi}^{3,0}}\right|\\
&+a\left|(\py^2 Z,\py^2\Gp)_{H_{\Psi}^{3,0}}\right|\\
=&\sum_{i=1}^{7} R_i.
\end{aligned}
\end{align}
For $R_1$, notice
\begin{align*}
\py^2\T_u\px \Gp=
\T_{\py^2u}\px\Gp+2\T_{\py u}\py\px \Gp+\T_u\px \py^2\Gp.
\end{align*}
By integrating by parts, lemma \ref{gu1}, lemma \ref{gu3}, lemma
\ref{ub} and \eqref{xinqingbucuo}, we have
\begin{align*}
R_1\leq&C a
\left|(\T_{\py^2u}\px\Gp,\py^2\Gp)_{H_{\Psi}^{3,0}}\right|
+C a\left|(\T_{\py u}\py\px \Gp,\py^2\Gp)_{H_{\Psi}^{3,0}}\right|\\
&+C a\left|(\T_u\px \py^2\Gp,\py^2\Gp)_{H_{\Psi}^{3,0}}\right|\\
\leq&C a
\left|(\T_{\py^2u}\px\Gp,\py^2\Gp)_{H_{\Psi}^{3,0}}\right|
+C a\left|(\T_{\py u}\px \Gp,\py^3\Gp)_{H_{\Psi}^{3,0}}\right|\\
&+C a\left|(\T_{\py u}\px \Gp,\frac{y}{2\ltr}\py^2\Gp)_{H_{\Psi}^{3,0}}\right|
+C a\left|(\T_u\px \py^2\Gp,\py^2\Gp)_{H_{\Psi}^{3,0}}\right|\\
\leq&C\|\py^2u\|_{L^{\ty}_{\vm}(H_{\h}^{\frac{1}{2}+})}
\|\sqrt{a}\Gp\|_{H_{\Psi}^{4,0}}\|\sqrt{a}\py^2\Gp\|_{H_{\Psi}^{3,0}}
+C\|u\|_{L^{\ty}_{\vm}(H_{\h}^{\frac{3}{2}+})}
\|\sqrt{a}\py^2\Gp\|^2_{H_{\Psi}^{3,0}}
\\
&+C\|\py u\|_{L^{\ty}_{\vm}(H_{\h}^{\frac{1}{2}+})}
\|\sqrt{a}\Gp\|_{H_{\Psi}^{4,0}}\|\sqrt{a}\py^3\Gp\|_{H_{\Psi}^{3,0}}\\
\leq&C\ep^{\frac{1}{2}}\ltr^{\al-\ga_0-\frac{1}{4}}\tad
\|\ltr^{-1}\sqrt{a}\Gp\|_{H_{\Psi}^{\frac{13}{3},0}}
\|\sqrt{a}\py^2\Gp\|_{H_{\Psi}^{\frac{10}{3},0}}
+C\ep^{\frac{1}{2}}\ltr^{\al-\ga_0-\frac{1}{4}}\tad
\|\sqrt{a}\py^2\Gp\|^2_{H_{\Psi}^{\frac{10}{3},0}}
\\
&+C\ep\ltr^{-\ga_0+\frac{1}{4}}
\|\ltr^{-1}\sqrt{a}\Gp\|_{H_{\Psi}^{\frac{13}{3},0}}
\|\sqrt{a}\py^3\Gp\|_{H_{\Psi}^{3,0}}\\
\leq&C\ep^{\frac{1}{2}}\tad
\|\ltr^{-1}\sqrt{a}\Gp\|^2_{H_{\Psi}^{\frac{13}{3},0}}
+C\ep^{\frac{1}{2}}\tad
\|\sqrt{a}\py^2\Gp\|^2_{H_{\Psi}^{\frac{10}{3},0}}
+\frac{1}{100}\eta\|\sqrt{a}\py^3\Gp\|^2_{H_{\Psi}^{3,0}}\\
&+C\eta^{-1}\ep^{\frac{3}{2}}\ltr^{\al-2\ga_0+\frac{1}{2}}\tad
\|\ltr^{-1}\sqrt{a}\Gp\|^2_{H_{\Psi}^{\frac{13}{3},0}}\\
\leq&C(\ep^{\frac{1}{2}}+\eta^{-1}\ep^{\frac{3}{2}})\tad
\|\ltr^{-1}\sqrt{a}\Gp\|^2_{H_{\Psi}^{\frac{13}{3},0}}
+C\ep^{\frac{1}{2}}\tad
\|\sqrt{a}\py^2\Gp\|^2_{H_{\Psi}^{\frac{10}{3},0}}
+\frac{1}{100}\eta\|\sqrt{a}\py^3\Gp\|^2_{H_{\Psi}^{3,0}},
\end{align*}
here we use $\al\leq\ga_0+\frac{5}{36}$ and $\ga_0>1.$

For $R_2$, notice
\begin{align*}
\py^2\T_{\py G}\vp=
\T_{\py^3 G}\vp+2\T_{\py^2 G}\py\vp+\T_{\py G}\py^2\vp.
\end{align*}
By integrating by parts, we can deduce
\begin{align*}
R_2=&a\left|(T_{\py^3 G}\vp+\T_{\py^2 G}\py\vp,\py^2\Gp)_{H_{\Psi}^{3,0}}-
(\T_{\py G}\py\vp,\py^3\Gp+\frac{y}{2\ltr}\py^2\Gp)_{H_{\Psi}^{3,0}}
\right|.
\end{align*}
Next applying lemma \ref{gu1}, \eqref{T**} and \eqref{vuG}, we can derive
\begin{align*}
R_2\leq&
\left(C\|\py^3 G\|_{H_{\Psi}^{\frac{1}{2}+,0}}
\|\sqrt{a}\vp\|_{L^{\ty}_{\vm}(H_{\h}^{3})}
+C\|\py^2 G\|_{H_{\Psi}^{\frac{1}{2}+,0}}
\|\sqrt{a}\py\vp\|_{L^{\ty}_{\vm}(H_{\h}^{3})}\right)
\|\sqrt{a}\py^2 G\|_{H_{\Psi}^{3,0}}\\
&+C\|\py G\|_{H_{\Psi}^{\frac{1}{2}+,0}}
\|\sqrt{a}\py\vp\|_{L^{\ty}_{\vm}(H_{\h}^{3})}
\|\sqrt{a}\py^3 G\|_{H_{\Psi}^{3,0}}\\
\leq&C\ep\ltr^{-\ga_0-\frac{5}{4}}
\|\sqrt{a}\Gp\|_{H_{\Psi}^{4,0}}\|\sqrt{a}\py^2 G\|_{H_{\Psi}^{3,0}}
+C\ep\ltr^{-\ga_0-\frac{3}{4}}
\|\sqrt{a}\py\Gp\|_{H_{\Psi}^{4,0}}\|\sqrt{a}\py^2 G\|_{H_{\Psi}^{3,0}}\\
&+C\ep\ltr^{-\ga_0-\frac{1}{4}}
\|\sqrt{a}\py\Gp\|_{H_{\Psi}^{4,0}}
\|\sqrt{a}\py^3 G\|_{H_{\Psi}^{3,0}}\\
\leq&C\ep^{\frac{1}{2}}\ltr^{\al-\ga_0-\frac{1}{4}}\tad
\|\ltr^{-1}\sqrt{a}\Gp\|^2_{H_{\Psi}^{\frac{13}{3},0}}
+C\ep^{\frac{1}{2}}\ltr^{\al-\ga_0-\frac{1}{4}}\tad
\|\sqrt{a}\py^2 G\|^2_{H_{\Psi}^{\frac{10}{3},0}}\\
&+C\ep
\|\ltr^{-1}\sqrt{a}\py\Gp\|^2_{H_{\Psi}^{4,0}}
+C\ep^{\frac{1}{2}}\ltr^{\al-2\ga_0+\frac{1}{4}}\tad
\|\sqrt{a}\py^2 G\|^2_{H_{\Psi}^{\frac{10}{3},0}}\\
&+C\ep\ltr^{-\ga_0+\frac{3}{4}}
\|\ltr^{-1}\sqrt{a}\py\Gp\|^2_{H_{\Psi}^{4,0}}
+C\ep\ltr^{-\ga_0+\frac{3}{4}}
\|\sqrt{a}\py^3 G\|^2_{H_{\Psi}^{3,0}}.
\end{align*}
Using the fact $\al\leq\ga_0+\frac{5}{36}$ and $\ga_0>1$, we can deduce
\begin{align*}
R_2
\leq&C\ep^{\frac{1}{2}}\tad
\|\ltr^{-1}\sqrt{a}\Gp\|^2_{H_{\Psi}^{\frac{13}{3},0}}
+C\ep^{\frac{1}{2}}\tad
\|\sqrt{a}\py^2 G\|^2_{H_{\Psi}^{\frac{10}{3},0}}\\
&+C\ep
\|\ltr^{-1}\sqrt{a}\py\Gp\|^2_{H_{\Psi}^{4,0}}
+C\ep
\|\sqrt{a}\py^3 G\|^2_{H_{\Psi}^{3,0}}.
\end{align*}
For $R_3,$ using lemma \ref{gu7}, \eqref{T**} and the estimates of $R_2$, it gives
\begin{align*}
R_3
\leq&C\ep^{\frac{1}{2}}\tad
\|\ltr^{-1}\sqrt{a}\Gp\|^2_{H_{\Psi}^{\frac{13}{3},0}}
+C\ep^{\frac{1}{2}}\tad
\|\sqrt{a}\py^2 G\|^2_{H_{\Psi}^{\frac{10}{3},0}}\\
&+C\ep
\|\ltr^{-1}\sqrt{a}\py\Gp\|^2_{H_{\Psi}^{4,0}}
+C\ep
\|\sqrt{a}\py^3 G\|^2_{H_{\Psi}^{3,0}}.
\end{align*}
For $R_4,$ noticing
\begin{align*}
\py^2(\frac{y}{\ltr}\int_y^{\ty}\T_{\py u}\vp dz)
=-\frac{2}{\ltr}\T_{\py u}\vp
-\frac{y}{\ltr}\py\T_{\py u}\vp,
\end{align*}
and combining \eqref{xinqingbucuo}, then we can use a similar derivation of $R_2$ to get the estimates of $R_4$.
\begin{align*}
R_4
\leq&C\ep^{\frac{1}{2}}\tad
\|\ltr^{-1}\sqrt{a}\Gp\|^2_{H_{\Psi}^{\frac{13}{3},0}}
+C\ep^{\frac{1}{2}}\tad
\|\sqrt{a}\py^2 G\|^2_{H_{\Psi}^{\frac{10}{3},0}}\\
&+C\ep
\|\ltr^{-1}\sqrt{a}\py\Gp\|^2_{H_{\Psi}^{4,0}}
+C\ep
\|\sqrt{a}\py^3 G\|^2_{H_{\Psi}^{3,0}},
\end{align*}
For the sake of brevity, no detailed calculation is given here.

For $R_5$, by \eqref{PN} we have
\begin{align*}
\tad\py\Pp=\py(b\px b+h\py b)_{\Phi}=(b\py\px b+h\py^2 b)_{\Phi}.
\end{align*}
Applying \eqref{bony}, lemma \ref{gu1}, lemma \ref{gu7} and lemma
\ref{ub}, we can deduce
\begin{align}\label{haoleia}
\begin{aligned}
\tad\|\py\Pp\|_{H_{\Psi}^{s,0}}\leq&
\|\bp\|_{L^{\ty}_{\vm,\frac{1}{2}\Psi}(H_{\h}^{\frac{1}{2}+})}
\|\py \bp\|_{H_{\frac{1}{2}\Psi}^{s+1,0}}
+\|\py\bp\|_{L^{\ty}_{\vm,\frac{1}{2}\Psi}(H_{\h}^{\frac{3}{2}+})}
\|\bp\|_{H_{\frac{1}{2}\Psi}^{s,0}}\\
&+\|\hp\|_{L^{\ty}_{\vm,\frac{1}{2}\Psi}(H_{\h}^{\frac{1}{2}+})}
\|\py^2 \bp\|_{H_{\frac{1}{2}\Psi}^{s,0}}
+\|\py^2 \bp\|_{L^{2}_{\vm,\frac{1}{2}\Psi}(H_{\h}^{\frac{1}{2}+})}
\|h\|_{L^{\ty}_{\vm,\frac{1}{2}\Psi}(H_{\h}^{s})}\\
\leq&C\ep\ltr^{-\ga_0-\frac{1}{4}}\|\py\Gtp\|_{H_{\Psi}^{s+1,0}}
+C\ep\ltr^{-\ga_0-\frac{3}{4}}\|\Gtp\|_{H_{\Psi}^{s,0}}\\
&+C\ep\ltr^{-\ga_0+\frac{1}{4}}\|\py^2\Gtp\|_{H_{\Psi}^{s,0}}
+C\ep\ltr^{-\ga_0-\frac{3}{4}}\|\Gtp\|_{H_{\Psi}^{s+1,0}}\\
\leq&C\ep\ltr^{-\ga_0-\frac{1}{4}}\|\py\Gtp\|_{H_{\Psi}^{s+1,0}}
+C\ep\ltr^{-\ga_0-\frac{3}{4}}\|\Gtp\|_{H_{\Psi}^{s+1,0}}\\
&+C\ep\ltr^{-\ga_0+\frac{1}{4}}\|\py^2\Gtp\|_{H_{\Psi}^{s,0}}.
\end{aligned}
\end{align}
Thus by integrating by parts, \eqref{xinqingbucuo} and
\eqref{haoleia}, it's easy to get
\begin{align*}
R_5\leq&a\tad\|\py\Pp\|_{H_{\Psi}^{3,0}}
\|\py^3\Gp\|_{H_{\Psi}^{3,0}}\\
\leq&C\ep\ltr^{-\ga_0-\frac{1}{4}}\|\sqrt{a}\py\Gtp\|_{H_{\Psi}^{4,0}}
\|\sqrt{a}\py^3\Gp\|_{H_{\Psi}^{3,0}}
+C\ep\ltr^{-\ga_0-\frac{3}{4}}\|\sqrt{a}\Gtp\|_{H_{\Psi}^{4,0}}
\|\sqrt{a}\py^3\Gp\|_{H_{\Psi}^{3,0}}\\
&+C\ep\ltr^{-\ga_0+\frac{1}{4}}\|\sqrt{a}\py^2\Gtp\|_{H_{\Psi}^{3,0}}
\|\sqrt{a}\py^3\Gp\|_{H_{\Psi}^{3,0}}\\
\leq&C\ep\ltr^{-\ga_0+\frac{3}{4}}\|\ltr^{-1}\sqrt{a}\py\Gtp\|^2_{H_{\Psi}^{4,0}}
+C\ep\ltr^{-\ga_0+\frac{3}{4}}\|\sqrt{a}\py^3\Gp\|^2_{H_{\Psi}^{3,0}}\\
&+C\eta^{-1}\ep^{\frac{3}{2}}\ltr^{\al-2\ga_0+\frac{1}{2}}\tad
\|\ltr^{-1}\sqrt{a}\Gtp\|^2_{H_{\Psi}^{\frac{13}{3},0}}
+\frac{1}{100}\eta\|\sqrt{a}\py^3\Gp\|^2_{H_{\Psi}^{3,0}}\\
&+C\eta^{-1}\ep^{\frac{3}{2}}\ltr^{\al-2\ga_0+\frac{1}{2}}\tad
\|\sqrt{a}\py^2\Gtp\|^2_{H_{\Psi}^{\frac{10}{3},0}}
+\frac{1}{100}\eta\|\sqrt{a}\py^3\Gp\|^2_{H_{\Psi}^{3,0}}.
\end{align*}
Using the fact that $\al\leq\ga_0+\frac{5}{36}$ and $\ga_0>1$, we have
\begin{align*}
R_5
\leq &C\ep\|\ltr^{-1}\sqrt{a}\py\Gtp\|^2_{H_{\Psi}^{4,0}}
+C\eta^{-1}\ep^{\frac{3}{2}}\tad
\|\ltr^{-1}\sqrt{a}\Gtp\|^2_{H_{\Psi}^{\frac{13}{3},0}}\\
&+C\eta^{-1}\ep^{\frac{3}{2}}\tad
\|\sqrt{a}\py^2\Gtp\|^2_{H_{\Psi}^{\frac{10}{3},0}}
+(C\ep+\frac{1}{50}\eta)\|\sqrt{a}\py^3\Gp\|^2_{H_{\Psi}^{3,0}}.
\end{align*}
For $R_6,$ notice that
\begin{align*}
\py(\frac{y}{\ltr}\int_y^{\ty}\Pp dz)=
\frac{1}{\ltr}\int_y^{\ty}\Pp dz-\frac{y}{\ltr}\Pp.
\end{align*}
Then by using \eqref{huisushele} and the estimates of $R_5$, we can know that $R_6$ satisfies
\begin{align*}
R_6
\leq &C\ep\|\ltr^{-1}\sqrt{a}\py\Gtp\|^2_{H_{\Psi}^{4,0}}
+C\eta^{-1}\ep^{\frac{3}{2}}\tad
\|\ltr^{-1}\sqrt{a}\Gtp\|^2_{H_{\Psi}^{\frac{13}{3},0}}\\
&+C\eta^{-1}\ep^{\frac{3}{2}}\tad
\|\sqrt{a}\py^2\Gtp\|^2_{H_{\Psi}^{\frac{10}{3},0}}
+(C\ep+\frac{1}{50}\eta)\|\sqrt{a}\py^3\Gp\|^2_{H_{\Psi}^{3,0}}.
\end{align*}
For $R_7$, applying lemma \ref{pypyZdeguji}, we have
\begin{align*}
R_7\leq&C\|\sqrt{a}\py^2 Z\|_{H_{\Psi}^{\frac{8}{3},0}}
\|\sqrt{a}\py^2 \Gp\|_{H_{\Psi}^{\frac{10}{3},0}}
\\
\leq&
C\ep\ltr^{-\ga_0-\frac{5}{4}}
\|\sqrt{a}\Gp\|_{H_{\Psi}^{\frac{10}{3},0}}
\|\sqrt{a}\py^2 \Gp\|_{H_{\Psi}^{\frac{10}{3},0}}
+C\ep\ltr^{-\ga_0-\frac{3}{4}}
\|\sqrt{a}\py\Gp\|_{H_{\Psi}^{\frac{10}{3},0}}
\|\sqrt{a}\py^2 \Gp\|_{H_{\Psi}^{\frac{10}{3},0}}\\
&+C\ep\ltr^{-\ga_0-\frac{1}{4}}
\|\sqrt{a}\py^2\Gp\|_{H_{\Psi}^{\frac{10}{3},0}}
\|\sqrt{a}\py^2 \Gp\|_{H_{\Psi}^{\frac{10}{3},0}}
+C\ep\ltr^{-\ga_0+\frac{1}{4}}
\|\sqrt{a}\py^3\Gp\|_{H_{\Psi}^{\frac{8}{3},0}}
\|\sqrt{a}\py^2 \Gp\|_{H_{\Psi}^{\frac{10}{3},0}}\\
\leq&C\ep^{\frac{1}{2}}\tad
\|\ltr^{-1}\sqrt{a}\Gp\|^2_{H_{\Psi}^{\frac{13}{3},0}}
+C\ep^{\frac{1}{2}}\tad
\|\ltr^{-\frac{1}{2}}\sqrt{a}\py\Gp\|^2_{H_{\Psi}^{\frac{10}{3},0}}\\
&+C(\ep^{\frac{1}{2}}+\eta^{-1}\ep^{\frac{3}{2}})\tad
\|\sqrt{a}\py^2\Gp\|^2_{H_{\Psi}^{\frac{10}{3},0}}
+\frac{1}{100}\eta\|\sqrt{a}\py^3\Gp\|^2_{H_{\Psi}^{3,0}},
\end{align*}
here we use $\al\leq\ga_0+\frac{5}{36}$ and $\ga_0>1.$

Taking $R_1-R_7$ into \eqref{zaijianchisanxia}, and integrating over $[0,t]$, we can deduce the estimates of $\py^2\Gp$.
\begin{align*}
&\|\sqrt{a}\py^2\Gp(t)\|^2_{H_{\Psi}^{3,0}}
-\int_0^t\|\sqrt{a'}\py^2\Gp\|^2_{H_{\Psi}^{3,0}}ds
+2\int_0^t\lsr^{-1}\|\sqrt{a}\py^2\Gp\|^2_{H_{\Psi}^{3,0}}ds\\
&+2\la\int_0^t\tad\|\sqrt{a}\py^2\Gp\|^2_{H_{\Psi}^{\frac{10}{3},0}}ds
+\int_0^t\|\sqrt{a}\py^3\Gp\|^2_{H_{\Psi}^{3,0}}ds\\
\leq&\|\sqrt{a}(0)\py^2\Gp(0)\|^2_{H_{\Psi}^{3,0}}
+C(\ep^{\frac{1}{2}}+\eta^{-1}\ep^{\frac{3}{2}})
\int_0^t\tad
\|\sqrt{a}(\py^2\Gp,\py^2\Gtp)\|^2_{H_{\Psi}^{\frac{10}{3},0}}ds\\
&+(C\ep+\frac{3}{25}\eta)\int_0^t\|\sqrt{a}\py^3\Gp\|^2_{H_{\Psi}^{3,0}}ds
+C(\ep^{\frac{1}{2}}+\eta^{-1}\ep^{\frac{3}{2}})\int_0^t\tad
\|\lsr^{-1}\sqrt{a}(\Gp,\Gtp)
\|^2_{H_{\Psi}^{\frac{13}{3},0}}ds\\
&+C\ep^{\frac{1}{2}}\int_0^t\tad
\|\lsr^{-\frac{1}{2}}\sqrt{a}\py\Gp
\|^2_{H_{\Psi}^{\frac{10}{3},0}}ds
+C\ep\int_0^t
\|\lsr^{-1}\sqrt{a}(\py\Gp,\py\Gtp)
\|^2_{H_{\Psi}^{4,0}}ds.
\end{align*}
This completes the estimates of $\py^2 \Gp$. And the estimates of $\py^2\Gtp$ can be obtained in the same way, so we have completed the proof of proposition \ref{pypyGdexianyanguji}.

\end{proof}

Using proposition \ref{pypyGdexianyanguji}, we can prove proposition \ref{propositionpy2G}.

\textbf{Proof of Proposition \ref{propositionpy2G}:}
Taking $a(t)=\ltr^{4+2l_{\ka}-2\eta}$ in \eqref{pypyGdexianyanguji2}, then we have
\begin{align*}
&\|\ltr^{2+l_{\ka}-\eta}(\py^2\Gp(t),\py^2\Gtp(t))
\|^2_{H_{\Psi}^{3,0}}
+(4l_{\ka}-\frac{3}{25}\eta-C\ep)\int_0^t\|\lsr^{2+l_{\ka}-\eta}
(\py^3\Gp,\py^3\Gtp)\|^2_{H_{\Psi}^{3,0}}ds\\
&+2\left(\la-C(\ep^{\frac{1}{2}}+\eta^{-1}\ep^{\frac{3}{2}})
\right)\int_0^t\tad\|\lsr^{2+l_{\ka}-\eta}
(\py^2\Gp,\py^2\Gtp)\|^2_{H_{\Psi}^{\frac{10}{3},0}}ds
\\
\leq&\|(\py^2\Gp(0),\py^2\Gtp(0))\|^2_{H_{\Psi}^{3,0}}
+C(\ep^{\frac{1}{2}}+\eta^{-1}\ep^{\frac{3}{2}})\int_0^t\tad
\|\lsr^{1+l_{\ka}-\eta}(\Gp,\Gtp)
\|^2_{H_{\Psi}^{\frac{13}{3},0}}ds\\
&+C\ep^{\frac{1}{2}}\int_0^t\tad
\|\lsr^{\frac{3}{2}+l_{\ka}-\eta}(\py\Gp,\py\Gtp)
\|^2_{H_{\Psi}^{\frac{10}{3},0}}ds
+C\ep\int_0^t
\|\lsr^{1+l_{\ka}-\eta}(\py\Gp,\py\Gtp)
\|^2_{H_{\Psi}^{4,0}}ds\\
&+(2+2l_{\ka}-2\eta)\int_0^t\|\lsr^{\frac{3}{2}+l_{\ka}-\eta}
(\py^2\Gp,\py^2\Gtp)\|^2_{H_{\Psi}^{3,0}}ds.
\end{align*}
By \eqref{propositionG1} and \eqref{propositionpyG1}, we can deduce
\begin{align*}
&\|\ltr^{2+l_{\ka}-\eta}(\py^2\Gp(t),\py^2\Gtp(t))
\|^2_{H_{\Psi}^{3,0}}
+(4l_{\ka}-\frac{3}{25}\eta-C\ep)\int_0^t\|\lsr^{2+l_{\ka}-\eta}
(\py^3\Gp,\py^3\Gtp)\|^2_{H_{\Psi}^{3,0}}ds\\
&+2\left(\la-C(\ep^{\frac{1}{2}}+\eta^{-1}\ep^{\frac{3}{2}})
\right)\int_0^t\tad\|\lsr^{2+l_{\ka}-\eta}
(\py^2\Gp,\py^2\Gtp)\|^2_{H_{\Psi}^{\frac{10}{3},0}}ds
\\
\leq&
\dfrac{C}{\eta(4l_{\ka}-\frac{7}{25}\eta)}
\left(\|(\py\Gp(0),\py\Gtp(0))\|^2_{H_{\Psi}^{3,0}}
+\|(\Gp(0),\Gtp(0))\|^2_{H_{\Psi}^{4,0}}
+\|(\py^2\Gp(0),\py^2\Gtp(0))\|^2_{H_{\Psi}^{3,0}}\right)\\
&+\dfrac{C}{\eta(4l_{\ka}-\frac{7}{25}\eta)}\int_0^t\tad
\|\lsr^{l_{\ka}-\eta}\sqrt{a}(\uph,\bp)\|^{2}_{H_{\Psi}^{\frac{22}{3},0}}
ds.
\end{align*}

Thus when $\ep<\ep_3$ and $\la>\la_2$, proposition \ref{propositionpy2G} holds.

\hfill $\square$
\section{The Gevrey Estimates of $\py^3\Gp$ and $\py^3\Gtp$}\label{gujipy3G}
In this section we come to estimate the Gevrey norm of $(\py^3\Gp,\py^3\Gtp)$, while in \cite{CW} only the Sobolev norm is estimated. The readers can see the description in section \ref{jiusuibianbianhao} for the reason.

Acting $\py$ on \eqref{pypygfephi}, we can obtain
\begin{align}\label{pypypygfephi}
\begin{aligned}
&~~~~~\pt\py^3\Gp+\la\tad(t)[D_x]^{\frac{2}{3}}\py^3\Gp-\py^5\Gp
+\ltr^{-1}\py^3\Gp+\py\mathfrak{F}=\py^3 Z,
\end{aligned}
\end{align}
where $\mathfrak{F}$ is defined as follows
\begin{align}\label{FF}
\begin{aligned}
\mathfrak{F}=&\py^2\left(\T_u\px\Gp
+\T_{\py G}\vp
-\frac{1}{2\ltr}\T_{\py(y\psi)}\vp+\frac{y}{\ltr}\int_y^{\ty}\T_{\py u}\vp dz-\tad\Pp\right.\\
&\left.+\frac{y}{2\ltr}\int_y^{\ty}\tad\Pp dz\right).
\end{aligned}
\end{align}
By \eqref{pypygfephi} and \eqref{pypyg}, we can know the boundary condition
\begin{align}\label{zuihoudebianjietiaojian}
\begin{aligned}
(-\py^4\Gp+\mathfrak{F}-\py^2Z)_{y=0}=0.
\end{aligned}
\end{align}
First we give the estimates of $\mathfrak{F}$.
\begin{Lemma}\label{FFdeguji}
Let $\mathfrak{F}$ be defined by \eqref{FF}. Then for any $t<T_*,$ $\mathfrak{F}$ satisfies
\begin{align*}
\|\mathfrak{F}\|_{H_{\Psi}^{s,0}}
\leq&C\ep\ltr^{-\ga_0+\frac{1}{4}}
\|(\py^3\Gp,\py^3\Gtp)\|_{H_{\Psi}^{s,0}}
+C\ep\ltr^{-\ga_0-\frac{1}{4}}
\|(\py^2\Gp,\py^2\Gtp)\|_{H_{\Psi}^{s+1,0}}\\
&+C\ep\ltr^{-\ga_0-\frac{3}{4}}
\|(\py\Gp,\py\Gtp)\|_{H_{\Psi}^{s+1,0}}
+C\ep\ltr^{-\ga_0-\frac{5}{4}}
\|(\Gp,\Gtp)\|_{H_{\Psi}^{s+1,0}}.
\end{align*}
\end{Lemma}
\begin{proof}
We estimate them term by term. First of all, by lemma \ref{gu1} and lemma \ref{ub}, we can know that
\begin{align*}
\|\py^2\T_{u}\px\Gp\|_{H_{\Psi}^{s,0}}\leq&
C\|\T_{\py^2 u}\px\Gp\|_{H_{\Psi}^{s,0}}
+C\|\T_{\py u}\py\px\Gp\|_{H_{\Psi}^{s,0}}
+C\|\T_{ u}\py^2\px\Gp\|_{H_{\Psi}^{s,0}}\\
\leq&C\|\py^2 u\|_{L^{\ty}_{\vm}(H_{\h}^{\frac{1}{2}+})}
\|\Gp\|_{H_{\Psi}^{s+1,0}}
+C\|\py u\|_{L^{\ty}_{\vm}(H_{\h}^{\frac{1}{2}+})}
\|\py\Gp\|_{H_{\Psi}^{s+1,0}}\\
&+C\| u\|_{L^{\ty}_{\vm}(H_{\h}^{\frac{1}{2}+})}
\|\py^2\Gp\|_{H_{\Psi}^{s+1,0}}\\
\leq&C\ep\ltr^{-\ga_0-\frac{5}{4}}
\|\Gp\|_{H_{\Psi}^{s+1,0}}
+C\ep\ltr^{-\ga_0-\frac{3}{4}}
\|\py\Gp\|_{H_{\Psi}^{s+1,0}}\\
&+C\ep\ltr^{-\ga_0-\frac{1}{4}}
\|\py^2\Gp\|_{H_{\Psi}^{s+1,0}}.
\end{align*}
By applying lemma \ref{gu1}, \eqref{T**} and \eqref{vuG}, we derive that
\begin{align*}
\|\py^2\T_{\py G}\vp\|_{H_{\Psi}^{s,0}}\leq&
C\|\T_{\py^3 G}\vp\|_{H_{\Psi}^{s,0}}
+C\|\T_{\py^2 G}\py\vp\|_{H_{\Psi}^{s,0}}
+C\|\T_{\py G}\py^2\vp\|_{H_{\Psi}^{s,0}}\\
\leq&C\|\py^3 G\|_{H_{\Psi}^{\frac{1}{2}+,0}}
\|\vp\|_{L^{\ty}_{\vm}(H_{\h}^{s})}
+C\|\py^2 G\|_{H_{\Psi}^{\frac{1}{2}+,0}}
\|\py\vp\|_{L^{\ty}_{\vm}(H_{\h}^{s})}\\
&+C\|\py G\|_{H_{\Psi}^{\frac{1}{2}+,0}}
\|\py^2\vp\|_{L^{\ty}_{\vm}(H_{\h}^{s})}\\
\leq&C\ep\ltr^{-\ga_0-\frac{5}{4}}
\|\Gp\|_{H_{\Psi}^{s+1,0}}
+C\ep\ltr^{-\ga_0-\frac{3}{4}}
\|\py\Gp\|_{H_{\Psi}^{s+1,0}}\\
&+C\ep\ltr^{-\ga_0-\frac{1}{4}}
\|\py^2\Gp\|_{H_{\Psi}^{s+1,0}}.
\end{align*}
Using lemma \ref{gu7} and \eqref{huisushele}, combined with the above estimates, we can similarly obtain
\begin{align*}
\|\frac{1}{2\ltr}\py^2\T_{\py (y\psi)}\vp\|_{H_{\Psi}^{s,0}}
\leq&C\ep\ltr^{-\ga_0-\frac{5}{4}}
\|\Gp\|_{H_{\Psi}^{s+1,0}}
+C\ep\ltr^{-\ga_0-\frac{3}{4}}
\|\py\Gp\|_{H_{\Psi}^{s+1,0}}\\
&+C\ep\ltr^{-\ga_0-\frac{1}{4}}
\|\py^2\Gp\|_{H_{\Psi}^{s+1,0}}.
\end{align*}
And
\begin{align*}
\|\py^2\left(\frac{y}{\ltr}\int_y^{\ty}\T_{\py u}\vp dz\right)\|_{H_{\Psi}^{s,0}}
\leq&C\ep\ltr^{-\ga_0-\frac{5}{4}}
\|\Gp\|_{H_{\Psi}^{s+1,0}}
+C\ep\ltr^{-\ga_0-\frac{3}{4}}
\|\py\Gp\|_{H_{\Psi}^{s+1,0}}\\
&+C\ep\ltr^{-\ga_0-\frac{1}{4}}
\|\py^2\Gp\|_{H_{\Psi}^{s+1,0}}.
\end{align*}
It follows from \eqref{PN} that
\begin{align*}
\tad\py^2\Pp=(\py b\px\py b+b\px\py^2 b
-\px b\py^2 b+h\py^3 b)_{\Phi}.
\end{align*}
Use \eqref{bony} to estimate them term by term,
\begin{align*}
\|(\py b\px\py b)_{\Phi}\|_{H_{\Psi}^{s,0}}\leq&
C\|\py\bp\|_{L^{\ty}_{\vm}(H_{\h}^{\frac{1}{2}+})}
\|\py\bp\|_{H_{\Psi}^{s+1,0}}+
C\|\py\bp\|_{L^{\ty}_{\vm}(H_{\h}^{\frac{3}{2}+})}
\|\py\bp\|_{H_{\Psi}^{s,0}}\\
\leq&C\ep\ltr^{-\ga_0-\frac{3}{4}}
\|\py\Gtp\|_{H_{\Psi}^{s+1,0}},
\end{align*}
\begin{align*}
\|(b\px\py^2 b)_{\Phi}\|_{H_{\Psi}^{s,0}}\leq&
C\|\bp\|_{L^{\ty}_{\vm}(H_{\h}^{\frac{1}{2}+})}
\|\py^2\bp\|_{H_{\Psi}^{s+1,0}}+
C\|\py^2\bp\|_{L^{\ty}_{\vm}(H_{\h}^{\frac{3}{2}+})}
\|\bp\|_{H_{\Psi}^{s,0}}\\
\leq&C\ep\ltr^{-\ga_0-\frac{1}{4}}
\|\py^2\Gtp\|_{H_{\Psi}^{s+1,0}}
+C\ep\ltr^{-\ga_0-\frac{5}{4}}
\|\Gtp\|_{H_{\Psi}^{s+1,0}},
\end{align*}
\begin{align*}
\|(\px b\py^2 b)_{\Phi}\|_{H_{\Psi}^{s,0}}\leq&
C\|\bp\|_{L^{\ty}_{\vm}(H_{\h}^{\frac{3}{2}+})}
\|\py^2\bp\|_{H_{\Psi}^{s,0}}+
C\|\py^2\bp\|_{L^{\ty}_{\vm}(H_{\h}^{\frac{1}{2}+})}
\|\bp\|_{H_{\Psi}^{s+1,0}}\\
\leq&C\ep\ltr^{-\ga_0-\frac{1}{4}}
\|\py^2\Gtp\|_{H_{\Psi}^{s+1,0}}
+C\ep\ltr^{-\ga_0-\frac{5}{4}}
\|\Gtp\|_{H_{\Psi}^{s+1,0}}.
\end{align*}
Noticing \eqref{vuG} also holds for $(\hp,\bp,\Gtp)$, thus we have
\begin{align*}
\|(h\py^3 b)_{\Phi}\|_{H_{\Psi}^{s,0}}\leq&
C\|\hp\|_{L^{\ty}_{\vm}(H_{\h}^{\frac{1}{2}+})}
\|\py^3\bp\|_{H_{\Psi}^{s,0}}+
C\|\py^3\bp\|_{H_{\Psi}^{\frac{1}{2}+,0}}
\|\hp\|_{L^{\ty}_{\vm}(H_{\h}^{s})}\\
\leq&C\ep\ltr^{-\ga_0+\frac{1}{4}}
\|\py^3\Gtp\|_{H_{\Psi}^{s,0}}
+C\ep\ltr^{-\ga_0-\frac{5}{4}}
\|\Gtp\|_{H_{\Psi}^{s+1,0}}.
\end{align*}
Combining the above estimates, we can deduce
\begin{align*}
\tad\|\py^2\Pp\|_{H_{\Psi}^{s,0}}
\leq&C\ep\ltr^{-\ga_0+\frac{1}{4}}
\|\py^3\Gtp\|_{H_{\Psi}^{s,0}}
+C\ep\ltr^{-\ga_0-\frac{1}{4}}
\|\py^2\Gtp\|_{H_{\Psi}^{s+1,0}}\\
&+C\ep\ltr^{-\ga_0-\frac{3}{4}}
\|\py\Gtp\|_{H_{\Psi}^{s+1,0}}
+C\ep\ltr^{-\ga_0-\frac{5}{4}}
\|\Gtp\|_{H_{\Psi}^{s+1,0}}.
\end{align*}
Finally using \eqref{huisushele}, it's similar to yield
\begin{align*}
\|\py^2(\frac{y}{2\ltr}\int_y^{\ty}\tad\Pp dz)\|_{H_{\Psi}^{s,0}}
\leq&C\ep\ltr^{-\ga_0+\frac{1}{4}}
\|\py^3\Gtp\|_{H_{\Psi}^{s,0}}
+C\ep\ltr^{-\ga_0-\frac{1}{4}}
\|\py^2\Gtp\|_{H_{\Psi}^{s+1,0}}\\
&+C\ep\ltr^{-\ga_0-\frac{3}{4}}
\|\py\Gtp\|_{H_{\Psi}^{s+1,0}}
+C\ep\ltr^{-\ga_0-\frac{5}{4}}
\|\Gtp\|_{H_{\Psi}^{s+1,0}}.
\end{align*}
Combining all the above estimates, we complete the proof of lemma \ref{FFdeguji}.
\end{proof}

With lemma \ref{FFdeguji}, we can now estimate $(\py^3\Gp,\py^3\Gtp)$.
\begin{Proposition}\label{py3Gdexianyanguji}
Let $\ka\in(0,2)$ be a given constant. Let $(G,\Gt)$ be defined by \eqref{gf}, and $a(t)$ be a non-negative and non-decreasing function on $\R_+$. Then when $\al\leq\ga_0+\frac{5}{36}$, the following inequality holds for any $t<T_*$ and sufficiently small $\eta>0$.
\begin{align}\label{py3Gdexianyanguji2}
\begin{aligned}
&\|\sqrt{a}(\py^3\Gp(t),\py^3\Gtp(t))\|^2_{H_{\Psi}^{2,0}}
-\int_0^t\|\sqrt{a'}(\py^3\Gp,\py^3\Gtp)\|^2_{H_{\Psi}^{2,0}}ds\\
&+2\int_0^t\lsr^{-1}\|\sqrt{a}(\py^3\Gp,\py^3\Gtp)\|^2_{H_{\Psi}^{2,0}}ds
+(4l_{\ka}-\frac{1}{25}\eta)\int_0^t\|\sqrt{a}(\py^4\Gp,\py^4\Gtp)\|^2_{H_{\Psi}^{2,0}}ds\\
&+2(\la-C\eta^{-1}\ep^{\frac{3}{2}})
\int_0^t\tad\|\sqrt{a}(\py^3\Gp,\py^3\Gtp)\|^2_{H_{\Psi}^{\frac{7}{3},0}}ds\\
\leq&\|\sqrt{a}(0)(\py^3\Gp(0),\py^3\Gtp(0))\|^2_{H_{\Psi}^{2,0}}
+C\eta^{-1}\ep^{\frac{3}{2}}\int_0^t\tad
\|\lsr^{-\frac{3}{2}}\sqrt{a}(\Gp,\Gtp)
\|^2_{H_{\Psi}^{\frac{13}{3},0}}ds\\
&+C\eta^{-1}\ep^{\frac{3}{2}}\int_0^t\tad
\|\lsr^{-1}\sqrt{a}(\py\Gp,\py\Gtp)
\|^2_{H_{\Psi}^{\frac{10}{3},0}}ds\\
&+C\eta^{-1}\ep^{\frac{3}{2}}\int_0^t\tad
\|\lsr^{-\frac{1}{2}}\sqrt{a}(\py^2\Gp,\py^2\Gtp)
\|^2_{H_{\Psi}^{\frac{13}{3},0}}ds,
\end{aligned}
\end{align}
where $l_{\ka}=\frac{\ka(2-\ka)}{4}\in(0,\frac{1}{4}].$
\end{Proposition}
\begin{proof}
Taking $H_{\Psi}^{2,0}$ inner product of \eqref{pypypygfephi} with $a(t)\py^3\Gp$, noticing the boundary condition \eqref{zuihoudebianjietiaojian}, using \eqref{ptpy} and integrating by parts we can get that
\begin{align}\label{zaijianchisixia}
\begin{aligned}
\frac{1}{2}&\frac{d}{dt}\|\sqrt{a}\py^3\Gp\|^2_{H_{\Psi}^{2,0}}
-\frac{1}{2}\|\sqrt{a'}\py^3\Gp\|^2_{H_{\Psi}^{2,0}}
+\ltr^{-1}\|\sqrt{a}\py^3\Gp\|^2_{H_{\Psi}^{2,0}}\\
&+\la\tad(t)\|\sqrt{a}\py^3\Gp\|^2_{H_{\Psi}^{\frac{7}{3},0}}
+\frac{1}{2}\|\sqrt{a}\py^4\Gp\|^2_{H_{\Psi}^{2,0}}\\
\leq&a\left|(\mathfrak{F}-\py^2Z,\py^4\Gp)_{H_{\Psi}^{2,0}}\right|
+a\left|(\mathfrak{F}-\py^2Z,
\frac{y}{2\ltr}\py^3\Gp)_{H_{\Psi}^{2,0}}\right|.
\end{aligned}
\end{align}
From lemma \ref{pypyZdeguji} and lemma \ref{FFdeguji} we can find
\begin{align*}
&a\left|(\mathfrak{F}-\py^2Z,\py^4\Gp)_{H_{\Psi}^{2,0}}\right|\\
\leq&C\ep\ltr^{-\ga_0+\frac{1}{4}}
\|\sqrt{a}(\py^3\Gp,\py^3\Gtp)\|_{H_{\Psi}^{2,0}}
\|\sqrt{a}\py^4\Gp\|_{H_{\Psi}^{2,0}}\\
&+C\ep\ltr^{-\ga_0-\frac{1}{4}}
\|\sqrt{a}(\py^2\Gp,\py^2\Gtp)\|_{H_{\Psi}^{3,0}}
\|\sqrt{a}\py^4\Gp\|_{H_{\Psi}^{2,0}}\\
&+C\ep\ltr^{-\ga_0-\frac{3}{4}}
\|\sqrt{a}(\py\Gp,\py\Gtp)\|_{H_{\Psi}^{3,0}}
\|\sqrt{a}\py^4\Gp\|_{H_{\Psi}^{2,0}}\\
&+C\ep\ltr^{-\ga_0-\frac{5}{4}}
\|\sqrt{a}(\Gp,\Gtp)\|_{H_{\Psi}^{3,0}}
\|\sqrt{a}\py^4\Gp\|_{H_{\Psi}^{2,0}}\\
\leq&C\eta^{-1}\ep^{\frac{3}{2}}\tad
\|\ltr^{-\frac{3}{2}}\sqrt{a}(\Gp,\Gtp)\|^2_{H_{\Psi}^{\frac{13}{3},0}}
+C\eta^{-1}\ep^{\frac{3}{2}}\tad
\|\ltr^{-1}\sqrt{a}(\py\Gp,\py\Gtp)\|^2_{H_{\Psi}^{\frac{10}{3},0}}\\
&+C\eta^{-1}\ep^{\frac{3}{2}}\tad
\|\ltr^{-\frac{1}{2}}\sqrt{a}(\py^2\Gp,\py^2\Gtp)\|^2_{H_{\Psi}^{\frac{10}{3},0}}
+C\eta^{-1}\ep^{\frac{3}{2}}\tad
\|\sqrt{a}(\py^3\Gp,\py^3\Gtp)\|^2_{H_{\Psi}^{\frac{7}{3},0}}\\
&+\frac{1}{100}\eta\|\sqrt{a}\py^4\Gp\|^2_{H_{\Psi}^{2,0}},
\end{align*}
here we use the fact that $\al\leq\ga_0+\frac{5}{36}$ and $\ga_0>1.$

Next by \eqref{huisushele} and the above estimates, we can deduce
\begin{align*}
&a\left|(\mathfrak{F}-\py^2Z,
\frac{y}{2\ltr}\py^3\Gp)_{H_{\Psi}^{2,0}}\right|
\\
\leq&C\eta^{-1}\ep^{\frac{3}{2}}\tad
\|\ltr^{-\frac{3}{2}}\sqrt{a}(\Gp,\Gtp)\|^2_{H_{\Psi}^{\frac{13}{3},0}}
+C\eta^{-1}\ep^{\frac{3}{2}}\tad
\|\ltr^{-1}\sqrt{a}(\py\Gp,\py\Gtp)\|^2_{H_{\Psi}^{\frac{10}{3},0}}\\
&+C\eta^{-1}\ep^{\frac{3}{2}}\tad
\|\ltr^{-\frac{1}{2}}\sqrt{a}(\py^2\Gp,\py^2\Gtp)\|^2_{H_{\Psi}^{\frac{10}{3},0}}
+C\eta^{-1}\ep^{\frac{3}{2}}\tad
\|\sqrt{a}(\py^3\Gp,\py^3\Gtp)\|^2_{H_{\Psi}^{\frac{7}{3},0}}\\
&+\frac{1}{100}\eta\|\sqrt{a}\py^4\Gp\|^2_{H_{\Psi}^{2,0}}.
\end{align*}
Taking all the above estimates into \eqref{zaijianchisixia}, and integrating over $[0,t]$, we can know that $\py^3\Gp$ satisfies
\begin{align*}
&\|\sqrt{a}\py^3\Gp(t)\|^2_{H_{\Psi}^{2,0}}
-\int_0^t\|\sqrt{a'}\py^3\Gp\|^2_{H_{\Psi}^{2,0}}ds
+2\int_0^t\lsr^{-1}\|\sqrt{a}\py^3\Gp\|^2_{H_{\Psi}^{2,0}}ds\\
&+2\la\int_0^t\tad\|\sqrt{a}\py^3\Gp\|^2_{H_{\Psi}^{\frac{7}{3},0}}ds
+\int_0^t\|\sqrt{a}\py^4\Gp\|^2_{H_{\Psi}^{2,0}}ds\\
\leq&\|\sqrt{a}(0)\py^3\Gp(0)\|^2_{H_{\Psi}^{2,0}}
+C\eta^{-1}\ep^{\frac{3}{2}}
\int_0^t\tad
\|\sqrt{a}(\py^3\Gp,\py^3\Gtp)\|^2_{H_{\Psi}^{\frac{7}{3},0}}ds
+\frac{1}{25}\eta\int_0^t\|\sqrt{a}\py^4\Gp\|^2_{H_{\Psi}^{2,0}}ds\\
&+C\eta^{-1}\ep^{\frac{3}{2}}\int_0^t\tad
\|\lsr^{-\frac{3}{2}}\sqrt{a}(\Gp,\Gtp)
\|^2_{H_{\Psi}^{\frac{13}{3},0}}ds
+C\eta^{-1}\ep^{\frac{3}{2}}\int_0^t\tad
\|\lsr^{-1}\sqrt{a}(\py\Gp,\py\Gtp)
\|^2_{H_{\Psi}^{\frac{10}{3},0}}ds\\
&+C\eta^{-1}\ep^{\frac{3}{2}}\int_0^t\tad
\|\lsr^{-\frac{1}{2}}\sqrt{a}(\py^2\Gp,\py^2\Gtp)
\|^2_{H_{\Psi}^{\frac{13}{3},0}}ds.
\end{align*}
This completes the estimates of $\py^3 \Gp$. And the estimates of $\py^3\Gtp$ can be obtained in the same way, so we have completed the proof of proposition \ref{py3Gdexianyanguji}.
\end{proof}

Finally we can use proposition \ref{py3Gdexianyanguji} to prove proposition \ref{propositionpy3G}.

\textbf{Proof of Proposition \ref{propositionpy3G}:}
Taking $a(t)=\ltr^{5+2l_{\ka}-2\eta}$ in \eqref{py3Gdexianyanguji2}, then we have
\begin{align*}
&\|\ltr^{\frac{5}{2}+l_{\ka}-\eta}
(\py^3\Gp(t),\py^3\Gtp(t))\|^2_{H_{\Psi}^{2,0}}
+(4l_{\ka}-\frac{1}{25}\eta)\int_0^t\|\lsr^{\frac{5}{2}+l_{\ka}-\eta}
(\py^4\Gp,\py^4\Gtp)\|^2_{H_{\Psi}^{2,0}}ds\\
&+2(\la-C\eta^{-1}\ep^{\frac{3}{2}})
\int_0^t\tad\|\lsr^{\frac{5}{2}+l_{\ka}-\eta}
(\py^3\Gp,\py^3\Gtp)\|^2_{H_{\Psi}^{\frac{7}{3},0}}ds\\
\leq&\|(\py^3\Gp(0),\py^3\Gtp(0))\|^2_{H_{\Psi}^{2,0}}
+C\eta^{-1}\ep^{\frac{3}{2}}\int_0^t\tad
\|\lsr^{1+l_{\ka}-\eta}(\Gp,\Gtp)
\|^2_{H_{\Psi}^{\frac{13}{3},0}}ds\\
&+C\eta^{-1}\ep^{\frac{3}{2}}\int_0^t\tad
\|\lsr^{\frac{3}{2}+l_{\ka}-\eta}(\py\Gp,\py\Gtp)
\|^2_{H_{\Psi}^{\frac{10}{3},0}}ds\\
&+C\eta^{-1}\ep^{\frac{3}{2}}\int_0^t\tad
\|\lsr^{2+l_{\ka}-\eta}(\py^2\Gp,\py^2\Gtp)
\|^2_{H_{\Psi}^{\frac{13}{3},0}}ds\\
&+(3+2l_{\ka}-2\eta)\int_0^t\|\lsr^{2+l_{\ka}-\eta}
(\py^3\Gp,\py^3\Gtp)\|^2_{H_{\Psi}^{2,0}}ds.
\end{align*}
By applying \eqref{propositionG1}, \eqref{propositionpyG1} and
\eqref{propositionpy2G1} we deduce
\begin{align*}
&\|\ltr^{\frac{5}{2}+l_{\ka}-\eta}
(\py^3\Gp(t),\py^3\Gtp(t))\|^2_{H_{\Psi}^{2,0}}
+(4l_{\ka}-\frac{1}{25}\eta)\int_0^t\|\lsr^{\frac{5}{2}+l_{\ka}-\eta}
(\py^4\Gp,\py^4\Gtp)\|^2_{H_{\Psi}^{2,0}}ds\\
&+2(\la-C\eta^{-1}\ep^{\frac{3}{2}})
\int_0^t\tad\|\lsr^{\frac{5}{2}+l_{\ka}-\eta}
(\py^3\Gp,\py^3\Gtp)\|^2_{H_{\Psi}^{\frac{7}{3},0}}ds\\
\leq&
\dfrac{C}{\eta(4l_{\ka}-\frac{7}{25}\eta)(4l_{\ka}-\frac{13}{100}\eta)}
\left(\|(\py\Gp(0),\py\Gtp(0))\|^2_{H_{\Psi}^{3,0}}
+\|(\Gp(0),\Gtp(0))\|^2_{H_{\Psi}^{4,0}}\right.\\
&\left.+\|(\py^2\Gp(0),\py^2\Gtp(0))\|^2_{H_{\Psi}^{3,0}}
+\|(\py^3\Gp(0),\py^3\Gtp(0))\|^2_{H_{\Psi}^{2,0}}\right)\\
&+\dfrac{C}{\eta(4l_{\ka}-\frac{7}{25}\eta)(4l_{\ka}-\frac{13}{100}\eta)}
\int_0^t\tad
\|\lsr^{l_{\ka}-\eta}\sqrt{a}(\uph,\bp)\|^{2}_{H_{\Psi}^{\frac{22}{3},0}}
ds.
\end{align*}
Thus when $\ep<\ep_3$ and $\la>\la_2$, proposition \ref{propositionpy3G} holds.

\hfill $\square$

\subsection*{Acknowledgements}
This work was supported by National Natural Science Foundation of China, NSFC (NO: 11926316, 11531010, 12071391) and Guangdong Basic and Applied Basic Research Foundation (2022A1515010860).
\subsection*{Competing interests}

This work does not have any conflicts of interest.

\end{document}